\documentclass{article}
\usepackage[left=3cm,right=3cm,top=2.6cm,bottom=3.6cm]{geometry} 
\usepackage{graphicx} 
\usepackage{amsmath}
\usepackage{amsfonts}
\usepackage{amsthm}
\usepackage{amssymb}
\usepackage{mathrsfs}
\usepackage[all,cmtip,2cell]{xy}\UseAllTwocells
\usepackage[hidelinks]{hyperref}
\usepackage{cleveref}
\usepackage{xcolor}
\usepackage{comment}
\usepackage{verbatim}

\usepackage[backend=biber,style=alphabetic,maxalphanames=99, maxnames=99]{biblatex}
\title{Relative and lax volutive categories}
\author{Tim Lüders}
\date{}

\newcommand{\yon}{y}%yoneda embedding
\newcommand{\Prof}{\operatorname{Prof}}%Profunctor
\newcommand{\opp}{\operatorname{op}}%opposite
\newcommand{\id}{\operatorname{id}}%identity
\newcommand{\ev}{\operatorname{ev}}%evaluation
\newcommand{\coev}{\operatorname{coev}}%coevaluation
\newcommand{\Cat}{\operatorname{Cat}}%Categories
\newcommand{\Aut}{\operatorname{Aut}}%Automorphism
\newcommand{\Set}{\operatorname{Set}}%Sets
\newcommand{\Ban}{\operatorname{Ban}}%Banach spaces
\newcommand{\Fun}{\operatorname{Fun}}%Functors
\newcommand{\Hom}{\operatorname{Hom}}%Hom
\newcommand{\Hilb}{\operatorname{Hilb}}%Hilbert spaces
\newcommand{\Cha}{\operatorname{Ch}}%Chain complex
\newcommand{\Rmod}{R\text{-}\hspace{-0.1cm}\operatorname{mod}}%R-modules
\newcommand{\Mor}{\operatorname{Mor}}%Morphisms also Morita 2-category
\newcommand{\Born}{\mathfrak{B}\hspace{-0.05cm}\operatorname{orn}}%bornological space
\newcommand{\sBorn}{s\mathfrak{B}\hspace{-0.05cm}\operatorname{orn}}%separated bornological space
\newcommand{\cBorn}{c\mathfrak{B}\hspace{-0.05cm}\operatorname{orn}}%complete bornological space 
\newcommand{\Ev}{\operatorname{Ev}}%evaluation
\newcommand{\Nor}{\operatorname{Nor}}%normed space
\newcommand{\sNor}{\operatorname{sNor}}%semi normed space
\newcommand{\vN}{\operatorname{vN}}%von Neumann
\newcommand{\Alg}{\operatorname{Alg}}%algebras
\newcommand{\bi}{\operatorname{bi}}%bi-involutive
\newcommand{\Adj}{\operatorname{Adj}}%Adjunction
\newcommand{\AdjEqv}{\operatorname{AdjEqv}}%ADjoint equivalence
\newcommand{\Mod}{\operatorname{Mod}}%Modules
\newcommand{\Grpd}{\operatorname{Grpd}}%groupoid
\newcommand{\unb}{\operatorname{unb}}%unbounded
\newcommand{\HilbRel}{\operatorname{HilbRel}}%Hilbert spaces and relations
\newcommand{\rev}{\operatorname{rev}}%reverse
\newcommand{\refl}{\operatorname{refl}}%reflexive module
\newcommand{\Ring}{\operatorname{Ring}}%The category of rings
\newcommand{\coeq}{\operatorname{coeq}}%coequalizer
\newcommand{\equa}{\operatorname{eq}}%equalizer
\newcommand{\AmAdj}{\operatorname{AmbAdj}}%Ambidextrous adjunction
\newcommand{\Ch}{\operatorname{Ch}}%Chain complexes
\newcommand{\LaxHerm}{\operatorname{LaxHerm}}%LaxHerm
\newcommand{\hh}{\operatorname{h}\hspace{-0.05cm}}%homotopy 
\newcommand{\Oone}{\operatorname{O(1)}}
\newcommand{\Sotwo}{\operatorname{SO(2)}}

\newcommand{\Onn}{\operatorname{O(n)}}
\newcommand{\B}{\mathcal{B}}
\newcommand{\C}{\mathcal{C}}
\newcommand{\D}{\mathcal{D}}
\newcommand{\E}{\mathcal{E}}
\newcommand{\G}{\mathcal{G}}

\newcommand{\Z}{\mathbb{Z}}

\numberwithin{equation}{section}

\makeatletter
\newcommand{\xRrightarrow}[2][]{\ext@arrow 0359\Rrightarrowfill@{#1}{#2}}
\newcommand{\Rrightarrowfill@}{\arrowfill@\equiv\equiv\Rrightarrow}
\newcommand{\xLleftarrow}[2][]{\ext@arrow 3095\Lleftarrowfill@{#1}{#2}}
\newcommand{\Lleftarrowfill@}{\arrowfill@\Lleftarrow\equiv\equiv}
\makeatother
\newcommand{\RRightarrow}{\xRrightarrow{\phantom{--}}} %A long right arrow with triple lines

\theoremstyle{definition}               
\newtheorem{definition}{Definition}[subsection] 
\newtheorem{construction}[definition]{Construction} 
\newtheorem{example}[definition]{Example} 
\newtheorem{lemma}[definition]{Lemma} 
\newtheorem{proposition}[definition]{Proposition} 
\newtheorem{remark}[definition]{Remark} 
 
\newtheorem{variant}[definition]{Variant} 
\newtheorem{corollary}[definition]{Corollary} 
\newtheorem{warning}[definition]{Warning} 
\newtheorem{notation}[definition]{Notation} 
\newtheorem{Iexample}{Example}
\newtheorem{Idefinition}{Definition}
\theoremstyle{theorem}               
\newtheorem{theorem}[definition]{Theorem} 
\newtheorem{Itheorem}{Theorem}

\begin{document}
\maketitle
\begin{abstract}
    In this paper we introduce the notion of a relative volutive (higher) category, specializing to the notion of 
    a lax volutive (higher) category. Our primary motivation to study these objects is the following: while any 
    rigid symmetric monoidal category admits a volutive structure, any closed symmetric monoidal category admits 
    a lax volutive structure. We develop some of the basic theory of relative volutive categories and provide 
    several equivalent formulations of lax volutive categories. We then study examples of interest, including 
    categories of complete bornological vector spaces and modules over star-rings. We will also separately 
    discuss unbounded operators between Hilbert spaces and Morita 2-categories, the latter of which in the context of 
    fully closed symmetric monoidal 2-categories.
\end{abstract}

\tableofcontents

\section{Introduction}
Duality is a powerful tool in the construction of new (higher) volutive and dagger categories. A prominent conjecture \cite{ferrer2024daggerncategories} (proven for 
$n \leq 2$) asserts that any symmetric monoidal $n$-category $\C$ with all duals and adjoints admits an $\Onn$-volutive structure, which reduces on the maximal 
subgroupoid $\C^\times$ to the canonical $\Onn$-action predicted by the cobordism hypothesis. On the other hand, many symmetric monoidal categories that appear 
in practice only admit weaker notions of duality, such as the categories of chain complexes or bornological vector spaces. In this paper, we provide a first glance 
at the appropriate generalizations of (higher) volutive structures that account for such weaker kinds of duality. In dimension one, we prove the following result, 
see \Cref{Theorem: closed symmetric monoidal categories are lax volutive} in the main text.
\begin{Itheorem}\label{Theorem: csm categories are lax volutive}
    Any closed symmetric monoidal category $\C$ admits a lax $\Oone$-volutive structure.
\end{Itheorem}
We summarize several good classes of examples (not necessarily constructed via \Cref{Theorem: csm categories are lax volutive}) of lax $\Oone$-volutive 
categories in the following. 
\begin{Iexample}
    The following categories admit lax $\Oone$-volutive structures:
    \begin{itemize}
        \item[(i)] the category of (unbounded) chain complexes over a commutative ring, 
        \item[(ii)] the category of complete (convex) bornological vector spaces, 
        \item[(iii)] the full subcategory of Banach spaces,
        \item[(iv)] the category of modules over a $*$-ring (e.g. any commutative ring), and
        \item[(v)] any cartesian closed category.
    \end{itemize}
\end{Iexample}
By definition, $\Oone$-volutive categories are $\Oone$-homotopy fixed points with respect to the canonical non-trivial $\Oone$-action on the (2,1)-category 
of categories, assigning each category $\C$ to its opposite $\C^{\opp}$. Since it is not true that lax $\Oone$-volutive categories are lax $\Oone$-homotopy 
fixed points with respect to the aforementioned action, a different approach to formalizing these structures is required. Our second main result gives several 
equivalent characterizations of lax $\Oone$-volutive categories, see \Cref{Theorem: four equivalent versions of lax O(1)-volutive categories} in the main text.
\begin{Itheorem}\label{Theorem: equivalent formulations of lax Oone volutive categories}
    The following notions are equivalent: 
    \begin{itemize}
        \item[(i)] Lax $\Oone$-volutive categories in the sense of \Cref{Definition: lax O(1)-volutive category},
        \item[(ii)] representable symmetric pairings in the sense of \Cref{Definition: properties of pairings},
        \item[(iii)] volutive adjunctions in $\Cat$ in the sense of \Cref{Definition: volutive adjunction}, and
        \item[(iv)] (2-contravariant) $\Oone$-volutive functors $\Adj \to \Cat$.
    \end{itemize}
\end{Itheorem}
The last characterization in \Cref{Theorem: equivalent formulations of lax Oone volutive categories} is our primary motivation to introduce the theory of 
(higher) relative volutive categories, generalizing (higher) lax volutive categories. Built on the observation that both $\Adj$ (the walking adjunction) and 
$\Cat$ (the 2-category of categories) carry (2-contravariant) $\Oone$-volutive structures, the $\Oone$-case presents as follows, see 
\Cref{Definition: relative Oone volutive category} in the main text.
\begin{Idefinition}\label{Definition: relative Oone volutive categories}
    A \emph{relative $\Oone$-volutive category} consists of a 2-contravariant $\Oone$-volutive 2-category $\B$ and a 
    2-contravariant $\Oone$-volutive functor $F \colon \B \to \Cat$. 
\end{Idefinition}
For $\B = \star$ and $\B=\Adj$, \Cref{Definition: relative Oone volutive categories} reduces to the notions of $\Oone$-volutive and lax $\Oone$-volutive categories, 
respectively; it may be instructive for the reader to think of the final 2-category $\star$ as equivalent to the walking adjoint equivalence, so that a lax $\Oone$-volutive 
category may be understood as an $\Oone$-volutive category with lesser degrees of invertibility, which is consistent with the explicit descriptions we provide. \\

As the primary instance of higher relative and lax volutive categories, we introduce relative and lax $\Sotwo$-volutive 2-categories. In our discussion we put 
a special emphasis on analogies with the relative and lax $\Oone$-volutive case, respectively. On one hand we will establish that any closed symmetric monoidal
2-category with adjoints admits a lax $\Sotwo$-volutive structure, given by the Serre morphism. On the other hand, our definition of a relative $\Sotwo$-volutive 
2-category as a (2-contravariant) $\Sotwo$-volutive 3-functor $\B \to 2\hspace{-0.05cm}\Cat^{\operatorname{adj}}$ generalizes the one of $\Sotwo$-volutive 
2-categories for $\B = \star$ and lax $\Sotwo$-volutive 2-categories for (a suitable localization of) $\B = \AmAdj[1]$, respectively, where $\AmAdj[1]$ refers to a categorification 
of the walking ambidextrous adjunction. Thus, our two main theorems (establishing that dualities give rise to lax volutive structures, and that lax volutive structures are special relative volutive 
structures) regarding relative and lax $\Oone$-volutive categories have corresponding analogues in the $\Sotwo$-case.\\

As an application of the theory of lax $\Oone$-volutive categories, we recontextualize the dagger structure on the category of Hilbert spaces from the perspective
of the following table, containing motivating examples of categories and the structures (we discuss here) they support.
\begin{center}
\begin{tabular}{c|c|c|c} 
 Dagger & $\Oone$-volutive & lax $\Oone$-volutive & closed symmetric monoidal \\ 
  \hline
 $\Hilb$ & $\Ban^{\refl}$ & $\Ban$ & $\cBorn$ \\
\end{tabular}
\end{center}
From left to right, the second line refers to the categories of Hilbert spaces, (reflexive) Banach spaces, and complete (convex) bornological vector spaces (each 
with bounded linear maps as morphisms). \\

We also discuss lax volutive structures in the context of fully closed symmetric monoidal 2-categories. In particular, we will show that each fully closed 
symmetric monoidal 2-category admits a lax $\Oone$-volutive structure which upon passing to the associated dagger 2-category yields a category enriched in 
lax $\Oone$-volutive categories. This suffices to pass locally to the categories of lax hermitian fixed points and lax isometries.
We will illustrate these constructions in detail at the example of Morita 2-categories. We prove the following result, see 
\Cref{Theorem: closedness of the Morita 2-category} and \Cref{Corollary: Morita 2-categories of nice categories} in the main text. 
\begin{Itheorem}\label{Theorem: Morita 2-categories are closed}
    Let $\C$ be a closed symmetric monoidal category with (co)equalizers. Then, $\Mor(\C)$ is a symmetric monoidal 1-closed 2-category with duals for objects.
\end{Itheorem}
As a particular instance of \Cref{Theorem: Morita 2-categories are closed}, the 2-category $\Mor(\cBorn)$ of complete (convex) bornological algebras is a fully 
closed symmetric monoidal 2-categories. We will discuss lax hermitian fixed points in the context of this 2-category, and compare the results to more familiar 
structures such as von Neumann algebras and hermitian/Hilbert modules.\\

\textbf{Acknowledgements.} We thank Lukas Müller, Aleksandar Ivanov, and Ödül Tetik for many helpful discussions. We acknowledge financial support from the 
Studienstiftung des deutschen Volkes e.V.\\

\textbf{Conventions.} In this article, we will generically refer to bicategories as 2-categories, to pseudofunctors as 2-functors, etc., and similarly for 3-categories. 
Whenever we refer to a 1-morphism in a 2-category as ``invertible'' or a ``1-isomorphism'', we mean that it is an equivalence in the traditional sense. Given a 
2-morphism $f \colon X \to Y$ between invertible 1-morphisms, we will generically not distinguish in notation between $f$ and the induced 2-morphism $Y^{-1} \to X^{-1}$.
We will call a (symmetric) monoidal (higher) category with all duals and adjoints \emph{rigid}.

\section{Preliminaries}\label{Section: Preliminaries}
In this section, we review essential parts of the theory of ordinary dagger category theory; the results presented here are essentially well-known to experts and 
we do not claim originality for the covered material. We start by reviewing $\Oone$-volutive and dagger categories, then once categorify the theory, 
establish the respective 2-categories of $\Oone$-volutive and dagger categories, and then discuss the adjunction between them. We highlight the microcosmic 
approach that a once-categorified theory of $\Oone$-volutive and dagger categories supplies us with. 
\subsection{$\Oone$-volutive and dagger categories}
In this section, we review the notions of $\Oone$-volutive and dagger categories. We start with the former.
\begin{definition}\label{Definition: O(1)-volutive category}
    Let $\C$ be a category. An \emph{$\Oone$-volutive structure} on $\C$ consists of a functor $d \colon \C \to \C^{\opp}$ and a natural isomorphism 
    $\eta \colon \id \to d^{\opp} \circ d$ satisfying $d(\eta_a) = \eta_{d(a)}^{-1}$ for all $a \in \C$. 
    A category together with an $\Oone$-volutive structure is called an \emph{$\Oone$-volutive category}.
\end{definition}
\begin{example}\label{Example: Banach spaces are Oone volutive}
    The category $\Ban^{\refl}$ of reflexive Banach spaces admits an $\Oone$-volutive structure. The functor $d$ assigns each Banach space $V$ to its 
    dual space $V'$ and the natural isomorphism $\eta$ has as its component at $V$ the canonical isomorphism $V \to V''$ from $V$ into its bidual. 
\end{example}
A good class of examples of $\Oone$-volutive categories comes from the following; we will prove a more general result later in this paper.
\begin{proposition}\label{Proposition: rigid symmetric monoidal categories are Oone volutive}
    Let $\C$ be a symmetric monoidal category with duals. Then, $\C$ admits an $\Oone$-volutive structure. 
\end{proposition}
\begin{remark}\label{Remark: Oone volutive categories as homotopy fixed points}
    An $\Oone$-volutive category is the same as an $\Oone$-homotopy fixed point with respect to the canonical non-trivial $\Oone$-action on the (2,1)-category 
    $\Cat_{(2,1)}$ of categories, functors, and natural isomorphisms. Explicitly, this $\Oone$-action is given by the strictly involutive 2-functor assigning 
    each category $\C$ to its opposite, $\C^{\opp}$. 
\end{remark}
\begin{remark}\label{Remark: Oone volutive structures and groupoids}
    Any groupoid $\G$ admits an $\Oone$-volutive structure with underlying functor being inversion of morphisms. More generally, any groupoid with $\Oone$-action 
    defines an $\Oone$-volutive category by suitably combining the $\Oone$-action and inversion. Conversely, any $\Oone$-volutive structure $(d,\eta)$ on a 
    category $\C$ induces an $\Oone$-action on the maximal subgroupoid $\C^\times \subseteq \C$. 
\end{remark}
\begin{definition}\label{Definition: coherent dagger category}
    A \emph{(coherent) dagger category} is an $\Oone$-volutive category $(\C,d,\eta)$ together with a fully faithful subgroupoid
    $\C_0 \hookrightarrow (\C^\times)^{\Oone}$ such that the induced functor $\C_0 \hookrightarrow (\C^\times)^{\Oone} \to \C^\times$ is essentially surjective. 
\end{definition}
\begin{variant}
    A \emph{(classical) dagger category} is a category $\C$ together with a functor $d \colon C \to \C^{\opp}$ satisfying $d^{\opp} \circ d=\id_{\C}$ and 
    $d(a) = a$ for all objects $a \in \C$.
\end{variant}
\begin{example}
    The category of Hilbert spaces $\Hilb$ admits a dagger structure given by the adjoint operator construction. 
\end{example}
\begin{remark}
    Coherent and classical dagger categories encode essentially the same information, with the main advantage of the former being that it generalizes arguably 
    more cleanly to higher dimensions. 
    To see the former claim, let $(\C,d)$ be a classical dagger category. We obtain an $\Oone$-volutive category $(\C,d,\id)$ and we may choose the fully faitful 
    subgroupoid $(\C^\times)^{\Oone} \subseteq (\C^\times)^{\Oone}$, for which the induced functor in the definition of a coherent dagger category is essentially 
    surjective since any object $a$ in $\C$ admits the structure of an $\Oone$-homotopy fixed point with respect to the induced $\Oone$-action by $d$: namely, 
    $(a,\id_a) \in (\C^\times)^{\Oone}$. 
    Conversely, let $(\C,d,\eta,\C_0)$ be a coherent dagger category. Define a category $\tilde{C}$ whose objects are the objects of $\C_0$ and 
    whose morphisms $(a,\theta_a) \to (b,\theta_b)$ are the morphisms $a \to b$ in $\C$. The category $\tilde{C}$ then carries a dagger structure given by sending 
    a morphism $X \colon (a,\theta_a) \to (b,\theta_b)$ in $\tilde{\C}$ to the morphism $(b,\theta_b) \to (a,\theta_a)$ defined by
    \begin{equation}
        \xymatrix{
            b \ar[r]^-{\theta_b} & d(b) \ar[r]^-{d(X)} & d(a) \ar[r]^-{\theta_a^{-1}} & a.
        }
    \end{equation}
    The second claim is part of the ongoing program of defining higher dagger categories. In this paper, we will freely choose in our usage of coherent and classical 
    dagger categories depending on the problem at hand. 
\end{remark}
\begin{remark}\label{Remark: Oone volutive categories give dagger categories}
    Any $\Oone$-volutive category $(\C,d,\eta)$ gives rise to a dagger category, as discussed in \cite{Stehouwer2023DaggerCV}. We will describe a closely 
    related version of this construction in \Cref{Remark: hermitian fixed points of a lax Oone volutive category}.
\end{remark}
\begin{example}
    Applying \Cref{Remark: Oone volutive categories give dagger categories} to the $\Oone$-volutive category $\Ban^{\refl}$ of reflexive Banach spaces 
    yields the dagger category of hermitian Banach spaces, which contains the dagger category of Hilbert spaces as a full subcategory. 
\end{example}
\begin{definition}
    Let $(\C,d)$ be a (classical) dagger category. A morphism $X \colon a \to b$ in $\C$ is said to be an \emph{isometry} if $d(X) \circ X = \id_a$ and 
    said to be \emph{unitary} if $d(X) = X^{-1}$.  
\end{definition}

\subsection{$\Oone$-volutive 2-categories and dagger categories}
In this section we give specific once-categorifed versions of the notions of $\Oone$-volutive and dagger categories. We start with the former.
\begin{definition}\label{Definition: 2-contravariant Oone volutive 2-categories}
    A \emph{2-contravariant $\Oone$-volutive structure} on a 2-category $\B$ consists of a 2-functor $d \colon \B \to \B^{2\opp}$, an invertible 2-transformation 
    $\eta \colon \id \to d^{2\opp} \circ d$, and an invertible modification $\tau \colon \id_{d^{\opp}} \to (\id_{d^{\opp}}\circ \eta^{2\opp}) \bullet (\eta \circ \id_{d^{\opp}})$
    satisying 
    \begin{equation}
        \begin{pmatrix}
            \xymatrix{
                & d^{\opp} \circ d \ar[d]^-{\id} \ar[dr]^-{\id \circ \eta} & \\
                \id \ar[r]^-{\eta} \ar[ur]^-{\eta} \ar[dr]_-{\eta} & d^{\opp} \circ d \ar[d]^-{\id} & d^{\opp} \circ d \circ d^{\opp} \circ d \ar[l]^-{\id \circ \eta^{\opp} \circ \id} \\
                & d^{\opp} \circ d \ar[ur]_-{\eta \circ \id} 
            }
        \end{pmatrix}
        = 
        \begin{pmatrix}
            \xymatrix{
                & d^{\opp} \circ d \ar[d]^-{\id \circ \eta}  \\
                \id \ar[ur]^-{\eta} \ar[dr]_-{\eta} & d^{\opp} \circ d \circ d^{\opp} \circ d \\
                & d^{\opp} \circ d \ar[u]_-{\eta \circ \id} 
        }
        \end{pmatrix}
    \end{equation}
    where in the left diagram we have used the modification $\tau$ and its 2-opposite, while in the right diagram we have only used the interchange law.
    A 2-category together with a 2-contravariant $\Oone$-volutive structure is called a \emph{2-contravariant $\Oone$-volutive 2-category}.
\end{definition}
Our primary example of a 2-contravariant $\Oone$-volutive 2-category in this paper is the following.
\begin{example}\label{Example: The Oone volutive structure on Cat}
    The 2-category $(\Cat,(-)^{\opp},\id,\id)$ of categories, functors, and natural transformations together with the 2-functor assigning each category $\C$ to 
    its opposite $\C^{\opp}$ is a 2-contravariant $\Oone$-volutive category.
\end{example}
\begin{remark}
    The 2-contravariant $\Oone$-volutive structure on $\Cat$ described in \Cref{Example: The Oone volutive structure on Cat} can be interpreted from the 
    perspective of the fully closed symmetric monoidal 2-category $\Prof$, into which $\Cat$ embedds; we will adress this point in \Cref{Subsection:Profunctors}.
\end{remark}
\begin{remark}\label{Remark: variances of Oone volutive 2-categories}
    2-contravariant $\Oone$-volutive 2-categories are $\Oone$-homotopy fixed points for the $\Oone$-action on $2\Cat_{(3,1)}$ given by the 3-functor $(-)^{2\opp}$ 
    assigning each 2-category to its 2-opposite. While this motivates our terminology, it also suggests two variants of the notion of an $\Oone$-volutive 
    2-category: 1-contravariant ones (based on the 3-functor $(-)^{1\opp}$) and (1,2)-contravariant ones (based on the 3-functor $(-)^{(1,2)\opp}$). Spelled out,
    these amount to essentially the same data as presented in \Cref{Definition: 2-contravariant Oone volutive 2-categories} with variances replaced appropriately. 
    The 1-contravariant version has been presented in \cite[Definition 2.40]{carqueville2025orbifoldshigherdaggerstructures}.
\end{remark}
\begin{remark}\label{Remark: 21categories and 2-contravariant Oone volutive structures}
    The underlying (2,1)-category of a 2-contravariant $\Oone$-volutive 2-category inherits an $\Oone$-action, analogous to the one-dimensional case discussed 
    in \Cref{Remark: Oone volutive structures and groupoids}. Similarly, any $\Oone$-action on a (2,1)-category $\D$ defines a 2-contravariant $\Oone$-volutive 
    structure. 
\end{remark}
\begin{definition}
    A \emph{(coherent) 2-contravariant $\Oone$-dagger 2-category} consists of a 2-contravariant $\Oone$-volutive 2-category $(\C,d,\eta,\tau)$ together with 
    a fully faithful subgroupoid $\C_0 \hookrightarrow (\C^\times)^{\Oone}$ such that the induced functor 
    $\C_0 \hookrightarrow (\C^\times)^{\Oone} \to \C^\times$ is essentially surjective. 
\end{definition}
\begin{definition}\label{Definition: 2-contravariant Oone dagger 2-category}
    A \emph{(classical) 2-contravariant $\Oone$-dagger 2-category} consists of an identity-on-objects 2-functor $d \colon \B \to \B^{2\opp}$ and an
    identity 1-morphism component invertible 2-transformation $\eta \colon \id \to d^{2\opp} \circ d$ satisying $\id_{d^{2\opp}} = (\id_{d^{2\opp}}\circ \eta^{2\opp}) \bullet (\eta \circ \id_{d^{2\opp}})$.
\end{definition}
\begin{remark}
    Similar to our discussion in \Cref{Remark: variances of Oone volutive 2-categories}, there are two other variances of $\Oone$-dagger 2-categories.
\end{remark}
\begin{remark}
    Any (classical) 2-contravariant $\Oone$-dagger 2-category $(\C,d,\eta)$ defines a 2-contravariant $\Oone$-volutive 2-category $(\C,d,\eta,\id)$.
\end{remark}
The following is essentially the 2-contravariant analogue (in the sense of \Cref{Remark: variances of Oone volutive 2-categories}) of the discussion in 
\cite[Section 2.3.3]{carqueville2025orbifoldshigherdaggerstructures}.
\begin{construction}\label{Construction: turning 2-contra Oone vol 2-cats dagger}
    Let $(\B,d,\eta,\tau)$ be a 2-contravariant $\Oone$-volutive 2-category. We obtain a (classical) 2-contravariant $\Oone$-dagger 2-category whose 
    underlying 2-category has as its objects those of $(\B^\times)^{\Oone}$ and as its 1- and 2-morphisms those of $\B$. 
\end{construction}
\begin{remark}\label{Remark: 2-contra dagger 2-cats and Oone volutive enrichments}
    Each Hom-category of a (classical) 2-contravariant $\Oone$-dagger 2-category $(\C,d,\eta)$ carries an $\Oone$-volutive structure. Moreover, the composition 
    functors are $\Oone$-volutive with respect to these $\Oone$-volutive structures. In other words, a 2-contravariant $\Oone$-dagger 2-category may naturally 
    be interpreted as a category enriched over $\Oone$-volutive categories. 
\end{remark}

\subsection{The 2-categories of $\Oone$-volutive and dagger categories}
In this section we briefly review the 2-categories of $\Oone$-volutive and dagger categories as well as the adjunction between them. 
\begin{remark}\label{Remark: The dagger structure of volutive categories}
    Applying \Cref{Construction: turning 2-contra Oone vol 2-cats dagger} to the 2-contravariant $\Oone$-volutive 2-category $(\Cat,(-)^{\opp},\id,\id)$ 
    described in \Cref{Example: The Oone volutive structure on Cat} yields a (classical) 2-contravariant $\Oone$-dagger 2-category whose objects are 
    $\Oone$-volutive categories, and whose 1- and 2-morphisms are functors and natural transformations, respectively.
\end{remark}
An explicit proof of the following is given in \cite[Lemma 3.9]{Stehouwer2023DaggerCV}.
\begin{lemma}
    Let $(\C,d,\eta)$ and $(\C',d',\eta')$ be $\Oone$-volutive categories. Then, $\Fun(\C,\C')$ admits an $\Oone$-volutive structure $(\tilde{d},\tilde{\eta})$.
\end{lemma}
\begin{proof}
    This follows immediately from \Cref{Remark: The dagger structure of volutive categories} and \Cref{Remark: 2-contra dagger 2-cats and Oone volutive enrichments}.
\end{proof}
\begin{definition}
    Let $(\C,d,\eta)$ and $(\C',d',\eta')$ be $\Oone$-volutive categories. Applying \Cref{Remark: Oone volutive categories give dagger categories} 
    to the $\Oone$-volutive category $(\Fun(\C,\C'),\tilde{d},\tilde{\eta})$ yields a dagger category whose objects we call \emph{$\Oone$-volutive functors} 
    and whose morphisms are natural transformations. We call the isometries in this dagger category \emph{$\Oone$-volutive natural transformations}. 
\end{definition}
\begin{remark}
    This definition of $\Oone$-volutive functors and natural transformations matches the explicit one given e.g. in \cite{Stehouwer2023DaggerCV}.
\end{remark}
\begin{remark}
    Let $(\C,d,\eta)$ and $(\C',d',\eta')$ be $\Oone$-volutive categories and let $(F,\alpha) \colon (\C,d,\eta) \to (\C',d',\eta')$ be an $\Oone$-volutive functor. 
    Then, $(F,\alpha)$ induces a functor $F^{\Oone} \colon (\C^\times)^{\Oone} \to (\C'^\times)^{\Oone}$. 
\end{remark}
\begin{definition}
    Let $(\C,d,\eta,\C_0)$ and $(\C',d',\eta',\C_0')$ be (coherent) dagger categories. A \emph{dagger functor}
    $(F,\alpha) \colon (\C,d,\eta) \to (\C',d',\eta')$ is an $\Oone$-volutive functor for which the induced functor 
    \begin{equation}
        F^{\Oone}|_{\C_0} \colon \C_0 \to (\C'^\times)^{\Oone}
    \end{equation}
    factors through $\C_0' \subseteq (\C'^\times)^{\Oone}$. A \emph{dagger natural transformation} between dagger functors is an 
    $\Oone$-volutive natural transformation.
\end{definition}
\begin{definition}
    We denote by $\Cat^{\Oone\text{-vol}}$ the 2-category of $\Oone$-volutive categories, $\Oone$-volutive functors, and $\Oone$-volutive natural transformations.
\end{definition}
\begin{definition}
    We denote by $\Cat^{\dagger}$ the 2-category of (coherent) dagger categories, dagger functors, and dagger natural transformations between them.
\end{definition}
The following is \cite[Theorem 4.9]{Stehouwer2023DaggerCV}.
\begin{proposition}
    The construction of \Cref{Remark: Oone volutive categories give dagger categories} extends to a functor $S_{\Oone} \colon \Cat^{\Oone\text{-vol}} \to \Cat^{\dagger}$
    which is right adjoint to the forgetful functor $\Cat^{\dagger} \to \Cat^{\Oone\text{-vol}}$.
\end{proposition}

\section{Relative $\Oone$-volutive categories}\label{Section: Relative Oone volutive structures}
In this section we introduce relative $\Oone$-volutive categories. After setting up some of their general theory, we will provide a number of 
good examples of these novel structures and provide several equivalent formulations of a particular special case, namely, lax $\Oone$-volutive categories.
\subsection{Relative $\Oone$-volutive categories}
In this section we introduce the notion of a relative $\Oone$-volutive category. Recall the 2-contravariant $\Oone$-volutive 2-category $(\Cat,(-)^{\opp},\id,\id)$ 
described in \Cref{Example: The Oone volutive structure on Cat}. In the following, we will often only write $\Cat \equiv (\Cat,(-)^{\opp},\id,\id)$ and leave 
the 2-contravariant $\Oone$-volutive structure implicit. 
\begin{definition}\label{Definition: relative Oone volutive category}
    A \emph{relative $\Oone$-volutive category} consists of a 2-contravariant $\Oone$-volutive 2-category $\B$ and a 
    2-contravariant $\Oone$-volutive functor $F \colon \B \to \Cat$. 
\end{definition}
\begin{example}
    The \emph{canonical} relative $\Oone$-volutive category consists of the identity 2-contravariant $\Oone$-volutive functor $\id \colon \Cat \to \Cat$.
\end{example}
\begin{example}
    An $\Oone$-volutive category $(\C,d,\eta)$ is the same as a relative $\Oone$-volutive category whose domain is the final 2-category $\star$. 
\end{example}
Denote by $\Grpd$ the (2,1)-category of groupoids, functors, and natural isomorphisms.
\begin{example}
    Let $\B$ be a (2,1)-category and let $F \colon \B \to \Grpd$ be a functor. By \Cref{Remark: 21categories and 2-contravariant Oone volutive structures}, 
    $\B$ and $\Grpd$ carry canonical 2-contravariant $\Oone$-volutive structures, with respect to which $F$ admits a 2-contravariant $\Oone$-volutive structure. 
    Recalling that the inclusion $\Grpd \to \Cat$ admits a 2-contravariant $\Oone$-volutive structure as well, we obtain a relative $\Oone$-volutive category 
    $\B \to \Grpd \to \Cat$. 
\end{example}
\begin{construction}\label{Construction: from Oone volutive categories to 2-categories}
    Given a category $\C$, we may construct a 2-category $\C[1]$ which has two objects $0$ and $1$ and only one non-trivial Hom-category, namely 
    $\Hom_{\C[1]}(0,1) := \C$. If the category $\C$ is equiped with an $\Oone$-volutive structure $(d,\eta)$, then the 2-category $\C[1]$ 
    receives a 2-contravariant $\Oone$-volutive structure: the 2-functor $d[1] \colon \C[1] \to \C[1]^{2\opp}$ is the identity on objects 
    and $d \colon \Hom_{\C[1]}(0,1) = \C \to \C^{\opp} = \Hom_{\C[1]}(0,1)^{\opp}$ on Hom-categories. The 2-transformation $\eta[1] \colon \id \to d^{2\opp} \circ d$
    has identity 1-morphism components and as its 2-morphism components the components of $\eta^{-1}$. The 2-modification may be taken to be trivial, 
    due to the coherence condition of $\eta$. Hence, $(\C[1],d[1],\eta[1],\id)$ defines a 2-contravariant $\Oone$-volutive category. 
\end{construction}
\begin{remark}
    The terminology $\C[1]$ is motivated by the notation $[1] := \{0 \rightarrow 1\}$ for the walking arrow. In particular, we have $\star[1] \cong [1]$ where 
    $\star$ denotes the terminal category with a single object. On the other hand, we note that $\emptyset[1] \cong \{0,1\}$ is the discrete (2-)category on 
    two objects, where $\emptyset$ denotes the initial category with no objects. 
    \Cref{Construction: from Oone volutive categories to 2-categories} supplies us with a large class of examples of 2-contravariant $\Oone$-volutive 2-categories, 
    with respect to which we may consider relative $\Oone$-volutive categories. 
\end{remark}
For the convenience of the reader, we spell out some of the data of a relative $\Oone$-volutive category in the following. 
\begin{remark}
    Let $(\B,F)$ be a relative $\Oone$-volutive category. First, the functor $F \colon \B \to \Cat$
    assigns to each object $a \in \B$ a category $\C_a$, to each 1-morphism $X \colon a \to b$ a functor 
    $F(X) \colon C_a \to C_b$, and to each 2-morphism $f \colon X \to Y$ a natural transformation 
    $F(f) \colon F(X) \to F(Y)$. Second, the first piece of data of the 2-contravariant $\Oone$-volutive 
    structure is an invertible 2-transformation $\alpha \colon F^{2\opp} \circ d \cong (-)^{1\opp} \circ F$ whose 1-morphism 
    component at $a \in \B$ is given by an equivalence of categories $\alpha_a \colon C_{d(a)} \cong C_a^{\opp}$.
    Third, the second piece of data of the 2-contravariant $\Oone$-volutive structure is an invertible modification 
    fitting into the diagram 
    \begin{equation}
        \begin{aligned}
        \xymatrix{
            ((-)^{1\opp})^{2\opp} \circ F^{2\opp} \circ d \ar[r]^-{\id \circ \alpha} \ar[d]^-{(\alpha^{2\opp})^{-1} \circ \id} & ((-)^{1\opp})^{2\opp} \circ (-)^{1\opp} \circ F \ar[d]^-{\cong} \\
            F \circ d^{2\opp} \circ d \ar[r]^-{\id \circ \eta^{-1}} & F 
        }
        \end{aligned}
    \end{equation}
    whose component at $a \in \B$ is given by a natural isomorphism fitting into the diagram 
    \begin{equation}
        \begin{aligned}
            \xymatrix{
            C_a \ar[ddr]_-{\id} \ar[r]^-{\alpha_a^{\opp}} & C_{d(a)}^{\opp} \ar[d]^-{\alpha_{d(a)}} \\
            & C_{d^2(a)} \ar[d]^-{F(\eta_a^{-1})} \\
            & C_a
        }
        \end{aligned}
    \end{equation}
    which itself satisfies a coherence condition.
\end{remark}
\begin{remark}
    We may think of relative $\Oone$-volutive categories as objects in the slice 3-category $2\Cat^{\Oone_2\text{-vol}}_{/\Cat}$. Here the subscript $2$ refers 
    to the fact that we consider 2-contravariant $\Oone$-volutive 2-categories (a fact that we will leave implicit in the following). In the following we wish to 
    describe the (higher) morphisms in the 3-category 
    $2\Cat^{\Oone_2\text{-vol}}_{/\Cat}$ more explicitly. First, a \emph{functor} of relative $\Oone$-volutive categories $(\B,F)$ and $(\B',F')$ consists of 
    an $\Oone$-volutive 2-functor $G \colon \B \to \B'$ together with an invertible $\Oone$-volutive 2-transformation $\alpha \colon F \to F' \circ G$. Second, 
    a \emph{transformation} of relative $\Oone$-volutive functors $(G_1,\alpha_1)$ and $(G_2,\alpha_2)$ consists of an $\Oone$-volutive 2-transformation $\beta \colon G_1 \to G_2$ 
    and an invertible $\Oone$-volutive modification $\xi$ satisfying the evident compatibility condition with $\alpha_1,\alpha_2$. Finally, a \emph{modification} of 
    relative $\Oone$-volutive transformations $(\beta^a,\xi^a)$ and $(\beta^b,\xi^b)$ consists of an $\Oone$-volutive modification $\omega \colon \beta^a \to \beta^b$
    satisfying a compatibility condition with $\xi^a,\xi^b$.
\end{remark}
\begin{remark}
    If we wish to consider relative $\Oone$-volutive categories with a fixed domain $\B$, we may also consider the 2-category 
    $\Hom_{2\Cat^{\Oone_2\text{-vol}}}(\B,\Cat)$.
\end{remark}

\subsection{Lax $\Oone$-volutive categories}
In this section we introduce lax $\Oone$-volutive categories and prove that every closed symmetric monoidal category admits a lax $\Oone$-volutive structure.
\begin{definition}\label{Definition: lax O(1)-volutive category}
    A \emph{lax $\Oone$-volutive structure} on a category $\C$ consists of a functor $d \colon \C \to \C^{\opp}$
    and a natural transformation $\eta \colon \id \to d^2$ satisfying $d(\eta_a) \circ \eta_{d(a)} = \id_{d(a)}$ for all $a \in \C$. 
    A category $\C$ together with a lax $\Oone$-volutive structure $(d,\eta)$ is called a \emph{lax $\Oone$-volutive category}.
\end{definition}
\begin{remark}
    Lax $\Oone$-volutive categories have appeared in the literature on hermitian algebraic K-theory before, under the name \emph{category with duality}, see e.g. 
    \cite[Definition 3.1]{Schlichting2010}. However, adopting this terminology would be rather confusing in the upcoming discussion. 
\end{remark}
\begin{remark}\label{Remark: lax homotopy fixed points vs lax volutive categories}
    Lax $\Oone$-volutive categories are \emph{not} lax homotopy fixed points with respect to the canonical non-trivial $\Oone$-action on the (2,1)-category 
    $\Cat_{(2,1)}$. The latter are simply $\Oone$-volutive categories, as the invertibility of the natural transformation $\eta$ automatically 
    implies the invertiblity of $d$. 
\end{remark}
\begin{remark}\label{Remark: lax Oone volutive structures are many}
    Let $\C$ be a category and let $(d,\eta)$ be a lax $\Oone$-volutive structure on $\C$. Let $k$ be a non-negative integer. We obtain a new lax $\Oone$-volutive 
    structure whose functor is $\overline{d}_k := (d \circ d^{\opp})^k \circ d$ and whose natural transformation is 
    \begin{equation}
        \overline{\eta}_k := \eta^{k+1} \colon \id \to (d \circ d^{\opp})^{2k+2} \cong (d^{\opp} \circ d)^k \circ d^{\opp} \circ (d \circ d^{\opp})^k \circ d = \overline{d}^{\opp}_k \circ \overline{d}_k
    \end{equation}
    which satisfies the coherence condition of a lax $\Oone$-volutive structure since $\eta$ does. If we assume that $(d,\eta)$ is an $\Oone$-volutive structure, 
    $(\overline{d}_k,\overline{\eta}_k)$ is an $\Oone$-volutive structure as well and equivalent to $(d,\eta)$. 
\end{remark}
\begin{example}\label{Example: Banach spaces are lax volutive}
    Let $\Ban$ denote the category of Banach spaces and continuous/bounded linear maps. We construct a lax $\Oone$-volutive structure on $\Ban$. First, we define 
    a functor $(-)^* \colon \Ban \to \Ban^{\opp}$ by assigning each Banach space $V$ to its dual $V^* = \Hom_{\Ban}(V,\mathbb{C})$ and each continuous linear map 
    $f \colon V \to W$ to the transpose $f^* \colon W^* \to V^*, \phi \mapsto \phi \circ f$. Second, we define a natural transformation $\iota \colon \id \to (-)^{**}$ 
    whose component at $V \in \Ban$ is given by the continuous linear map $\iota_V \colon V \to V^{**}, \iota_V(v)(\phi) = \phi(v)$. To see that $\iota$ is indeed 
    a natural transformation, it suffices to note that 
    \begin{equation}
        ((\iota_W \circ f)(v))(\phi) = \iota_W(f(v))(\phi) = \phi(f(w)) = f^*(\phi)(v) = ((f^{**} \circ \iota_V)(v))(\phi).
    \end{equation}
    Lastly, we check that $\iota_V^* \circ \iota_{V^*} = \id_{V^*}$. Indeed, we have 
    \begin{equation}
        ((\iota_V^* \circ \iota_{V^*})(\phi))(v) = (\iota_V^*(\iota_{V^*}(\phi)))(v) = \iota_{V^*}(\phi)(\iota_V(v)) = \iota_V(v)(\phi) = \phi(v)
    \end{equation}
    so that $\iota_V^* \circ \iota_{V^*} = \id_{V^*}$ holds. Hence, $\Ban$ supports a lax $\Oone$-volutive structure.
\end{example}
\begin{example}
    Any $\Oone$-volutive category is a lax $\Oone$-volutive category. 
\end{example}
In the converse direction, we have the following.
\begin{construction}\label{Construction: Oone volutions from lax Oone volutions}
    Any lax $\Oone$-volutive category $(\C,d,\eta)$ gives rise to an $\Oone$-volutive category which can be described as follows: 
    the underlying category is the full subcategory $\hat{\C} \subset \C$ whose objects $a$ are those for which 
    $\eta \colon a \to d^2(a)$ is an isomorphism, and the $\Oone$-volutive structure is given by the restriction of $(d,\eta)$ to $\hat{C}$.
\end{construction}
\begin{example}
    The category of reflexive Banach spaces discussed in \Cref{Example: Banach spaces are Oone volutive} inherits its $\Oone$-volutive structure 
    from the lax $\Oone$-volutive category $\Ban$ discussed in \Cref{Example: Banach spaces are lax volutive} via \Cref{Construction: Oone volutions from lax Oone volutions}.
\end{example}
\begin{definition}\label{Definition: right closed monoidal category}
    Let $\C$ be a monoidal category. We say that $\C$ is \emph{right closed} if and only if for every pair of objects $a,b \in \C$ there exists another object 
    $b^a$ and a map $\ev_a^b \colon b^a \otimes a \to b$ with the following universal property: for every $c \in \C$, the induced map 
    $\psi \colon \hom_{\C}(c,b^a) \to \hom_{\C}(c \otimes a,b)$ is a bijection which is natural in $b$ and $c$. Put differently, the construction $b \mapsto b^a$ 
    determines a right adjoint to the functor $c \mapsto c \otimes a$ for every $a \in \C$.\\
    Similarly, we will say that a monoidal category $\C$ is \emph{left closed} if the functor $a \otimes (-) \colon \C \to \C$ 
    admits a right adjoint for every $a \in \C$. For symmetric/braided monoidal categorues being right closed is equivalent to being left closed, so that we will 
    simply refer to a category as being \emph{closed} in this case.
\end{definition}
\begin{construction} \label{Construction: O(1)-volutive categories from closed monoidal ones}
    Let $\C$ be a right closed monoidal category in the sense of \Cref{Definition: right closed monoidal category}. In the following we will use the notation 
    $\ev_a := \ev_a^1$ for $a \in \C$ and $1$ for the monoidal unit. We claim that the construction $a \mapsto 1^a$ extends to a functor 
    $1^{(-)} \colon \C \to \C^{\opp}$. Indeed, for a morphism $X \colon a \to b$, we set
    \begin{equation}
        1^X := \psi^{-1}\left( \xymatrix{1^b \otimes a \ar[r]^-{\id \otimes X} & 1^b \otimes b \ar[r]^-{\ev_b} & 1} \right) \in \hom_{\C}(1^b,1^a)
    \end{equation}
    and note first that $1^{\id_a} = \psi^{-1}(\ev_a) = \id_{1^a}$. Using the naturality of $\psi$ we find
    \begin{align}\label{Equation: pulling morphisms through evaluations}
        \begin{split}
        \ev_a \circ (1^X \otimes \id_a) 
        &= \psi(\id_{1^a}) \circ(1^X \otimes \id_a) \\ %Formula from the line above
        &= \psi(\id_{1^a} \circ 1^X)\\ %Naturality of $\psi$
        &= \psi(1^X)\\ %Trivial
        &= \psi \psi^{-1} (\ev_b \circ (\id_{1^b} \otimes X))\\ %Definition
        &= \ev_b \circ (\id_{1^b} \otimes X), %Trivial
        \end{split}
    \end{align}
    implying the commutativity of the diagram 
    \begin{equation}
        \begin{aligned}
        \xymatrix{
            & 1^c \otimes b \ar[dr]^{1^Y \otimes \id} \ar[rr]^{\id \otimes Y} && 1^c \otimes c \ar[dr]^{\ev_c} & \\
            1^c \ar[ur]^{\id \otimes X} \ar[dr]^{1^Y \otimes \id} \otimes a && 1^b \otimes b \ar[rr]^{\ev_b} && 1 \\
            & 1^b \otimes a \ar[ur]^{\id \otimes X} \ar[rr]^{1^X \otimes \id} && 1^a \otimes a \ar[ur]^{\ev_a} &
        }
        \end{aligned}
    \end{equation}
    for $X \colon a \to b$ and $Y \colon b \to c$ morphisms in $\C$ and hence functoriality. \\
    Assume now in addition that $\C$ is a (right) closed \emph{braided} monoidal category. We obtain for each $a \in \C$ a canonical map $\eta_a \colon a \mapsto 1^{1^a}$ defined as
    \begin{equation}
         \eta_a := \psi^{-1}\left( \xymatrix{a \otimes 1^a \ar[r]^-{\beta_{a,1^a}} & 1^a \otimes a \ar[r]^-{\ev_a} & 1} \right).
    \end{equation}
    We claim that $\eta \colon \id \to 1^{1^{(-)}}$ is a natural transformation. Indeed, we have 
    \begin{align}
        \begin{split}
            1^{1^X} \circ \eta_a &= 
            \psi^{-1}(\ev_{1^a} \circ (\id \otimes 1^X)) \circ \eta_a \\ %Definition
            &= \psi^{-1}(\ev_{1^a} \circ (\id \otimes 1^X) \circ (\eta_a \otimes \id)) \\ %Naturality of $\psi$
            &= \psi^{-1}(\ev_{1^a} \circ (\eta_a \otimes \id) \circ (\id \otimes 1^X)) \\ %Interchange law
            &= \psi^{-1}(\ev_{a} \circ \beta \circ (\id \otimes 1^X)) \\ %Apply $\psi^{-1}$ to both sides, use naturality of $\psi$, and $\psi(\id)=\ev$
            &= \psi^{-1}(\ev_{a} \circ (1^X \otimes \id) \circ \beta) \\
            %Use the naturality of the braiding 
            &= \psi^{-1}(\ev_{a} \circ (\id \otimes X) \circ \beta) \\
            %Ue the naturality expression of $\ev$ from above
            &= \psi^{-1}(\ev_{a} \circ \beta \circ (X \otimes \id)) \\
            %Use naturality of the braiding 
            &= \psi^{-1}(\ev_{a} \circ \beta) \circ X %Use naturality of $\psi$
            = \eta_b \circ X. %Definition
        \end{split} 
    \end{align} 
    Assume now in addition that $\C$ is a closed \emph{symmetric} monoidal category. 
    We claim that $1^{\eta_a} \circ \eta_{1^a} = \id_{1^a}$ is satisfied for every $a \in \C$. Indeed, we have 
    \begin{align}
        \begin{split}
            1^{\eta_a} \circ \eta_{1^a} 
            &= \psi^{-1}(\ev_{1^{1^a}} \circ (\id \otimes \eta_a)) \circ \eta_{1^a} \\ %Definition
            &= \psi^{-1}(\ev_{1^{1^a}} \circ (\id \otimes \eta_a) \circ (\eta_{1^a} \otimes \id)) \\ %Naturality of $\psi$
            &= \psi^{-1}(\ev_{1^{1^a}} \circ (\eta_{1^a} \otimes \id) \circ (\id \otimes \eta_a)) \\ %Interchange law
            &= \psi^{-1}(\ev_{1^a} \circ \beta \circ (\id \otimes \eta_a)) \\
            %Same relation as in the naturality proof
            &= \psi^{-1}(\ev_{1^a} \circ (\eta_a \otimes \id) \circ \beta) \\
            %Naturality of the braiding
            &= \psi^{-1}(\ev_a \circ \beta \circ \beta) \\ %Same relation as in the naturality proof
            &= \psi^{-1}(\ev_a)  %Symmetry of the braiding
            = \id_{1^a}. %Formula from functoriality proof
        \end{split}
    \end{align}
    Hence, every closed symmetric monoidal category admits a functor $1^{(-)} \colon \C \to \C^{\opp}$ and a natural transformation $\eta \colon \id \to 1^{1^{(-)}}$ satisfying $1^{\eta_a} \circ \eta_{1^a} = \id_{1^a}$ for every $a \in \C$. 
\end{construction}
We record our findings in the following.
\begin{theorem}\label{Theorem: closed symmetric monoidal categories are lax volutive}
    Every closed symmetric monoidal category admits a lax $\Oone$-volutive structure in the sense of \Cref{Definition: lax O(1)-volutive category}.
\end{theorem}
\begin{remark}
    More generally, any symmetric monoidal involution on a closed symmetric monoidal category gives rise to a lax $\Oone$-volutive structure; 
    \Cref{Theorem: closed symmetric monoidal categories are lax volutive} is then recovered as the special case of the identity involution. 
\end{remark}
\begin{remark}
    We have explained in \Cref{Remark: lax homotopy fixed points vs lax volutive categories} that lax $\Oone$-volutive categories are \emph{not} lax homotopy fixed 
    points. Instead, the terminology is motivated by the following observation: the functor $1^{(-)} \colon \C \to \C^{\opp}_{\otimes^{\opp}}$ is oplax monoidal. 
    Here, $\otimes^{\opp}$ refers to the opposite monoidal structure, i.e. $a \otimes^{\opp} b := b \otimes a$. To see the claim, we construct oplax\footnote{While 
    the data we supply here resembles lax monoidality data, we recall that the target of our functor is $\C^{\opp}$, so that this data is reversed.} monoidality 
    data for our functor, which amounts to morphisms $\phi_{a,b} \colon 1^a \otimes^{\opp} 1^b \to 1^{a \otimes b}$ in $\C$ for each $a,b$ and $1 \to 1^1$. The former 
    are given as the preimage of the morphism 
    \begin{equation}
        \xymatrix{
            1^a \otimes^{\opp} 1^b \otimes (a \otimes b) = 1^b \otimes 1^a \otimes (a \otimes b) \ar[rr]^-{\id \otimes \ev_a \otimes \id} &&  1^b \otimes b \ar[r]^-{\ev_b} & 1
        }
    \end{equation}
    under the bijection $\hom_{\C}(1^a \otimes^{\opp} 1^b, 1^{a \otimes b}) \cong \hom_{\C}((1^a \otimes^{\opp} 1^b) \otimes (a \otimes b), 1)$ and the 
    latter is given as the preimage $\hom_{\C}(1,1^1) \cong \hom_{\C}(1 \otimes 1, 1) \cong \hom_{\C}(1,1)$. The naturality of this data is deduced from 
    the naturality of the defining bijections $\hom_{\C}(c,b^a) \to \hom_{\C}(c \otimes a,b)$. 
\end{remark}
\begin{example}
    The category of (complete) bornological vector spaces is closed symmetric monoidal and hence admits a lax $\Oone$-volutive structure. It contains 
    the lax $\Oone$-volutive category of Banach spaces as a full subcategory. We will discuss this example in more detail in \Cref{Subsection: Bornological vector spaces}.
\end{example}
\begin{example}
    Let $R$ be a $*$-ring. Then, the category $\Mod_{R}$ of right $R$-modules admits a lax $\Oone$-volutive structure. We will discuss 
    this example in more detail in \Cref{Subsection: Modules over star-rings}. 
\end{example}
\begin{example}
    Let $R$ be a commutative ring. Then, the category $\Cha(R)$ of unbounded chain complexes of $R$-modules is a closed symmetric monoidal (model) category, 
    see e.g. \cite[Proposition 4.2.13]{hovey2007model}, and hence admits a lax $\Oone$-volutive structure. More precisely, the monoidal product on 
    $\Cha(R)$ is given by $(X \otimes Y)_n = \bigoplus_n (X_k \otimes_R Y_{n-k})$, the monoidal unit $1$ is $R$ interpreted 
    as a chain complex concentrated in degree $0$, and we have a chain complex $Y^X$ with components
    \begin{equation}
        (Y^X)_n \equiv \Hom(X,Y)_n = \Pi_k \hom_{R}(X_k,Y_{n+k}) = \hom_{\Cha(R)}(X,Y[n])
    \end{equation}
    and differentials induced %(with the appropriate signs) 
    from $X$ and $Y$, respectively. In particular, we have 
    \begin{equation}
        (1^X)_n \equiv \Hom(X,1)_n  = \Pi_k \hom_{R}(X_k,1_{n+k}) = \hom_{R}(X_n,R)
    \end{equation}
    and 
    \begin{equation}
        (1^{1^X})_n = \Hom(1^X,1)_n = \hom_{R}((1^X)_n,R) = \hom_{R}(\hom_{R}(X_n,R),R)
    \end{equation}
    with a map of chain complexes $X \to (1^{1^X})$ given in degree $n$ by 
    \begin{equation}
        X_n \to \hom_{R}(\hom_{R}(X_n,R),R), a \mapsto \Ev_a \qquad \text{where} \qquad \Ev_a(f) = f(a)
    \end{equation}
    omitting the respective differentials. In particular, under the inclusion $\Rmod \to \Cha(R)$ sending each $R$-module $X$ to the chain 
    complex complex concentrated in degree zero with value $X$, this reproduces the lax $\Oone$-volutive structure of 
    \Cref{Proposition: modules over star rings admits lax volutive structure}. Moreover, noting that $(1^{1^X})_n = (1^{1^{X_n}})$, applying 
    \Cref{Construction: Oone volutions from lax Oone volutions} yields the $\Oone$-volutive category 
    $\widehat{\Cha(R)} \subseteq \Cha(R)$ consisting of those chain complexes which are degreewise reflexive modules. 
\end{example}
\begin{remark}
    There are notions of duality for (symmetric) monoidal categories between rigidity and closedness, two of which we will recall in the following and 
    place in the (lax) $\Oone$-volutive context. A \emph{symmetric Grothendieck-Verdier (GV) category} or \emph{symmetric $*$-autonomous category} 
    is a closed symmetric monoidal category $\C$ together with a \emph{dualizing object}, that is, an object $\mathfrak{d}$ such that the canonical morphism 
    $a \to \mathfrak{d}^{\mathfrak{d}^a}$ is an isomorphism for all $a \in \C$. Repeating \Cref{Construction: O(1)-volutive categories from closed monoidal ones} 
    with the monoidal unit replaced by the dualizing object $\mathfrak{d}$ yields an $\Oone$-volutive structure on $\C$.\\
    In the special case where the dualizing object $\mathfrak{d} = 1_{\C}$ is the monoidal unit, $\C$ is called an \emph{$r$-category}.
\end{remark}
\begin{remark}
    We informally summarize the constructions of (lax) $\Oone$-volutive categories from symmetric monoidal categories with some type of 
    duality in the following commutative diagram. 
    \begin{equation}
        \begin{aligned}
            \xymatrix{
                \text{rigid} \ar[r]^-{} \ar[dr] & \text{$r$-category} \ar[d] \ar[r]^-{} & \text{GV/*-autonomous}\ar[r]^-{} \ar[dl] & \text{closed} \ar[dl] \\
                & \text{$\Oone$-volutive} & \text{lax $\Oone$-volutive} \ar@{-->}[l] & 
            }
        \end{aligned}
    \end{equation}
    Here, solid arrows preserve the underlying categories while dashed arrows generically change them.
\end{remark}
In the remainder of this section, we establish some of the general theory of lax $\Oone$-volutive categories.
\begin{construction}\label{Construction: Functor categories inherit lax O(1)-volutions}
    Let $(\C,d,\eta)$ and $(\C',d',\eta')$ be lax $\Oone$-volutive categories. We construct a lax $\Oone$-volutive structure on the category of functors 
    $\E := \Fun(\C,\C')$. First, define a functor $\tilde{d} \colon \E \to \E^{\opp}$ by setting 
    \begin{equation}
        \begin{cases}
            \tilde{d}(F) = (d')^{\opp} \circ F^{\opp} \circ d \\
            \tilde{d}(\alpha \colon F \to G) = \id_{(d')^{\opp}} \circ \alpha^{\opp} \circ \id_{d} \colon \tilde{d}(G) \to \tilde{d}(F).
        \end{cases}
    \end{equation}
    We then define the component of the natural transformation $\tilde{\eta}$ at $F \in \E$ to be 
    \begin{equation}
        \xymatrix@C=4em{
            F \cong \id_{\D} \circ F \circ \id_{\C} \ar[r]^-{\eta' \circ \id_{F} \circ \eta} & ({d'}^{\opp} \circ d') \circ F \circ (d^{\opp} \circ d) = 
            (\tilde{d}^{\opp} \circ \tilde{d})(F)
        }
    \end{equation}
    which one checks to satisfy the coherence condition so that $(\tilde{d},\tilde{\eta})$ defines a lax $\Oone$-volutive structure on $\E$.
\end{construction}
\begin{construction}
    Let $(\C,d,\eta)$ and $(\C',d',\eta')$ be lax $\Oone$-volutive categories. Consider the lax $\Oone$-volutive category $(\Fun(\C,\C'),\tilde{d},\tilde{\eta})$
    described in \Cref{Construction: Functor categories inherit lax O(1)-volutions}. We define the category of \emph{lax $\Oone$-volutive functors} between 
    $(\C,d,\eta)$ and $(\C',d',\eta')$ to be $\LaxHerm(\Fun(\C,\C'),\tilde{d},\tilde{\eta})$, see \Cref{Definition: the category of lax hermitian fixed points}. 
    We will refer to the morphisms of this category as \emph{lax $\Oone$-volutive transformations}.
\end{construction}
\begin{remark}
    In the following we wish to describe functors and transformations of lax $\Oone$-volutive categories more explicitly. A lax $\Oone$-volutive functor 
    $(\C,d,\eta) \to (\C',d',\eta')$ is a pair $(F,\alpha)$ consisting of a functor $F \colon \C \to \C'$ and a natural transformation 
    $\alpha \colon d' \circ F \to F^{\opp} \circ d$ satisfying $(\alpha^{\opp} \circ \id_d)(\id_F \circ \eta) = (\id_{d^{\opp}} \circ \alpha)(\eta' \circ \id_F)$.
    Given two such lax $\Oone$-volutive functors $(F_1,\alpha_1)$ and $(F_2,\alpha_2)$, a lax $\Oone$-volutive natural transformation $(F_1,\alpha_1) \to (F_2,\alpha_2)$
    is a natural transformation $\xi \colon F_1 \to F_2$ satisfying $(\xi^{\opp} \circ \id_d)\alpha_2(\id_{d'}\xi) = \alpha_1$. 
\end{remark}
\begin{definition}
    We denote by $\Cat^{\text{lax}\Oone\text{-vol}}$ the 2-category of lax $\Oone$-volutive categories, lax $\Oone$-volutive functors, and lax $\Oone$-volutive 
    natural transformations. 
\end{definition}
\begin{definition}\label{Definition: product of lax Oone volutive categories}
    Let $(\C,d,\eta)$ and $(\C',d',\eta')$ be lax $\Oone$-volutive categories. Their \emph{(cartesian) product} is the lax $\Oone$-volutive category
    \begin{equation}
        (\C,d,\eta) \times (\C',d',\eta') := (\C \times \C',d \times d',\eta \times \eta').
    \end{equation}
    The cartesian product endows $\Cat^{\text{lax}\Oone\text{-vol}}$ with a symmetric monoidal structure. 
\end{definition}
\begin{remark}
    Let $(\C,d,\eta)$ and $(\C',d',\eta')$ be lax $\Oone$-volutive categories and let 
    $(F,\alpha) \colon (\C,d,\eta) \to (\C',d',\eta')$ be a lax $\Oone$-volutive functor. We have described a procedure in \Cref{Construction: Oone volutions from lax Oone volutions}
    that takes as an input a lax $\Oone$-volutive category, and produces an $\Oone$-volutive category. This procedure is \emph{not} compatibile with 
    lax $\Oone$-volutive functors, i.e. $(F,\alpha)$ does not descend to an ($\Oone$-volutive) functor between the associated $\Oone$-volutive categories. 
    The situation is better if $\alpha$ is a natural isomorphism, that is, $(F,\alpha)$ is really an $\Oone$-volutive functor between lax $\Oone$-volutive categories. 
\end{remark}

\subsection{Lax $\Oone$-volutive structures are relative}
In this section we present various equivalent formulations of the notion of a lax $\Oone$-volutive category; in particular, we will show that lax $\Oone$-volutive 
categories are special relative $\Oone$-volutive categories. We start with a point of view that is essentially covered in \cite{heine2025infinitycategories}.
\begin{remark}\label{Construction: pairings and lifts}
    Let $F \colon \C \times \C \to \Set$ be a functor, which we may also refer to as a \emph{pairing} on $\C$. The pairing $F$ induces under the usual Hom-tensor 
    adjunction a functor $F \colon \C \to \Fun(\C,\Set)$. Equivalently, this can be understood as a functor $F \colon \C \to \mathcal{P}(\C^{\opp})$ into the category 
    of presheaves on $\C^{\opp}$. The latter admits the yoneda embedding $\yon \colon \C^{\opp} \to \mathcal{P}(\C^{\opp})$, so that we may ask whether there exists 
    a lift in the diagram 
    \begin{equation}\label{Equation: the lifting diagram for pairings}
        \begin{aligned}
            \xymatrix{
            \C^{\opp} \ar[dr]^-{\yon} & \\
            \C \ar[r]^-{F} \ar@{-->}[u]_-{} & \mathcal{P}(\C^{\opp}) 
        }
        \end{aligned}
    \end{equation}
    and if so whether this lift is an equivalence of categories or not.
\end{remark}
\begin{definition}\label{Definition: properties of pairings}
    In the situation of \Cref{Construction: pairings and lifts}, we will say that $F$ is \emph{representable}\footnote{The terminology is motivated by 
    \cite[Definition 5.2.1.8]{HA}.} if there exists a lift in the diagram. 
    We will say that $F$ is \emph{perfect} or \emph{non-degenerate} if it is representable and the lift is an equivalence of categories. We will say that 
    $F$ is \emph{symmetric} if $F \circ \beta_{\C,\C} \cong F$ for $\beta$ the symmetric braiding on $\Cat$. 
\end{definition}
\begin{lemma}\label{Lemma: lax Oone volutive categories and representable symmetric functors}
    Lax $\Oone$-volutive categories in the sense of \Cref{Definition: lax O(1)-volutive category} are the same as representable symmetric functors 
    in the sense of \Cref{Definition: properties of pairings}.
\end{lemma}
\begin{proof}
    Let $(\C,d,\eta)$ be a lax $\Oone$-volutive category. We define a functor $\Phi_d \colon \C \times \C \to \Set$ as follows: 
    given an object $(a,b) \in \C \times \C$, we define 
    \begin{equation}
        \Phi_d(a,b) = \Hom_{\C}(d(a),b),
    \end{equation}
    and given a morphism $(X,Y) \colon (a,b) \to (a',b')$ we define 
    \begin{equation}
        \Phi_d(X,Y) \colon \Hom_{\C}(d(a),b) \to \Hom_{\C}(d(a'),b'), (f \colon d(a) \to b) \mapsto Y \circ f \circ d(X),
    \end{equation}
    which is clearly functorial. The functor $\Phi_d \colon \C \times \C \to \Set$ is symmetric since we have a bijection
    \begin{equation}
        \Hom_{\C}(d(a),b) \cong \Hom_{\C^{\opp}}(a,d^{\opp}(b)) = \Hom_{\C}(d(b),a)
    \end{equation}
    which is natural in $(a,b) \in \C \times \C$, where we have identified $\C \cong (\C^{\opp})^{\opp}$. Put differently, 
    we have a natural isomorphism $\Phi_d \circ \beta \cong \Phi_d$ where $\beta$ denotes the symmetric braiding on $\Cat$.
    The functor $\Phi_d$ is representable by construction. \\
    Conversely, let $F \colon \C \times \C \to \Set$ be a functor that is symmetric and representable. We will construct a lax $\Oone$-volutive structure on $\C$ 
    as follows: first, since $F$ is representable we have a functor $d \colon \C \to \C^{\opp}$ fitting into the commutative diagram of
    \Cref{Equation: the lifting diagram for pairings}. By definition, this amounts to an equivalence of functors
    \begin{equation}
        \Hom_{\C}(d(-),-) \cong F(-,-) \colon \C \times \C \to \Set
    \end{equation}
    so that by symmetry of $F$, we have a natural bijection 
    \begin{equation}
        \Hom_{\C}(d(a),b) \cong F(a,b) \cong F(b,a) \cong \Hom_{\C}(d(b),a) = \Hom_{\C^{\opp}}(a,d^{\opp}(b))
    \end{equation}
    for any $a,b \in \C$. In other words, $d^{\opp}$ is right adjoint to $d$. The adjunction supplies us with a natural transformation $\id \to d^{\opp} \circ d$. 
    Explicitly, the functor $d$ induces a map
    \begin{equation}
        \Hom_{\C}(a,b) \to \Hom_{\C^{\opp}}(d(a),d(b)) \cong \Hom_{\C}(a,d^2(b))
    \end{equation}
    for any $a \in \C^{\opp}, b \in \C$ where in the second step we have used the adjunction. By the yoneda lemma, this
    asserts the existence of maps $\eta_b \colon b \to d^2(b)$ which assemble into a natural transformation $\id \to d^{\opp} \circ d$.
    One checks that the pair $(d,\eta)$ defines a lax $\Oone$-volutive structure on $\C$.
\end{proof}
We have seen in the proof of \Cref{Lemma: lax Oone volutive categories and representable symmetric functors} that for a lax $\Oone$-volutive category $(\C,d,\eta)$, 
$d^{\opp}$ is right adjoint to $d$. The following two points of view are centered around this observation. We start by recalling the following. 
\begin{definition}\label{Definition: walking adjunction}
    We define a 2-category $\Adj$ which has two objects $a$ and $b$, which has 1-morphisms freely generated by 1-morphisms $X \colon a \to b$ 
    and $Y \colon b \to a$, and 2-morphisms freely generated by 2-morphisms $f \colon \id_a \to Y \circ X$ and $g \colon X \circ Y \to \id_b$
    subject to the condition that $(g \circ \id_X)\bullet(\id_X \circ f) = \id_X$ and $(\id_Y \circ g) \bullet (f \circ \id_Y) = \id_Y$ 
    holds. We call the 2-category $\Adj$ the \emph{walking adjunction}. 
\end{definition}
\begin{remark}
    An adjunction in a 2-category $\B$ is the same as a 2-functor $\Adj \to \B$; in this sense, $\Adj$ classifies all possible adjunctions. 
\end{remark}
\begin{lemma}\label{Lemma: 2-contravariant Oone volutive structure on the walking adjunction}
    The walking adjunction $\Adj$ admits a 2-contravariant $\Oone$-volutive structure. 
\end{lemma}
\begin{proof}
    The 2-contravariant $\Oone$-volutive structure is given by the strictly involutive 2-functor $F \colon \Adj \to \Adj^{2\opp}$ determined by the assignment 
    \begin{align}
        \begin{aligned}
        &F(a) = b \\
        &F(b) = a \\
        &(F(X) \colon F(a) \to F(b)) = Y \colon b \to a \\
        &(F(Y) \colon F(b) \to F(a)) = X \colon a \to b \\
        &(F(f) \colon F(YX) = F(Y) \circ F(X) \to F(\id_a)) =  g \colon X \circ Y \to \id_b \\
        &(F(g) \colon F(\id_b) \to F(XY) = F(X) \circ F(Y)) = f \colon \id_a \to Y \circ X.
        \end{aligned}
    \end{align}
\end{proof}
Before stating our next comparison result, we consider the following. 
\begin{definition}\label{Definition: volutive adjunction}
    Let $\B$ be a 2-category together with a 2-contravariant $\Oone$-volutive structure $(d,\eta,\tau)$. Let $(c,Z,h)$ be a triple where 
    $c \in \B$ is an object, $Z \colon c \to d(c)$ is a 1-morphism, and $h \colon \id_c \to \eta_c^{-1} \circ d(Z) \circ Z$ is a 2-morphism 
    for which the composition
    \begin{equation}\label{Equation: first coherence condition of h in volutive adjunction}
        \begin{aligned}
        \xymatrix@R=0.5em{
            Z \ar[r]^-{\id \circ h} & Z \eta_c^{-1}  d(Z)  Z \ar[r]^-{\eta_Z \circ \id} & \eta_{d(c)}^{-1} d^2(Z)d(Z)Z 
            \ar[r]^-{\tau_c \circ \id} & d(\eta_c^{-1}) d^2(Z) d(Z) Z \\
            \ar[r]_-{\cong} & d(\eta_c^{-1} d(Z)Z) Z \ar[r]_-{d(h) \circ \id} & Z
        }
        \end{aligned}
    \end{equation}
    coincides with the identity 2-morphism on $Z$ and the composition 
    \begin{equation}\label{Equation: second condition of h in volutive adjunction}
        \begin{aligned}
        \xymatrix@R=0.5em{
            \eta_c^{-1} d(Z) \ar[r]^-{h \circ \id} & \eta_c^{-1} d(Z) Z \eta_c^{-1} d(Z) \ar[r]^-{\id \circ \eta_Z \circ \id} & 
            \eta_c^{-1} d(Z) \eta_{d(c)}^{-1} d^2(Z) d(Z) \\
            \ar[r]_-{\id \circ \tau_c \circ \id} & 
            \eta_c^{-1} d(Z) d(\eta_c^{-1}) d^2(Z) d(Z) \ar[r]_-{\id \circ d(h)} &\eta_c^{-1} d(Z)
        }
        \end{aligned}
    \end{equation}
    coincides with the identity 2-morphism on $\eta_c^{-1} d(Z)$. \\
    We call the triple $(c,Z,h)$ a \emph{volutive adjunction} in $\B \equiv (\B,d,\eta,\tau)$.
\end{definition}
\begin{remark}\label{Remark: Simplicification of volutive adjunctions}
    In the situation of \Cref{Definition: volutive adjunction}, assume the functor $d$ is strictly involutive and that $\eta$ and $\tau$ are trivial. 
    In this case, we can describe a volutive adjunction as a triple $(c,Z,h)$ where $c \in \B$ is an object, $Z \colon c \to d(c)$ is a 1-morphism in $\B$, 
    and $h \colon \id_c \to d(Z) \circ Z$ is a 2-morphism in $\B$ satisfying $(d(h) \circ \id)\bullet (\id \circ h) = \id_Z$ and 
    $(\id \circ d(h)) \bullet (h \circ \id) = \id_{d(Z)}$. Note that the first equation already implies the second equation, by applying the 
    2-contravariant functor $d$. We expect that the same holds true for general 2-contravariant $\Oone$-volutive structures $(d,\eta,\tau)$; however, 
    we will not need this fact in the following.
\end{remark}
\begin{lemma}\label{Lemma: volutive adjunctions and lax volutive categories}
    Lax $\Oone$-volutive categories are the same as volutive adjunctions in the 2-contravariant $\Oone$-volutive 2-category $(\Cat,(-)^{\opp},\id,\id)$. 
\end{lemma}
\begin{proof}
    Spelling out 2-contravariant $\Oone$-volutive 2-category $(\Cat,(-)^{\opp},\id,\id)$ according to \Cref{Remark: Simplicification of volutive adjunctions} yields 
    precisely the same data as that of a lax $\Oone$-volutive category. 
\end{proof}
\begin{lemma}\label{Lemma: volutive adjunctions and volutive functors from the walking adjunction}
    Let $(\B,d,\eta,\tau)$ be a 2-contravariant $\Oone$-volutive 2-category. Then, volutive adjunctions in $(\B,d,\eta,\tau)$ are the same as 
    2-contravariant $\Oone$-volutive 2-functors $\Adj \to \B$. 
\end{lemma}
\begin{proof}
    Let $(c,Z,h)$ be a volutive adjunction in $(\B,d,\eta,\tau)$. We define a functor $G \colon \Adj \to \B$ by setting $G(a) = c, G(b) = d(c), G(X) = Z, G(Y) = \eta_c^{-1} \circ d(Z), G(f) = h, G(g) = d(h) \bullet (\tau_c \circ \id) \bullet (\eta_Z \circ \id)$. 
    We need to check that the Zorro moves are satisfied. The first Zorro move is satisfied by \Cref{Equation: first coherence condition of h in volutive adjunction}
    and the second Zorro move is satisfied by \Cref{Equation: second condition of h in volutive adjunction}. Next, we need to construct a 2-contravariant 
    $\Oone$-volutive structure $(\omega,\Xi)$ on $G$. We define an invertible 2-transformation $\omega \colon G^{2\opp} \circ F \to d \circ G$ as follows. 
    Its 1-morphism component at $a \in \Adj$ is given by the identity 1-morphism on $GF(a) = G(b) = d(c) = d(G(a))$ and its 1-morphism component at 
    $b \in \Adj$ is given by 
    \begin{equation}
        \xymatrix{
            GF(b) = G(a) = c \ar[r]^-{\eta_c} & d^2(c) =d G(b)
        }
    \end{equation}
    Its 2-morphism component at $X \colon a \to b$ in $\Adj$ is given by the identity 2-morphism filling the diagram 
    \begin{equation}
        \begin{aligned}
        \xymatrix{
            GF(a) \ar[r]^-{=} \ar[d]_-{GF(X)} & G(b) \ar[r]^-{=} \ar[d]_-{G(Y) }& d(c) \ar[d]_-{\eta_c^{-1} \circ d(Z)} \ar[r]^-{=} \ar[dr]^-{d(Z)} &dG(a) \ar[dr]^-{dG(X)}\\
            GF(b) \ar[r]_-{=} & G(a) \ar[r]_-{=} & c \ar[r]_-{\eta_c} & d^2(c) \ar[r]_-{=} &dG(b)
        }
        \end{aligned}
    \end{equation}
    Its 2-morphism component at $Y \colon b \to a$ in $\Adj$ is given by 
    \begin{equation}
        \begin{aligned}
        \xymatrix{
            GF(b) \ar[r]^-{=} \ar[d]_-{GF(Y)} & G(a) \ar[r]^-{=} \ar[d]_{G(X)} & c \ar[r]^-{\eta_c} \ar[d]_-{Z} & d^2(c) \ar[r]^-{=} \ar[d]^-{d(\eta_c^{-1} \circ d(Z))} \ar@{=>}[dl]^-{\eta_Z} & dG(b) \ar[d]^-{dG(Y)} \\
            GF(a) \ar[r]_-{=} & G(b) \ar[r]_-{=} & d(c) \ar[r]_-{=} & d(c) \ar[r]_-{=} & dG(a)
        }
        \end{aligned}
    \end{equation}
    One checks that this defines an invertible 2-transformation. Next, we define the invertible modification $\Xi$ fitting into the diagram 
    \begin{equation}
        \begin{aligned}
        \xymatrix{
        ddG \ar[r]^-{\id \circ \omega^{-1}} \ar[dr]_-{\eta \circ \id} & dGF \ar[r]^-{\omega^{-1} \circ \id} & GFF \ar[dl]^-{\id}\\
        & G &
        }
        \end{aligned}
    \end{equation}
    Its 2-morphism component at $a \in \Adj$ is given by the identity 2-morphism on $\eta_c^{-1}$. Its 2-morphism component at $b \in \Adj$ is given 
    by filling the following diagram 
    \begin{equation}
        \begin{aligned}
        \xymatrix{
        ddG(b) \ar[d]^-{=} \ar[r]^-{d(\omega_b^{-1})} & dGF(b) \ar[r]^-{\omega_{F(b)}^{-1}} \ar[d]^-{=} \ar@{=>}[dl]^-{\id} & GFF(b) \ar[d]^-{=} \ar@{=>}[dl]^-{\id} \\
        d^3(c) \ar[dr]_-{\eta_{d(c)}^{-1}} \ar[r]_-{d(\eta_c^{-1})} & d(c) \ar[r]_-{=} \ar@{=>}[d]^-{\tau_c^{-1}} & d(c) \ar[dl]^-{\id} \\
        & G(b) = d(c) &
        }
        \end{aligned}
    \end{equation}
    One checks that the triple $(G,\omega,\Xi)$ defines a 2-contravariant $\Oone$-volutive functor 
    \begin{equation}
        (G,\omega,\Xi) \colon (\Adj,F) \to (\B,d,\eta,\tau).
    \end{equation}
    Conversely, let $(G,\omega,\Xi) \colon \Adj \to \B$ be a 2-contravariant $\Oone$-volutive functor. We define a volutive adjunction 
    whose object is $c := G(a)$, whose 1-morphism $Z \colon c \to d(c)$ is given by the composition 
    \begin{equation}
        Z := \left( \xymatrix{ c = G(a) \ar[r]^-{G(X)} & G(b) = GF(a) \ar[r]^-{\omega_a} & dG(a) = d(c) } \right),
    \end{equation}
    and whose 2-morphism $h \colon \id_c \to \eta_c^{-1} \circ d(Z) \circ Z$ is given by the composition
    \begin{equation}
        \begin{aligned}
        \xymatrix@R=0.5em{
            \id_c = \id_{G(a)} = G(\id_a) \ar[r]^-{G(f)} & G(YX) = G(Y) \circ G(X) \ar[r]^-{=} & G(F(X)) \circ G(X)  \\
            \ar[r]_-{\omega_X \circ \id} & \omega_b^{-1}  dG(X)  \omega_a \omega_a^{-1} Z \ar[r]_-{=} & \omega_b^{-1} dG(X) Z \\
            \ar[r]_-{=} & \omega_b^{-1} d(\omega_a^{-1}\omega_a G(X)) Z \ar[r]_-{=} & \omega_b^{-1} d(\omega_a^{-1})d(\omega_a G(X)) Z\\
            \ar[r]_-{=} & \omega_b^{-1} d(\omega_a^{-1}) d(Z) Z \ar[r]_-{\Xi_c \circ \id} & \eta_c d(Z) Z
        }
        \end{aligned}
    \end{equation}
    One checks that the triple $(c,Z,h)$ defines a volutive adjunction.
\end{proof}
\begin{corollary}
    Lax $\Oone$-volutive categories are the same as 2-contravariant $\Oone$-volutive 2-functors $\Adj \to \Cat$. 
\end{corollary}
\begin{proof}
    Combine \Cref{Lemma: volutive adjunctions and volutive functors from the walking adjunction} and \Cref{Lemma: volutive adjunctions and lax volutive categories}.
\end{proof}
All together, we have proved the following. 
\begin{theorem}\label{Theorem: four equivalent versions of lax O(1)-volutive categories}
    The following four notions are equivalent: 
    \begin{itemize}
        \item[(i)] Lax $\Oone$-volutive categories in the sense of \Cref{Definition: lax O(1)-volutive category},
        \item[(ii)] representable symmetric functors in the sense of \Cref{Definition: properties of pairings},
        \item[(iii)] volutive adjunctions in $\Cat$ in the sense of \Cref{Definition: volutive adjunction}, and
        \item[(iv)] 2-contravariant $\Oone$-volutive functors $\Adj \to \Cat$.
    \end{itemize}
\end{theorem}
We leave it to the reader to complete the definitions in the following result, providing various equivalent formulations of $\Oone$-volutive categories.
\begin{corollary}
    The following four notions are equivalent: 
    \begin{itemize}
        \item[(i)] $\Oone$-volutive categories in the sense of \Cref{Definition: O(1)-volutive category},
        \item[(ii)] non-degenerate symmetric functors in the sense of \Cref{Definition: properties of pairings},
        \item[(iii)] volutive adjoint equivalences in $\Cat$,
        \item[(iv)] 2-contravariant $\Oone$-volutive functors $\AdjEqv \to \Cat$.
    \end{itemize}
\end{corollary}

\section{Relative $\Sotwo$-volutive 2-categories}\label{Section: Relative Sotwo volutive structures}
In this section, we introduce relative and lax $\Sotwo$-volutive 2-categories. We start by reviewing the theory of adjoints of 1- and 2-morphisms in a 3-category, 
highlighting in particular that the 3-category of 2-categories with adjoints has 2-adjoints. We then define $\Sotwo$-volutive 2-categories and consider the special 
case of rigid symmetric monoidal 2-categories, in which the $\Sotwo$-volutive structure is given by the Serre isomorphism. Finally, we introduce lax and 
relative $\Sotwo$-volutive 2-categories, and show that the latter notion generalizes the former. 
\subsection{Adjoints in 2- and 3-categories}
\begin{definition}
    A 2-category has \emph{right adjoints} if for each 1-morphism $X$ there exists another 1-morphism $X^R$ together with 2-morphisms 
    $\coev_X \colon \id \to X^R \circ X$ and $\ev_X \colon X \circ X^R \to \id$ satisfying the Zorro/snake/triangle moves. 
    We define left adjoints analogously. We will say that a 2-category has (all) adjoints if it has both right adjoints and left adjoints 
    for all 1-morphisms. 
\end{definition}
\begin{lemma}
    Let $\B$ be a 2-category with right adjoints. The assignment $a \mapsto a, X \mapsto X^R$ extends to a 2-functor $(-)^R \colon \B \to \B^{(1,2)\opp}$.
\end{lemma}
\begin{proof}
    We define the functor $(-)^R$ on 2-morphism level by mapping $f \colon X \to Y$ to 
    \begin{equation}
        \xymatrix@C=3.5em{
        Y^R \ar[r]^-{\coev_X \circ \id} & X^R \circ X \circ Y^R \ar[r]^-{\id \circ f \circ \id} & X^R \circ Y \circ Y^R \ar[r]^-{\id \circ \ev_Y} & X^R.
        }
    \end{equation}
    One checks that this defines a 2-functor. 
\end{proof}
\begin{lemma}\label{Lemma: compatibility of any 2-functor with adjoints}
    Let $\B,\B'$ be 2-categories with right adjoints and let $F \colon \B \to \B'$ a functor. Then, there exists an invertible 2-transformation 
    \begin{equation}
        \xi_F \colon (-)^R_{\B'} \circ F \to F^{(1,2)\opp} \circ (-)^R_{\B}.
    \end{equation}
    Let furthermore $G \colon \B \to \B'$ be another functor and let $\alpha \colon F \to G$ be a 2-transformation. 
    Then, there exists a modification $\xi_{\alpha}$ fitting into the diagram 
    \begin{equation}
        \begin{aligned}
        \xymatrix{
            (-)^R_{\B'} \circ F \ar[d]^-{\id \circ \alpha} \ar[r]^-{\xi_F} & F^{(1,2)\opp} \circ (-)^R_{\B}\\
            (-)^R_{\B'} \circ G \ar[r]_-{\xi_G}\ar@{=>}[ur]_-{\xi_{\alpha}}  & G^{(1,2)\opp} \circ (-)^R_{\B} \ar[u]_-{\alpha^{(1,2)\opp} \circ \id}.
        }
        \end{aligned}
    \end{equation}
    The 2-modifications $\xi_{\alpha}$ are moreover natural and compatible with respect to composition and units. 
\end{lemma}
\begin{proof}
    First, we construct the 2-transformation $\xi_F$. Recalling that $(-)^R$ is the identity on objects, we may take the 1-morphism components of $\xi_F$ to be 
    identities. The 2-morphism component of $\xi_F$ at $X \colon a \to b$ is the (unique) isomorphism $F(X^R) \to F(X)^R$ expressing the fact that $F(X^R)$ is a right adjoint of $F(X)$, 
    together with the uniqueness of adjoints. More precisely, the 2-morphisms  
    \begin{equation}
        \begin{aligned}
        \xymatrix@R=0.5em{
            F(X) \circ F(X^R) \ar[r]^-{\cong} & F(X \circ X^R) \ar[r]^-{F(\ev_X)} & F(\id) \cong \id \\
            \id \cong F(\id) \ar[r]^-{F(\coev_X)} & F(X^R \circ X) \ar[r]^-{\cong} & F(X^R) \circ F(X)
        }
        \end{aligned}
    \end{equation} 
    exhibit $F(X^R)$ as a right adjoint of $F(X)$, so that we have a 2-isomorphism 
    \begin{equation}
        \begin{aligned}
        \xymatrix@C=4.5em{
            (\xi_F)_X : F(X^R) \ar[r]^-{\id \circ \coev_{F(X)}} & F(X)^R \circ F(X) \circ F(X^R) \ar[r]^-{\id \circ F(\ev_X)} & F(X)^R;
        }
        \end{aligned}
    \end{equation}
    one checks that this defines an invertible 2-transformation $\xi_F$. Next, we construct the modification $\xi_{\alpha}$. Its 
    2-morphism component at $c \in \B$ is given by
    \begin{equation}
        \xymatrix{
        (\xi_\alpha)_c := \ev_{\alpha_c} \colon \alpha_c \circ \alpha_c^R \to \id;
        }
    \end{equation}
    one checks that this defines a modification and that these modifications are natural and compatible with composition and units in $\alpha$.
\end{proof}
\begin{lemma}\label{Lemma: Functor category and adjoints}
    Let $\B$ and $\B'$ be 2-categories with adjoints. Then $\Fun(\B,\B')$ has adjoints.
\end{lemma}
\begin{proof}
    Let $\alpha \colon F \to G$ be a 2-transformation. We consider the case of right adjoints. We may define a 2-transformation $\alpha^R \colon G \to F$ with 
    1-morphism components $(\alpha^R)_a = \alpha_a^R \colon G(a) \to F(a)$ and 2-morphism components $(\alpha^R)_X := \alpha_{X^L}^R \colon F(X) \circ \alpha_a^R 
    \to \alpha_b^R \circ G(X)$. The adjunction data $\ev_{\alpha},\coev_{\alpha}$ is then given componentwise by $(\ev_{\alpha})_a := \ev_{\alpha_a} \colon \alpha_a 
    \circ \alpha_a^R \to \id$ and $(\coev_{\alpha})_a := \coev_{\alpha_a} \colon \id \to \alpha^R_a \circ \alpha_a$. The case of left/all adjoints is completely 
    analogous.
\end{proof}
\begin{definition}
    Let $\D$ be a 3-category. A 2-morphism $f \colon X \to Y$ in $\D$ has a right/left adjoint if it has a right/left adjoint as a 1-morphism in the 2-category 
    $\Hom_{\D}(X,Y)$. We will say that $\D$ has left/right/(all) adjoints for 2-morphisms if each 2-morphism has left/right/(both) adjoints. We will say that 
    $\D$ has left/right/(all) adjoints for 1-morphisms if its homotopy 2-category $\hh_2 \D$ has left/right/(all) adjoints. We will say that $\D$ has (all) 
    adjoints if it has left and right adjoints for both 1- and 2-morphisms. 
\end{definition}
\begin{definition}
    We denote by $2\hspace{-0.05cm}\Cat^{\operatorname{adj}}$ the 3-category of 2-categories with (all) adjoints, 2-functors, 2-transformations, and modifications.
\end{definition}
\begin{corollary}
    The 3-category $2\hspace{-0.05cm}\Cat^{\operatorname{adj}}$ has adjoints for 2-morphisms. In particular, we have 3-functors 
    \begin{equation}
        (-)^{1\opp} \colon 2\hspace{-0.05cm}\Cat^{\operatorname{adj}} \to (2\hspace{-0.05cm}\Cat^{\operatorname{adj}})^{2\opp}, 
        (-)^{2\opp} \colon 2\hspace{-0.05cm}\Cat^{\operatorname{adj}} \to (2\hspace{-0.05cm}\Cat^{\operatorname{adj}})^{3\opp}
    \end{equation}
    and 3-functors 
    \begin{equation}
        (-)^L, (-)^R \colon 2\hspace{-0.05cm}\Cat^{\operatorname{adj}} \to (2\hspace{-0.05cm}\Cat^{\operatorname{adj}})^{(2,3)\opp}.
    \end{equation}
\end{corollary}
\begin{remark}
    We wish to understand the interaction of the functors $(-)^R,(-)^L$ and $(-)^{1\opp},(-)^{2\opp}$. We claim that $(-)^{1\opp} \circ (-)^R \cong (-)^L \circ (-)^{1\opp}$. 
    The functor $(-)^{1\opp} \circ (-)^R$ in particular assigns each 2-category with adjoints $\B$ to its 1-opposite $\B^{1\opp}$, each 2-functor $F \colon \B \to \B'$ 
    to its 1-opposite $F^{\opp} \colon \B^{1\opp} \to (\B')^{1\opp}$, and each 2-transformation $\alpha \colon F \to G$ to 
    $(\alpha^R)^{1\opp} \colon F^{1\opp} \to (\G')^{1\opp}$. The 1-morphism component of $(\alpha^R)^{1\opp}$ at $a \in \B^{1\opp}$ is given by 
    $\alpha^R_a \colon F(a) \to G(a)$ in $(\B')^{1\opp}$, that is, $\alpha_a^R \colon G(a) \to F(a)$ in $\B'$. We claim that $(\alpha^R)^{1\opp}$ is 
    a left adjoint of the 2-morphism $\alpha^{1\opp}$. Indeed, adjunction data is given by 
    \begin{equation}
        \begin{aligned}
        \xymatrix@R=0.5em{
            (\alpha^R)^{1\opp} \circ \alpha^{1\opp} \ar[r]^-{\cong} & (\alpha \circ \alpha^R)^{1\opp} \ar[r]^-{\ev_{\alpha}^{1\opp}} & \id^{1\opp} \cong \id \\
            \id \cong \id^{1\opp} \ar[r]^-{\coev_{\alpha}^{1\opp}} & (\alpha^R \circ \alpha)^{1\opp} \ar[r]^-{\cong} & \alpha^{1\opp} \circ (\alpha^R)^{1\opp}
        }
        \end{aligned}
    \end{equation}
    Hence, we obtain an invertible 3-transformation 
    \begin{equation}
        (-)^{1\opp} \circ (-)^R \cong (-)^L \circ (-)^{1\opp}.
    \end{equation}
    We also note that we have invertible 3-transformations 
    \begin{equation}
        (-)^{2\opp} \circ (-)^R \cong (-)^R \circ (-)^{2\opp}, \qquad (-)^{2\opp} \circ (-)^L \cong (-)^L \circ (-)^{2\opp},
    \end{equation}
    by a similar argument keeping tracking of the contravariances.
\end{remark}
\begin{lemma}\label{Lemma: Trivializing the double adjoint on 2-categories with adjoints}
    There exists an invertible 3-transformation 
    \begin{equation}
        S \colon \id_{2\hspace{-0.05cm}\Cat^{\operatorname{adj}}} \to (-)^{RR}.
    \end{equation}
\end{lemma}
\begin{proof}
    We construct the 3-transformation $S \colon \id_{2\hspace{-0.05cm}\Cat^{\operatorname{adj}}} \to (-)^{RR}$. Recalling that both functors are the identity 
    on objects, we may set the 1-morphism component of $S$ at a 2-category with adjoints $\B$ to be the double right adjoint functor 
    $S_{\B} := (-)^{RR}_{\B} \colon \B \to \B$. Recalling that both functors are the identity on 1-morphism level, we may set the 2-morphism component 
    of $S$ at a 2-functor $F \colon \B \to \B'$ to be the 2-transformation 
    \begin{equation}
        \begin{aligned}
        \xymatrix@C=5em{
            (-)^{RR}_{\B'} \circ F \ar[r]^-{\id \circ \xi_F} & ((-)^R_{\B'})^{(1,2)\opp} \circ F^{(1,2)\opp} \circ (-)^R_{\B} \ar[r]^-{(\xi_F^{-1})^{(1,2)\opp} \circ \id} & F \circ (-)^{RR}_{\B}
        }
        \end{aligned}
    \end{equation}
    where $\xi_F$ is the 2-transformation described in \Cref{Lemma: compatibility of any 2-functor with adjoints}. Next, we wish to describe the 3-morphism component 
    of $S$ at a 2-transformation $\alpha \colon F \to G$. First, we note that there are non-invertible modifications filling the two squares in the following diagram:
    \begin{equation}
        \begin{aligned}
        \xymatrix@C=5em{
            (-)^{RR}_{\B'} \circ F \ar[r]^-{\id \circ \xi_F} \ar[d]^-{\id \circ \alpha} & ((-)^R_{\B'})^{(1,2)\opp} \circ F^{(1,2)\opp} \circ (-)^R_{\B} \ar[r]^-{(\xi_F^{-1})^{(1,2)\opp} \circ \id} & F \circ (-)^{RR}_{\B} \ar[d]^-{\alpha^{RR} \circ \id} \\
            (-)^{RR}_{\B'} \circ G \ar[r]^-{\id \circ \xi_G} & ((-)^R_{\B'})^{(1,2)\opp} \circ G^{(1,2)\opp} \circ (-)^R_{\B} \ar[r]^-{(\xi_G^{-1})^{(1,2)\opp} \circ \id} \ar[u]^-{\id \circ \alpha^{(1,2)\opp} \circ \id} & G \circ (-)^{RR}_{\B}.
        }
        \end{aligned}
    \end{equation}
    Here, the left square is filled with the modification $\xi_\alpha$ described in \Cref{Lemma: compatibility of any 2-functor with adjoints};
    the right square may be filled with the modification 
    \begin{equation}
        (\xi_{\alpha^R}^L)^{\opp} \colon (\alpha^{RR} \circ \id) \bullet ((\xi_F^{-1})^{(1,2)\opp} \circ \id) \bullet (\id \circ \alpha^{(1,2)\opp}) \Rightarrow 
        ((\xi_G^{-1})^{(1,2)\opp} \circ \id).
    \end{equation}
    However, we can not immediately paste these modifications together, as they and $\alpha$ are generically non-invertible. The trick we may employ is that we have 
    honest invertible 2-modifications in the diagram 
    \begin{equation}
        \begin{aligned}
        \xymatrix@C=5em{
            (-)^{RR}_{\B'} \circ F \ar[r]^-{\id \circ \xi_F} \ar[d]^-{\id \circ \alpha} & ((-)^R_{\B'})^{(1,2)\opp} \circ F^{(1,2)\opp} \circ (-)^R_{\B} \ar[r]^-{(\xi_F^{-1})^{(1,2)\opp} \circ \id} \ar[d]^-{\id \circ (\alpha^R)^{(1,2)\opp} \circ \id} \ar@{=>}[dl]^-{\cong} & F \circ (-)^{RR}_{\B} \ar[d]^-{\alpha^{RR} \circ \id} \ar@{=>}[dl]^-{\cong} \\
            (-)^{RR}_{\B'} \circ G \ar[r]_-{\id \circ \xi_G} & ((-)^R_{\B'})^{(1,2)\opp} \circ G^{(1,2)\opp} \circ (-)^R_{\B} \ar[r]_-{(\xi_G^{-1})^{(1,2)\opp} \circ \id} & G \circ (-)^{RR}_{\B}.
        }
        \end{aligned}
    \end{equation}
    instead, whose pasting yields the 3-morphism component $S_\alpha$. It is straightforward to check that this defines an invertible 3-transformations, observing 
    that each component is evidently invertible. 
\end{proof}
\begin{corollary}\label{Corollary: The identity automorphism on 2-categories with adjoints}
    We obtain an invertible 3-transformation 
    \begin{equation}
        S \colon \id_{2\hspace{-0.05cm}\Cat^{\operatorname{adj}}_{(3,1)}} \to \id_{2\hspace{-0.05cm}\Cat^{\operatorname{adj}}_{(3,1)}}
    \end{equation}
    on the underlying (3,1)-category of $2\hspace{-0.05cm}\Cat^{\operatorname{adj}}$.
\end{corollary}

\subsection{$\Sotwo$-volutive 2-categories}
We start with a general discussion of $\Sotwo$-actions on 3-categories. 
\begin{remark}\label{Remark: Sotwo actions on 3-categories}
First, recall the low-dimensional homotopy groups of $\Sotwo$, 
\begin{equation}
    \pi_0(\Sotwo) = \star, \pi_1(\Sotwo) = \mathbb{Z}, \pi_2(\Sotwo) = \pi_3(\Sotwo) = \dotsc = \star.
\end{equation}
In particular, giving an $\Sotwo$-action on a 3-category $\D$ amounts to specifiying a 4-functor 
\begin{equation}
    \mathrm{B}\Pi_3(\Sotwo) \to 3\hspace{-0.05cm}\Cat, \ \star \mapsto \D.
\end{equation}
which equivalently may be understood as a braided monoidal 2-functor 
\begin{equation}
    \Z \to \Aut(\id_{\D})
\end{equation}
where $\Aut(\id_{\D})$ denotes the braided monoidal 2-category of autoequivalences of the identity 3-functor on $\D$. 
\end{remark}
For the convenience of the reader, we will include some details of the notion of a braided monidal 2-functor in the following. 
\begin{remark}
    Let $\B,\C$ be braided monoidal 2-categories. In the following, we wish to spell out the data of a braided monoidal 2-functor $F \colon \B \to \C$, 
    omitting the coherence conditions on the highest level. First, we have a 2-functor $F \colon \B \to \C$. Second, we have a monoidal structure on $F$, 
    which amounts to the following data: (i) for any two objects $a,b \in \B$ 1-isomorphisms $\varphi_{a,b} \colon F(a) \otimes F(b) \to F(a \otimes b)$, and a 
    1-isomorphism $\epsilon \colon 1_{\C} \to F(1_{\B})$ (ii) for any two 1-morphisms $X \colon a \to a'$ and $Y \colon b \to b'$, 2-isomorphisms $\varphi_{X,Y}$ 
    fitting into the diagram 
    \begin{equation}
        \begin{aligned}
        \xymatrix{
            F(a) \otimes F(b) \ar[rr]^-{\varphi_{a,b}} \ar[d]_-{F(X) \otimes F(Y)} && F(a \otimes b) \ar[d]^-{F(X \otimes Y)} \ar@{=>}[dll]^-{\varphi_{X,Y}}\\
            F(a') \otimes F(b') \ar[rr]_-{\varphi_{a',b'}} && F(a' \otimes b')
        }
        \end{aligned}
    \end{equation}
    (iii) for each triple of objects $a,b,c \in \B$ a 2-isomorphism $\phi_{a,b,c}$ fitting into the diagram 
    \begin{equation}
        \begin{aligned}
        \xymatrix{
            (F(a) \otimes F(b)) \otimes F(c) \ar[rr]^-{} \ar[d]_-{\varphi_{a,b} \otimes \id} && F(a) \otimes (F(b) \otimes F(c)) \ar[d]^-{\id \otimes \varphi_{b,c}} \ar@{=>}[ddll]^-{\phi_{a,b,c}}  \\
            F(a \otimes b) \otimes F(c) \ar[d]_-{\varphi_{a \otimes b,c}} && F(a) \otimes F(b \otimes c) \ar[d]^-{\varphi_{a, b \otimes c}}\\
            F((a \otimes b) \otimes c) \ar[rr]_-{} && F(a \otimes (b \otimes c))
        }
        \end{aligned}
    \end{equation}
    (iv) for each object $a \in \B$, 2-isomorphisms $\xi_a,\zeta_a$ fitting into the diagrams 
    \begin{equation}
        \begin{aligned}
        \xymatrix{
            1_{\C} \otimes F(a) \ar[r]^-{\epsilon \otimes \id} \ar[d]^-{} & F(1_{\B}) \otimes F(a) \ar[d]^-{\varphi_{1_{\B},a}} \ar@{=>}[dl]^-{\xi_a}\\
            F(a) & F(1_{\B} \otimes a) \ar[l] \\ 
        }
        \qquad 
        \xymatrix{
            F(a) \otimes 1_{\C} \ar[r]^-{\id \otimes \epsilon} \ar[d]^-{} & F(a) \otimes F(1_{\B}) \ar[d]^-{\varphi_{a,1_{\B}}} \ar@{=>}[dl]^-{\zeta_a} \\
            F(a) & F(a \otimes 1_{\B}) \ar[l]
        }
        \end{aligned}
    \end{equation}
    where we have omitted the symbols for the associators and unitors in (iii) and (iv). The 2-isomorphisms $\varphi_{X,Y},\phi_{a,b,c},\xi_a,\zeta_a$ are subject 
    to further coherence conditions we do not spell out here. \\
    Third, we have a braided structure on $F$, which amounts to the following data: for any two objects $a,b \in \B$, 2-isomorphisms $\mu_{a,b}$ 
    fitting into the diagram 
    \begin{equation}\label{Equation: diagram for braided monoidal 2-functor}
        \begin{aligned}
        \xymatrix@C=4em{
            F(a) \otimes F(b) \ar[r]^-{\beta_{F(a),F(b)}^{\C}} \ar[d]_-{\varphi_{a,b}} & F(b) \otimes F(a) \ar[d]^-{\varphi_{b,a}} \ar@{=>}[dl]^-{\mu_{a,b}} \\
            F(a \otimes b) \ar[r]_-{F(\beta_{a,b}^{\B})} & F(b \otimes a) 
        }
        \end{aligned}
    \end{equation}
    which are subject to further coherence conditions we again do not spell out here.
\end{remark}
\begin{remark}\label{Remark: spelling out the braided monoidal structure on automorphisms of the identity}
    Let $\C$ be a 3-category. In the following, we wish to describe the braided monoidal structure on the 2-category $\Aut(\id_{\C})$ in some detail. 
    By definition, the objects of $\Aut(\id_{\C})$ are invertible 3-transformations $\alpha \colon \id_{\C} \to \id_{\C}$; the 1-morphisms are 
    3-modifications, and the 2-morphisms are pertubations between them. The monoidal structure can be described as follows: given two 
    invertible 3-transformations, $\alpha, \gamma \colon \id_{\C} \to \id_{\C}$, we may consider their composition, $\alpha \bullet \gamma$ with respect to which the 
    the identity 3-transformation $\id_{\id_{\C}}$ acts as a monoidal unit. The braided structure is given as follows. We need to supply a 3-modification 
    \begin{equation}
        \beta_{\alpha,\gamma} \colon \alpha \bullet \gamma \to \gamma \bullet \alpha;
    \end{equation}
    first, given an object $a \in \C$, we need to supply the 2-morphism component 
    \begin{equation}
        (\beta_{\alpha,\gamma})_a \colon \alpha_a \circ \gamma_a = (\alpha \bullet \gamma)_a \to (\gamma \bullet \alpha)_a = \gamma_a \circ \alpha_a.
    \end{equation}
    Recall that $\alpha$ is an invertible 3-transformation, that is, it comes in particular with invertible 1-morphism components $\alpha_a \colon a \to a$ and 
    2-morphism components $\alpha_X \colon X \circ \alpha_a \to \alpha_b \circ X$, for $X \colon a \to b$ a 1-morphism in $\C$. In particular, we may choose 
    \begin{equation}
        (\beta_{\alpha,\gamma})_a := \alpha_{\gamma_a}^{-1} \colon \alpha_a \circ \gamma_a \to \gamma_a \circ \alpha_a.
    \end{equation}
    Second, given a 1-morphism $X \colon a \to b$ in $\C$, we need to supply the 3-morphism component $(\beta_{\alpha,\gamma})_X$. For this, we recall 
    that $\alpha$ also comes with an invertible 3-morphism component at $f \colon X \to Y$ of the form 
    \begin{equation}\label{Equation: naturality 3-isomorphisms of 3-transformation}
        \begin{aligned}
        \xymatrix{
            X \circ \alpha_a \ar[r]^-{\alpha_X} \ar[d]_-{f \circ \id} & \alpha_b \circ X \ar[d]^-{\id \circ f} \ar@{=>}[dl]^-{\alpha_f} \\
            Y \circ \alpha_a \ar[r]_-{\alpha_Y} & \alpha_b \circ Y.
        }
        \end{aligned}
    \end{equation}
    and for each pair of composable 1-morphisms $X \colon a \to b, Z \colon b \to C$ an invertible 3-morphism $\alpha_{X,Z}$ fitting into the diagram 
    \begin{equation}\label{Equation: compositional 3-isomorphisms of 3-transformation}
        \begin{aligned}
        \xymatrix{
            Z \circ (X \circ \alpha_a) \ar[r]^-{\id \circ \alpha_X} \ar[d] & Z \circ \alpha_b \circ X \ar[d]^-{\alpha_Z \circ \id} \ar@{=>}[dl]^-{\alpha_{X,Z}} \\
            (Z \circ X) \circ \alpha_a \ar[r]^-{\alpha_{Z \circ X}} & \alpha_c \circ Z \circ X 
        }
        \end{aligned}
    \end{equation}
    In particular, we have invertible 3-morphisms filling the diagram 
    \begin{equation}
        \begin{aligned}
        \xymatrix@C=4.5em@R=2.5em{
            X \circ \alpha_a \circ \gamma_a \ar[r]^-{\alpha_X \circ \id} \ar[d]_-{\id \circ \alpha_{\gamma_a}^{-1}} & \alpha_b \circ X \circ \gamma_a \ar[r]^-{\id \circ \gamma_X} & \alpha_b \circ \gamma_b \circ X \ar[d]^-{\alpha_{\gamma_b}^{-1} \circ \id} \\
            X \circ \gamma_a \circ \alpha_a \ar[ur]^-{\alpha_{X \circ \gamma_a}} \ar[r]_-{\gamma_X \circ \id} & \gamma_b \circ X \circ \alpha_a \ar[r]_-{\id \circ \alpha_X} \ar[ur]^-{\alpha_{\gamma_b \circ X}} & \gamma_b \circ \alpha_b \circ X
        }
        \end{aligned}
    \end{equation}
    where the left and right triangle are filled with the 3-isomorphisms from \Cref{Equation: compositional 3-isomorphisms of 3-transformation} 
    and the middle square is filled with the 3-isomorphisms from \Cref{Equation: naturality 3-isomorphisms of 3-transformation}; more specifically, 
    from left to right we have $\alpha_{\gamma_a,X}, \alpha_{\gamma_X},$ and $\alpha_{X,\gamma_b}^{-1}$. We have now described the data of the 3-transformation 
    $\beta_{\alpha,\gamma} \colon \alpha \bullet \gamma \to \gamma \bullet \alpha$ (omitting the various coherence conditions).
\end{remark}
Having established the data of braided monoidal 2-functors and the braided monoidal structure on $\Aut(\id_{\C})$ for some 3-category $\C$, we now turn 
our attention to a particular $\Sotwo$-action. 
\begin{construction}\label{Construction: The defining Sotwo action of volutive categories}
    In the following, we wish to construct an $\Sotwo$-action on the (3,1)-category $2\hspace{-0.05cm}\Cat^{\operatorname{adj}}_{(3,1)}$. 
    By our discussion in \Cref{Remark: Sotwo actions on 3-categories}, we need to supply a braided monoidal 2-functor 
    \begin{equation}
        \rho \colon \mathbb{Z} \to \Aut(\id_{2\hspace{-0.05cm}\Cat_{(3,1)}^{\operatorname{adj}}}).
    \end{equation}
    We start by defining the underlying 2-functor. It suffices to describe the image $S$ of $1 \in \mathbb{Z}$ under this 2-functor, as monoidality will 
    require us to send $k \mapsto S^k$. Hence, we need to specify a 3-transformation 
    \begin{equation}
        S \colon \id_{2\hspace{-0.05cm}\Cat_{(3,1)}^{\operatorname{adj}}} \to \id_{2\hspace{-0.05cm}\Cat_{(3,1)}^{\operatorname{adj}}}.
    \end{equation}
    We choose the 3-transformation described in \Cref{Corollary: The identity automorphism on 2-categories with adjoints}, whose components we will repeat here 
    for the convenience of the reader. First, its 1-morphism component at $\B \in 2\hspace{-0.05cm}\Cat^{\operatorname{adj}}_{(3,1)}$ is given by
    \begin{equation}
        S_{\B} := (-)^{RR}_{\B} \colon \B \to \B,
    \end{equation}
    the double right adjoint functor on $\B$ (whose inverse is given by the double left adjoint functor on $\B$). Second, its 2-morphism component 
    at $F \colon \B \to \B'$ in $2\hspace{-0.05cm}\Cat^{\operatorname{adj}}_{(3,1)}$ is given by the canonical 2-transformation 
    \begin{equation}
        \xymatrix{
            S_F \colon F \circ S_{\B} = F \circ (-)^{RR}_{\B} \ar[r]^-{\cong} & (-)^{RR}_{\B'} \circ F = S_{\B'} \circ F
        }
    \end{equation}
    expressing compatibility of any 2-functor with the (double) right adjoint functor. Third, its 3-morphism component at $\alpha \colon F \to G$ in 
    $2\hspace{-0.05cm}\Cat^{\operatorname{adj}}_{(3,1)}$ is given by the canonical modification $S_{\alpha}$ fitting into the diagram 
    \begin{equation}
        \begin{aligned}
        \xymatrix{
            F \circ (-)^{RR}_{\B} \ar[r]^-{\cong} \ar[d]^-{\alpha \circ \id} & (-)^{RR}_{\B'} \circ F \ar[d]^-{\id \circ \alpha} \\
            G \circ (-)^{RR}_{\B} \ar[r]^-{\cong} & (-)^{RR}_{\B'} \circ G
        }
        \end{aligned}
    \end{equation}
    using the fact that (invertible) 2-transformations are automatically compatible with (double) adjoints in precisely this sense.
    One checks that this defines a 3-transformation $S \colon  \id_{2\hspace{-0.05cm}\Cat_{(3,1)}^{\operatorname{adj}}} \to \id_{2\hspace{-0.05cm}\Cat_{(3,1)}^{\operatorname{adj}}}$, so that 
    we obtain a 2-functor $\rho \colon \mathbb{Z} \to \Aut(\id_{2\hspace{-0.05cm}\Cat_{(3,1)}^{\operatorname{adj}}})$. The data of the monoidal structure on $\rho$ 
    can essentially be taken to be trivial, since we define $\rho(k) := S^k$, for which e.g. $\rho(n) \circ \rho(m) = S^n \circ S^m \cong S^{n+m} = \rho(n+m)$ 
    is essentially trivial. The braided structure on the other hand is non-trivial. Making use of the hexagon axioms and noting that the braided structure 
    on $\mathbb{Z}$ is essentially trivial, it suffices to consider the component $\beta_{-1,1}$ of the braiding. More precisely, according to 
    \Cref{Equation: diagram for braided monoidal 2-functor} we need to provide a pertubation trivializing the 3-modification obtained as the composition
    \begin{equation}
        \xymatrix{
            q(S) := \id \ar[r]^-{\cong} & S^{-1} \circ S \ar[r]^-{\beta_{S^{-1},S}} & S \circ S^{-1} \ar[r]^-{\cong} & \id.
        }
    \end{equation}
    We construct an invertible pertubation $\sigma \colon q(S) \to \id$ as follows. First, we recall that the 2-morphism component of $q(S)$ at an object 
    $\B \in 2\hspace{-0.05cm}\Cat_{(3,1)}^{\operatorname{adj}}$ is given by the composition 
    \begin{equation}
        \xymatrix{
            \id_B \ar[r]^-{\cong} & (-)^{LL}_{\B} \circ (-)^{RR}_{\B} \ar[r]^-{} & (-)^{RR}_{\B} \circ (-)^{LL}_{\B} \ar[r]^-{\cong} & \id
        }
    \end{equation}
    where the middle arrow is given by the 2-transformation expressing compatibility of the double left adjoint functor with the (double) right adjoint functor. 
    The 3-morphism component of the invertible pertubation $\sigma$ is then constructed from the Zorro 3-isomorphisms expressing the adjunctions. In particular, 
    we note that the composition 
    \begin{equation}
        \xymatrix{
            X \ar[r]^-{\cong} & (X^L)^R \ar[r]^-{\cong} & (X^R)^L \ar[r]^-{\cong} & X,
        }
    \end{equation}
    whose constituents are the various 2-isomorphisms expressing uniqueness of adjoints, is isomorphic to the identity 2-morphism on $X$. A slightly more 
    sophisticated version of this observation (involving the various double adjoints) allows us to construct the component of $\sigma$. In conclusion, 
    we obtain an invertible pertubation $\sigma \colon q(S) \to \id$. The higher coherences of the braided structure then follow from the coherences encoded 
    in the various adjunctions. 
\end{construction}
\begin{definition}
    The (3,1)-category of \emph{$\Sotwo$-volutive 2-categories} $2\hspace{-0.05cm}\Cat^{\Sotwo\text{-vol}}_{(3,1)}$ is the {(3,1)}-category of homotopy fixed points 
    with respect to the $\Sotwo$-action $\rho$ on $2\hspace{-0.05cm}\Cat_{(3,1)}^{\operatorname{adj}}$ constructed in 
    \Cref{Construction: The defining Sotwo action of volutive categories}.
\end{definition}
\begin{remark}
    We wish to spell out the data of an $\Sotwo$-volutive 2-category. An $\Sotwo$-volutive 2-category is a 2-category $\B$ with adjoints together with an 
    invertible 2-transformation $S \colon \id \to (-)^{RR}$ such that the composed modification 
    \begin{equation}\label{Diagram: SO2volution}
        \begin{aligned}
        \xymatrix{
        \operatorname{id}_{\mathcal{B}} \ar@/_4pc/[dd]_-{\id}  \ar[r]^-{\id} \ar[d] & \operatorname{id}_{\mathcal{B}} \ar[d] \ar@{=>}[dl]^-{\sigma_{\B}} \ar@/^4pc/[dd]^-{\id} \\
        (-)^{LL} \circ (-)^{RR} \ar[r] \ar[d]_{S \circ S^{-1}} &  (-)^{RR} \circ (-)^{LL} \ar[d]^-{S^{-1} \circ S} \ar@{=>}[dl]^-{\beta_{S,S^{-1}}}\\
        \operatorname{id}_{\mathcal{B}} \ar[r]_-{\id} & \operatorname{id}_{\mathcal{B}} 
        }
        \end{aligned}
    \end{equation}
    is the identity on $\operatorname{id}_{\operatorname{id}_{\mathcal{B}}}$; here, the upper square is filled with the modification $\sigma_{\B}$ described in 
    \Cref{Construction: The defining Sotwo action of volutive categories}, the lower square is filled with the modification expressing naturality 
    of the braiding $\beta$ on $\Aut_{\id_{2\hspace{-0.05cm}\Cat_{(3,1)}^{\operatorname{adj}}}}$, and the left and right diagrams are filled with the modifications expressing
    naturality of invertibility/adjunction data. \\
    We also wish to describe functors of $\Sotwo$-volutive 2-categories. Let $(\B,S)$ and $(\B',S')$ be $\Sotwo$-volutive 2-categories. A functor of 
    $\Sotwo$-volutive 2-categories is a 2-functor $F \colon \B \to \B'$ together with an invertible modification 
    \begin{equation}
        \begin{aligned}
            \eta \colon 
            \left(
            \begin{matrix}
            \xymatrix{
                & \ar@{}[d]^(.3){}="a"^(1){}="b" \ar@{=>}^{S} "a";"b" &  \\
                B \ar[rr]_-{(-)^{RR}} \ar[d]_-{F} \ar@/^2pc/[rr]^-{\id_{\B}} && B \ar[d]^-{F} \ar@{=>}[dll]^-{\cong} 
                \\
                B' \ar[rr]_-{(-)^{RR}} && B'
            }
            \end{matrix}\right)
            \;\RRightarrow\;
            \left(
            \begin{matrix}
                \xymatrix{
                B \ar[rr]^-{\id_{\B}} \ar[d]_-{F} && B \ar[d]^-{F} \ar[d]^-{F} \ar@{=>}[dll]_-{\cong}\\
                B' \ar[rr]^-{\id_{\B'}} \ar@/_2pc/[rr]_-{(-)^{RR}} & \ar@{}[d]^(0){}="a"^(0.7){}="b" \ar@{=>}^{S'} "a";"b" & B'\\
                & & 
            }
            \end{matrix}
            \right)
        \end{aligned}
    \end{equation}
    where the square in the left diagram is filled by the 2-transformation expressing compatibility of the 2-functor $F$ with (double) adjoints; we refer to 
    \cite{carqueville2025orbifoldshigherdaggerstructures} for the higher morphisms in the 3-category $2\hspace{-0.05cm}\Cat_{(3,1)}^{\Sotwo\text{-vol}}$.
\end{remark}
\begin{remark}
    The 3-category of $\Sotwo$-volutive 2-category has been discussed in \cite[Section 2.2]{carqueville2025orbifoldshigherdaggerstructures}
    with the exception of the condition presented in \Cref{Diagram: SO2volution}, which had only been hinted at there. We stress that the notion of an 
    $\Sotwo$-volutive structure there is incomplete and misses this important condition. 
\end{remark}
\begin{remark}
    The condition in \Cref{Diagram: SO2volution} asserts that the \emph{twisted quantum dimension} of $S$ is trivial. The terminology is motivated by the fact 
    that an $\Sotwo$-action on a 2-category $\B$ is the same as an invertible 2-transformation $P \colon \id_{\B} \to \id_{\B}$ whose quantum dimension 
    $q(P)$ is trivial. The quantum dimension $q(P)$ is to be understood as the quantum dimension in the category $\Aut_{\id_{\B}}$ whose braided monoidal 
    structure is described in \Cref{Remark: spelling out the braided monoidal structure on automorphisms of the identity}. Explicity, for any object 
    $P \in \Aut_{\id_{\B}}$, the quantum dimension is the composition 
    \begin{equation}
        \xymatrix{
            \id \ar[r]^-{\cong} & P^{-1} \circ P \ar[r]^-{\beta_{P^{-1},P}} & P \circ P^{-1} \ar[r]^-{\cong} & \id 
        }
    \end{equation}
    where $\beta$ is the braiding in $\Aut_{\id_{\B}}$. In the graphical calculus notation, the quantum dimension is also called the figure eight diagram. 
\end{remark}
\begin{remark}
    Let $(\B,S)$ be an $\Sotwo$-volutive 2-category. Then, the maximal subgroupoid $\B^\times \subseteq \B$ inherits an $\Sotwo$-action. Conversely, any 
    2-groupoid with $\Sotwo$-action defines an $\Sotwo$-volutive 2-category. 
\end{remark}
The proof of the following result will occupy our attention for the remainder of this section. 
\begin{proposition}\label{Proposition: rigid symmetric monoidal 2-categories have Sotwo volutive structures}
    Every rigid symmetric monoidal 2-category admits an $\Sotwo$-volutive structure.
\end{proposition}
\begin{lemma}
    Let $\B$ be a rigid symmetric monoidal 2-category. Then, there are invertible 2-transformations 
    \begin{equation}
        \mathcal{P} \colon \id \to (-)^{**} \qquad \text{and} \qquad \mathcal{Q} \colon (-)^{**} \to (-)^{LL}.
    \end{equation}
    where $(-)^*$ denotes the functor induced by duality data (for objects) and $(-)^L$ denotes the functor induced by left adjunction data (for 1-morphisms). 
\end{lemma}
\begin{proof}
    The 2-transformation $\mathcal{P}$ is constructed from the symmetric braiding similar to the 1-categorical case. In order to construct the invertible 
    2-transformation $\mathcal{Q}$, we will work in slightly more generality. Let $\C$ be any monoidal 2-category with duals and adjoints. We will construct 
    an invertible 2-transformation $\mathcal{Q}' \colon {}^{*}\!(-) \to (-)^{*} \circ (-)^{RR}$, which gives rise to a 2-transformation 
    ${}^{**}\!(-) \to (-)^{RR}$ which yields $\mathcal{Q}$ by inverting both sides in the case of $\C = \B$. We now construct $\mathcal{Q}'$. Its 1-morphism 
    component at an object $a \in \C$ is given by 
    \begin{equation}\label{Equation: 1-morphism component of Q}
    \xymatrix{
        Q_a' \colon {}^{*}\!a \ar[rr]^-{\id \circ \ev_a^{R}} && {}^{*}\!a \circ a \circ a^* \ar[rr]^-{\widetilde{\ev}_a \circ \id} && a^*
    }
    \end{equation}
    and its 2-morphism component $Q_X'$ at a 1-morphism $X \colon a \to b$ in $\C$ is given by pasting the cells in the diagram
    \begin{equation}\begin{aligned}
        \xymatrix@C=7em@R=4em{
            & {}^{*}\!b \ar[r]^-{\id \circ \ev_b^{R}} & {}^{*}\!b  b b^* \ar[dr]^-{\widetilde{\ev}_b \circ \id} \\
            {}^{*}\!a \ar[ur]^-{{}^{*}\!X} \ar[dr]_-{\id \circ \ev_a^{R}} \ar[r]_-{\id \circ \ev_b^{R}} & {}^{*}\!a b b^* \ar[r]_-{\id \circ X \circ \id} \ar[ur]_-{{}^{*}\!X \circ \id \circ \id} &  {}^{*}\!a a b^* \ar[r]^-{\widetilde{\ev}_a \circ \id}  & b^*  \\
            & {}^{*}\!a  a a^* \ar[ur]_-{\id \circ \id \circ (X^{LL})^*} \ar[r]_-{\widetilde{\ev}_a \circ \id} & a^* \ar[ur]_-{(X^{LL})^*}.
        }
    \end{aligned}\end{equation}
    here, the upper left and lower right square commute by the interchange law. The upper right diagram is filled with the naturality data of the duality 
    structure, that is, the 2-isomorphism 
    \begin{equation}
        \widetilde{\Omega}_X \colon \widetilde{\ev}_b \bullet ({}^{*}\!X \circ \id) \cong \widetilde{\ev}_a \circ (\id \circ X),
    \end{equation}
    and the lower left diagram is filled with the 2-isomorphism 
    \begin{equation}
        \Omega_{X^{L}}^{R} \colon (X \circ \id) \bullet \ev_b^{R} \cong (\id \circ (X^{LL})^*)\bullet \ev_a^{R}.
    \end{equation}
    where we have used the interaction between the duality and adjunction functors. One checks that this defines an invertible 2-transformation.
\end{proof}
\begin{corollary}\label{Corollary: the Serre isomorphism}
    Let $\B$ be a rigid symmetric monoidal 2-category. Then, there is an invertible 2-transformation
    \begin{equation}
        \mathcal{S} := \left(\xymatrix{\id \ar[r]^-{\mathcal{P}^{-1}} & {}^{**}\!(-) \ar[r]^-{\mathcal{Q}^{-1}} & (-)^{RR}} \right)
    \end{equation}
    which is called the \emph{Serre isomorphism}. Its 1-morphism component at $a \in \B$ is given by 
    \begin{equation}
        S_a \colon (\id \otimes \ev_a) \circ (\beta_{a,a} \otimes \id) \circ (\id \otimes \ev^R_a) \colon a \xymatrix{\ar[r]&} a.
    \end{equation}
\end{corollary}
\begin{proof}[Proof of \Cref{Proposition: rigid symmetric monoidal 2-categories have Sotwo volutive structures}]
    Let $\B$ be a rigid symmetric monoidal 2-category, that is, $\B$ has all duals and adjoints. We have already constructed an invertible 2-transformation 
    $\mathcal{S} \colon \id_{\B} \to (-)^{RR}$ in \Cref{Corollary: the Serre isomorphism}, the Serre isomorphism. It remains to show that the twisted 
    quantum dimension of the Serre isomorphism is trivial, that is, we need to check that the diagram in \Cref{Diagram: SO2volution} is trivial. By 
    \Cref{Corollary: the Serre isomorphism}, this is a tedious but straightforward computation.
\end{proof}

\subsection{Relative $\Sotwo$-volutive 2-categories}
\begin{definition}
    A \emph{2-contravariant $\Sotwo$-volutive 3-category} consists of a 3-category $\D$ with adjoints for 2-morphisms, an invertible 3-transformation 
    $S \colon \id_{\D} \to (-)^{RR}$, and an invertible pertubation $\sigma$ trivializing the 3-modification 
    \begin{equation}\label{Diagram: SO2volution 2}
    \begin{aligned}
    \xymatrix{
        \operatorname{id}_{\D} \ar@/_4pc/[dd]_-{\id}  \ar[r]^-{\id} \ar[d] & \operatorname{id}_{\D} \ar[d] \ar@{=>}[dl]^-{\sigma_{\D}} \ar@/^4pc/[dd]^-{\id} \\
        (-)^{LL} \circ (-)^{RR} \ar[r] \ar[d]_{S^{-1} \circ S} &  (-)^{RR} \circ (-)^{LL} \ar[d]^-{S \circ S^{-1}} \ar@{=>}[dl]^-{\beta_{S^{-1},S}}\\
        \operatorname{id}_{\D} \ar[r]_-{\id} & \operatorname{id}_{\D} 
    }
    \end{aligned}
    \end{equation}
\end{definition}
\begin{remark}
    Let $(\D,S,\sigma)$ be a 2-contravariant $\Sotwo$-volutive 3-category. Then, the pair $(S,\sigma)$ induces an $\Sotwo$-action on the 
    underlying (3,1)-category $\D_{(3,1)} \subseteq \D$. Conversely, any $\Sotwo$-action on a (3,1)-category defines a 2-contravariant 
    $\Sotwo$-volutive 3-category. 
\end{remark}
\begin{example}\label{Example: The Sotwo-volutive 3-category of 2-categories with adjoints}
    The 3-category $2\hspace{-0.05cm}\Cat^{\operatorname{adj}}$ admits a 2-contravariant $\Sotwo$-volutive structure which reduces to the 
    $\Sotwo$-action on $2\hspace{-0.05cm}\Cat_{(3,1)}^{\operatorname{adj}}$ described in \Cref{Construction: The defining Sotwo action of volutive categories}. 
    This follows essentially from our explicit description of this action, together with \Cref{Lemma: Trivializing the double adjoint on 2-categories with adjoints}.
\end{example}
\begin{definition}\label{Definition: functor of 2-contravariant Sotwo volutive 3-categories}
    A \emph{functor} of 2-contravariant $\Sotwo$-volutive 3-categories consists of a 3-functor $F \colon \D \to \D'$ and an invertible modification
    \begin{equation}
        \begin{aligned}
            \eta \colon 
            \left(
            \begin{matrix}
            \xymatrix{
                & \ar@{}[d]^(.3){}="a"^(1){}="b" \ar@{=>}^{S} "a";"b" &  \\
                \D \ar[rr]_-{(-)^{RR}} \ar[d]_-{F} \ar@/^2pc/[rr]^-{\id_{\D}} && \D \ar[d]^-{F} \ar@{=>}[dll]^-{\cong} \\
                \D' \ar[rr]_-{(-)^{RR}} && \D'
            }
            \end{matrix}\right)
            \;\RRightarrow\;
            \left(
            \begin{matrix}
                \xymatrix{
                \D \ar[rr]^-{\id_{\D}} \ar[d]_-{F} && \D \ar[d]^-{F} \ar[d]^-{F} \ar@{=>}[dll]_-{\cong}\\
                \D' \ar[rr]^-{\id_{\D'}} \ar@/_2pc/[rr]_-{(-)^{RR}} & \ar@{}[d]^(0){}="a"^(0.7){}="b" \ar@{=>}^{S'} "a";"b" & \D'\\
                & & 
            }
            \end{matrix}
            \right)
        \end{aligned}
    \end{equation}
    satisfying a compatibility condition with $\sigma,\sigma'$, namely the two pertubations
    \begin{equation}
        \left(\begin{matrix}
        \xymatrix@C=0.65em{
            F \ar@/_5pc/[dd]_-{\id}  \ar[r]^-{\id} \ar[d] & F \ar[d] \ar@{=>}[dl]^-{\sigma_{\B} \circ \id} \ar@/^5pc/[dd]^-{\id} \\
            (-)^{LL}_{\D'} \circ (-)^{RR}_{\D'} \circ F \ar[r] \ar[d]_{S' \circ S'^{-1} \circ \id} &  (-)^{RR}_{\D'} \circ (-)^{LL}_{\D'} \circ F \ar[d]^-{S'^{-1} \circ S' \circ \id} \ar@{=>}[dl]^-{\beta_{S',S'^{-1}} \circ \id}\\
            F \ar[r]_-{\id} & F
        } 
        \end{matrix}\right)
        \overset{\sigma \circ \id}{\RRightarrow}
        \left(\begin{matrix}
        \xymatrix{
            F \ar[r]^-{\id} \ar[dd]^-{\id} & F \ar[dd]^-{\id} \ar@{=>}[ddl]^-{\cong} \\
            & \\
            F \ar[r]^-{\id} & F
        }
        \end{matrix}\right)
    \end{equation}
    and 
    \begin{align}\label{Equation: diagram for functor of 2-contravariant Sotwo volutive 2-categories}
        \begin{aligned}
        &\left(\begin{matrix}
        \xymatrix@C=0.65em{
            F \ar@/_4.5pc/[dd]_-{\id}  \ar[r]^-{\id} \ar[d] & F \ar[d] \ar@{=>}[dl]^-{\sigma_{\B} \circ \id} \ar@/^4.5pc/[dd]^-{\id} \\
            (-)^{LL}_{\D'} \circ (-)^{RR}_{\D'} \circ F \ar[r] \ar[d]_{S' \circ S'^{-1} \circ \id} &  (-)^{RR}_{\D'} \circ (-)^{LL}_{\D'} \circ F \ar[d]^-{S'^{-1} \circ S' \circ \id} \ar@{=>}[dl]^-{\beta_{S',S'^{-1}} \circ \id}\\
            F \ar[r]_-{\id} & F
        } 
        \end{matrix}\right) \RRightarrow  \\
        &\left(\begin{matrix}
        \xymatrix@C=0.65em{
            F \ar@/_4.5pc/[dd]_-{\id}  \ar[r]^-{\id} \ar[d] & F \ar[d] \ar@{=>}[dl]^-{\id \circ \sigma_{\B}} \ar@/^4.5pc/[dd]^-{\id} \\
            F \circ (-)^{LL}_{\D} \circ (-)^{RR}_{\D} \ar[r] \ar[d]_{\id \circ S \circ S^{-1}} & F \circ (-)^{RR}_{\D} \circ (-)^{LL}_{\D} \ar[d]^-{\id \circ S^{-1} \circ S} \ar@{=>}[dl]^-{\id \circ \beta_{S,S^{-1}}}\\
            F \ar[r]_-{\id} & F
        } 
        \end{matrix}\right) 
        \overset{\id \circ \sigma}{\RRightarrow}
        \left(\begin{matrix}
        \xymatrix{
            F \ar[r]^-{\id} \ar[dd]^-{\id} & F \ar[dd]^-{\id} \ar@{=>}[ddl]^-{\cong} \\
            & \\
            F \ar[r]^-{\id} & F
        }
        \end{matrix}\right)
        \end{aligned}
    \end{align}
    are required to be equal; we spell out the unlabeled pertubation appearing here in more detail in 
    \Cref{Remark: Pertubations in functor of 2-contravariant Sotwo volutive 2-categories}.
\end{definition}
\begin{remark}\label{Remark: Pertubations in functor of 2-contravariant Sotwo volutive 2-categories}
    In the following, we wish to describe the unlabeled pertubation in \Cref{Equation: diagram for functor of 2-contravariant Sotwo volutive 2-categories} more explicitly. 
    In order to describe the pertubation, we first need to describe the vertical 3-transformations and the 
    respective 3-modifications. The only non-trivial vertical 3-transformations are the (invertible) 3-transformations 
    \begin{align}
        \begin{aligned}
                (-)^{LL}_{\D'} \circ (-)^{RR}_{\D'} \circ F \to F \circ (-)^{LL}_{\D} \circ (-)^{RR}_{\D} \\
                (-)^{RR}_{\D'} \circ (-)^{LL}_{\D'} \circ F \to F \circ (-)^{RR}_{\D} \circ (-)^{LL}_{\D} 
        \end{aligned}
    \end{align}
    expressing the compatibility of the 3-functor $F$ with adjoints. Next, we need to describe the 3-modifications ensuring that all the horizontal and 
    vertical 3-transformations are compatible. There are five such non-trivial (invertible) 3-modifications in total, which we will now describe. 
    First, we have the two invertible 3-modifications
    \begin{equation}
        \begin{aligned}
        \xymatrix{
        F \ar[r] \ar[d] & (-)^{LL}_{\D'} \circ (-)^{RR}_{\D'} \circ F \ar[d]^{\cong} \ar@{=>}[dl]^-{\cong} \\
        F \ar[r] & F \circ (-)^{LL}_{\D} \circ (-)^{RR}_{\D} 
        }
        \hspace{1cm}
        \xymatrix{
        F \ar[r] \ar[d] & (-)^{RR}_{\D'} \circ (-)^{LL}_{\D'} \circ F \ar[d]^{\cong} \ar@{=>}[dl]^-{\cong} \\
        F \ar[r] & F \circ (-)^{RR}_{\D} \circ (-)^{LL}_{\D} 
        }
        \end{aligned}
    \end{equation}
    respectively, which record the fact that the data expressing that any 3-functor is compatible with adjoints, is compatible with the adjunction data itself. 
    Second, we have the two invertible 3-modifications obtained by pasting together the 3-modifications in the diagrams 
    \begin{equation}
        \begin{aligned}
        \xymatrix{
            (-)^{LL}_{\D'} \circ (-)^{RR}_{\D'} \circ F \ar[rr]^-{S' \circ S'^{-1} \circ \id} \ar[dr]^-{\cong} \ar[dd]^-{\cong} && F \ar[dd]^-{\id} \\
            &(-)^{LL}_{\D'} \circ F \circ (-)^{RR}_{\D} \ar[ur]^-{S' \circ \id \circ S^{-1}} \ar[dr]^-{S' \circ \id \circ S^{-1}} & \\
            F \circ (-)^{LL}_{\D} \circ (-)^{RR}_{\D} \ar[rr]^-{\id \circ S \circ S^{-1}} \ar[ur]^-{\cong} && F
        }
        \end{aligned}
    \end{equation}
    and 
    \begin{equation}
        \begin{aligned}
        \xymatrix{
            (-)^{RR}_{\D'} \circ (-)^{LL}_{\D'} \circ F \ar[rr]^-{S' \circ S'^{-1} \circ \id} \ar[dr]^-{\cong} \ar[dd]^-{\cong} && F \ar[dd]^-{\id} \\
            &(-)^{RR}_{\D'} \circ F \circ (-)^{LL}_{\D} \ar[ur]^-{S' \circ \id \circ S^{-1}} \ar[dr]^-{S' \circ \id \circ S^{-1}} & \\
            F \circ (-)^{RR}_{\D} \circ (-)^{LL}_{\D} \ar[rr]^-{\id \circ S \circ S^{-1}} \ar[ur]^-{\cong} && F
        }
        \end{aligned}
    \end{equation}
    which are in turn filled with the 3-modifications recording the interchange law (respective right triangles), terms built from $\eta$ and its inverse
    (respective upper and lower triangles), and the construction of compatibility data for iterated adjoints and a 3-functor (respective left triangle). 
    Finally, we have an (invertible) 3-modification in the diagram 
    \begin{equation}
        \begin{aligned}
        \xymatrix{
            (-)^{LL}_{\D'} \circ (-)^{RR}_{\D'} \circ F \ar[r] \ar[d] & (-)^{RR}_{\D'} \circ (-)^{LL}_{\D'} \circ F \ar[d] \ar@{=>}[dl]^-{\cong} \\
            F \circ (-)^{LL}_{\D} \circ (-)^{RR}_{\D} \ar[r] & F \circ (-)^{RR}_{\D} \circ (-)^{LL}_{\D} 
        }
        \end{aligned}
    \end{equation}
    which records the naturality of the compatibility data of iterated adjoints and any 3-functor. Having discussed all five non-trivial 3-modifications, 
    we are now in position to discuss the four individual pertubations (upper, lower, left, and right) whose pasting yields the unlabeled pertubation 
    of \Cref{Equation: diagram for functor of 2-contravariant Sotwo volutive 2-categories} we set out to describe. The left and right pertubations express 
    naturality with respect to 3-modifications (here constructed from $\eta$) of 
    the naturality data with respect to 3-transformations (here constructed from $S$ and $S'$) of the adjunction data (here for $(-)^{RR}$ and $(-)^{LL}$). 
    The upper pertubation is constructed by unwinding various levels of compatibility of adjunction data with itself. The lower pertubation expresses 
    naturality with respect to 3-modifications (here constructed from $\eta$) of the symmetric braiding $\beta$. Together, these form the unlabeled 
    pertubation in \Cref{Equation: diagram for functor of 2-contravariant Sotwo volutive 2-categories}. 
\end{remark}
\begin{remark}
    Weaker notions of functors between 2-contravariant $\Sotwo$-volutive 3-categories are possible, e.g. where $\eta$ is not required to be invertible.
\end{remark}
\begin{definition}
    A \emph{relative $\Sotwo$-volutive 2-category} consists of a 2-contravariant $\Sotwo$-volutive 3-category $\D$ and a 
    functor of 2-contravariant $\Sotwo$-volutive 3-categories $\D \to 2\hspace{-0.05cm}\Cat^{\operatorname{adj}}$.
\end{definition}
\begin{notation}
    Occasionally we will call a relative $\Sotwo$-volutive 2-category $\D \to 2\hspace{-0.05cm}\Cat^{\operatorname{adj}}$ a 
    relative $\Sotwo$-volutive 2-category \emph{over} $\D$. This terminology is motivated as usual via straightening-unstraightening.
\end{notation}
\begin{example}
    An $\Sotwo$-volutive 2-category is the same as a functor of 2-contravariant $\Sotwo$-volutive 3-categories $\star \to 2\hspace{-0.05cm}\Cat^{\operatorname{adj}}$. 
\end{example}
\begin{remark}
    Let $U \colon 2\hspace{-0.05cm}\Cat^{\operatorname{adj}} \to 2\hspace{-0.05cm}\Cat$ be the forgetful functor from the 3-category of 2-categories with adjoints 
    to the 3-category of 2-categories. This functor has a right adjoint $U^R$ which maps a 2-category $\B$ to the sub 2-category 
    with the same objects, those 1-morphisms which have both adjoints, and all 2-morphisms between them. It is not known to us whether $U$ also has a left adjoint.
\end{remark}

\subsection{Lax $\Sotwo$-volutive 2-categories}
\begin{definition}
    A \emph{lax $\Sotwo$-volutive 2-category} is a 2-category $\B$ with adjoints together with a 2-transformation $S \colon \id \to (-)^{RR}$ such that the composed 
    2-modification 
    \begin{equation}\label{Diagram: SO2volution 3}
        \begin{aligned}
        \xymatrix{
        \operatorname{id}_{\mathcal{B}} \ar@/_4pc/[dd]_-{\id}  \ar[r]^-{\id} \ar[d] & \operatorname{id}_{\mathcal{B}} \ar[d] \ar@{=>}[dl]^-{\sigma_{\B}} \ar@/^4pc/[dd]^-{\id} \\
        (-)^{LL} \circ (-)^{RR} \ar[r] \ar[d]_{S \circ S^{R}} &  (-)^{RR} \circ (-)^{LL} \ar[d]^-{S^{R} \circ S} \ar@{=>}[dl]^-{\beta_{S,S^{R}}}\\
        \operatorname{id}_{\mathcal{B}} \ar[r]_-{\id} & \operatorname{id}_{\mathcal{B}} 
        }
        \end{aligned}
    \end{equation}
    is the identity on $\operatorname{id}_{\operatorname{id}_{\mathcal{B}}}$; here, the upper square is filled with the modification $\sigma_{\B}$ described in 
    \Cref{Construction: The defining Sotwo action of volutive categories}, the lower square is filled with the modification expressing naturality 
    of the braiding $\beta$ on $\Aut_{\id_{2\hspace{-0.05cm}\Cat^{\operatorname{adj}}}}$, and the left and right diagrams are filled with the modifications expressing
    naturality of adjunction data.
\end{definition}
\begin{example}
    Any $\Sotwo$-volutive 2-category is in particular a lax $\Sotwo$-volutive 2-category. Conversely, any lax $\Sotwo$-volutive 2-category $(\B,S)$ gives rise 
    to an $\Sotwo$-volutive 2-category by considering the full sub 2-category on those objects $a$, for which $S_a$ is a 1-isomorphism. 
\end{example}
\begin{example}
    Any 0-closed symmetric monoidal 2-category with adjoints admits a lax $\Sotwo$-volutive structure. The construction is essentially the same as in the 
    case of rigid symmetric monoidal 2-categories, with the exception that the 1-morphism component of the Serre morphism, that is 
    \begin{equation}
            S_a \colon (\id \otimes \ev_a) \circ (\beta_{a,a} \otimes \id) \circ (\id \otimes \ev^R_a) \colon a \xymatrix{\ar[r]&} a,
    \end{equation}
    is generically not a 1-isomorphism anymore. In other words, we really have a lax $\Sotwo$-volutive structure, rather than an ordinary $\Sotwo$-volutive one.
\end{example}
\begin{definition}
    Consider the 2-category generated freely by two objects $a$ and $b$, two 1-morphisms $f \colon a \to b$ and $g \colon b \to a$, four 2-morphisms 
    $\ev \colon f \circ g \to \id, \coev \colon \id \to g \circ f, \widetilde{\ev} \colon g \circ f \to \id$, and $\widetilde{\coev} \colon \id \to f \circ g$, 
    subject to the Zorro moves for $\ev,\coev$ and $\widetilde{\ev},\widetilde{\coev}$, respectively. We denote this 2-category by $\AmAdj$, the 
    \emph{walking ambidextrous adjunction}. 
\end{definition}
\begin{remark}
    There are two 2-functors $\Adj \to \AmAdj$, recording the two adunctions $f \dashv g$ and $g \dashv f$. In particular, $\AmAdj$ has all adjoints, since 
    $g$ is both left and right adjoint to $f$ and vice versa. The double adjoint functor $(-)^{RR} \colon \AmAdj \to \AmAdj$ is by construction the identity 
    on objects and on 1-morphisms. We wish to understand whether $(-)^{RR}$ is really equivalent (or even equal) to the identity functor on $\AmAdj$. In order 
    to solve this problem, we need a short digression.\\
    First, we note the following: we have 2-morphisms $\ev^R \colon \id \cong \id^R \to (f \circ g)^R \cong g^R \circ f^R \cong f \circ g$ 
    and $\coev^R \colon g \circ f \cong f^R \circ g^R \cong (g \circ f)^R \to \id^R \cong \id$ which satisfy the Zorro moves by virtue of the Zorro moves 
    for $\ev,\coev$. In particular, $\coev^R,\ev^R$ exhibit $g$ as a left adjoint to $f$, in the same way as $\widetilde{\ev},\widetilde{\coev}$. \\
    However, adjunction data is essentially unique in the following sense: let $X \colon c \to d$ be a 1-morphism in any 2-category $\B$ with adjoints
    and let $(Y,\ev,\coev)$ and $(Y',\ev',\coev')$ be right adjunction data for $X \dashv (-)$. Then, there exists a unique 2-isomorphism $\Omega \colon Y \to Y'$ 
    satisfying $\ev = \ev' \bullet (\id \circ \Omega)$ and $(\Omega \circ \id) \bullet \coev = \coev'$. Explicitly, $\Omega$ is given by the composition 
    \begin{equation}
        \xymatrix@C=3em{
            Y \ar[r]^-{\coev' \circ \id} & Y' \circ X \circ Y \ar[r]^-{\id \circ \ev} & Y'.
        }
    \end{equation}
    Next, we recall the obvious fact that $X$ is a left adjoint of $X^R$, with adjunction data given by $\ev_X,\coev_X$. In particular, we have 
    2-isomorphisms $\Psi_X \colon X \cong (X^R)^L$. In fact, for $Z \colon M \to N$ any 2-morphism, we have a commutative diagram 
    \begin{equation}
        \begin{aligned}
        \xymatrix{
            M \ar[r]^-{\Psi_M} \ar[d]^-{Z} & (M^R)^L \ar[d]^-{(Z^R)^L} \\
            N \ar[r]^-{\Psi_N} & (N^R)^L
        }
        \end{aligned}
    \end{equation}
    so that $(Z^R)^L = \Psi_N \bullet Z \bullet \Psi^{-1}_M$. Combining our simple observations, we find 
    \begin{align*}
        Z^{RR} \cong (Z^R)^R &= [(\id \circ \ev_N) \bullet (\id \circ Z \circ \id) \bullet (\coev_M \circ \id)]^R \\
        &= (\id \circ \coev_M^R) \bullet (\id \circ Z^R \circ \id) \bullet (\ev_N^R  \circ \id)\\
        &= (\id \circ [\widetilde{\ev}_M \bullet (\Omega_M \circ \id)]) \bullet (\id \circ Z^R \circ \id) \bullet ([(\id \circ \Omega_N^{-1}) \bullet \widetilde{\coev}_N] \circ \id)\\
        &= (\id \circ \widetilde{\ev}_M) \bullet (\id \circ (\Omega_M \bullet Z^R \bullet \Omega_N^{-1}) \circ \id) \bullet  (\widetilde{\coev}_N \circ \id)\\
        &= (\Omega_M \bullet Z^R \bullet \Omega_N^{-1})^L \\
        &= \Omega_N \bullet (Z^R)^L \bullet \Omega_M^{-1} \\
        &= \Omega_N \bullet \Psi_N \bullet Z \bullet \Psi^{-1}_M \bullet \Omega_M^{-1} \cong (\Omega_N \bullet \Psi_N) \bullet Z \bullet (\Omega_M \bullet \Psi_M)^{-1}.
    \end{align*}
    Hence, we obtain an invertible 2-transformation $S \colon \id \cong (-)^{RR}$ whose 1-morphism components are identites and whose 2-morphism component at a 
    1-morphism $M$ is given by $\Omega_M \bullet \Psi_M$. Moreover, the twisted quantum dimension (see \Cref{Diagram: SO2volution}) of $S$ is trivial. 
\end{remark}
We record our findings in the following.
\begin{lemma}\label{Lemma: Ambidextrous adjunction is volutive}
    $\AmAdj$ admits an $\Sotwo$-volutive structure.
\end{lemma}
\begin{construction}\label{From Sotwo volutive 2-categories to 3-categories}
    Given any 2-category $\B$, we can construct a 3-category $\B[1]$ which has two objects we denote $0$ and $1$ and only one non-trivial Hom 2-category, namely 
    $\Hom_{[1](\B)}(0,1) := \B$. If the 2-category $\B$ is equiped with an $\Sotwo$-volutive structure, then the 3-category 
    $\B[1]$ receives a 2-contravariant $\Sotwo$-volutive structure. In particular, if $\B$ has adjoints then $\B[1]$ has adjoints for 2-morphisms. 
    The notation is motivated by the notation $[1] := \{0 \rightarrow 1\}$ for the walking arrow, which we recover upon setting $\B = \star$, the final 2-category.
\end{construction}
\begin{remark}
    \Cref{From Sotwo volutive 2-categories to 3-categories} provides us with a large class of examples of 2-contravariant $\Sotwo$-volutive 3-categories. 
    Indeed, any $\Sotwo$-volutive 2-category (such as any rigid symmetric monoidal 2-category) gives rise to one. 
\end{remark}
\begin{corollary}\label{Corollary: the categorified ambidextrous adjunction admits an Sotwo volutive structure}
    $\AmAdj[1]$ is a 2-contravariant $\Sotwo$-volutive 3-category.
\end{corollary}
\begin{proof}
    Combine \Cref{Lemma: Ambidextrous adjunction is volutive} and \Cref{From Sotwo volutive 2-categories to 3-categories}.
\end{proof}
\begin{remark}
    Let $(\B,S)$ be an $\Sotwo$-volutive 2-category. Then, we have an invertible 2-modification in the diagram 
    \begin{equation}
        \begin{aligned}
        \xymatrix{
        (-)^{RR} \circ (-)^{RR} \ar[r]^-{\cong} \ar[d]_-{S^{-1} \circ \id} &(-)^{RR} \circ (-)^{RR} \ar[d]^-{\id \circ S^{-1}} \ar@{=>}[dl]_-{\cong} \\
        \id \circ (-)^{RR} \ar[r]_-{\cong} & (-)^{RR} \circ \id
        }
        \end{aligned}
    \end{equation}
    expressing the fact that the compatibility data of any functor with the (double) adjoint functor is natural. In particular, the components of 
    this modification provide 2-isomorphisms $S_a \cong S_a^{RR}$ for any object $a \in \B$ so that $S \cong S^{RR}$ in $2\hspace{-0.05cm}\Cat^{\operatorname{adj}}$.
    We will tacitly use this identification (or its equivalent form, $S^R \cong S^L$) in the following.
\end{remark}
\begin{proposition}\label{Proposition: lax and relative Sotwo volutive 2-categories}
    Any lax $\Sotwo$-volutive 2-category defines a relative $\Sotwo$-volutive 2-category over $\AmAdj[1]$.
\end{proposition}
\begin{proof}
    Let $(\B,S)$ be a lax $\Sotwo$-volutive 2-category. We define a 2-contravariant $\Sotwo$-volutive 3-functor 
    $F \colon \AmAdj[1] \to 2\hspace{-0.05cm}\Cat^{\operatorname{adj}}$ as follows. We set 
    \begin{align}
        \begin{aligned}
            &F(0) = \B, F(1) = \B^{(1,2)\opp} \\
            &F(a) = (-)^R, F(b) = (-)^L \\
            &F(f) = S, F(g) = S^R \cong S^L\\
            &F(\ev) = \ev_S, F(\coev) = \coev_S, \\
            &F(\widetilde{\ev}) = \widetilde{\ev}_{S}, F(\widetilde{\coev}) = \widetilde{\coev}_S
        \end{aligned}
    \end{align} 
    where in the last line we have tacitly identified $S^R$ and $S^L$. 
    First, this is a 3-functor since $S^R$ is a right (2-)adjoint to $S$ in $2\hspace{-0.05cm}\Cat^{\operatorname{adj}}$, so that the Zorro moves are satisfied. 
    The $\Sotwo$-volutive structure on the 3-functor, that is, the invertible modification $\eta$ is given by $S$ and the coherence condition is satisfied 
    because the twisted quantum dimension of $S$ is trivial. 
\end{proof}
\begin{remark}\label{Remark: on lax and relative Sotwo volutive 2-categories}
    2-contravariant $\Sotwo$-volutive 3-functor $F \colon \AmAdj[1] \to 2\hspace{-0.05cm}\Cat^{\operatorname{adj}}$ arising from 
    \Cref{Proposition: lax and relative Sotwo volutive 2-categories} are very special. For instance, we note that on 1-morphism level, the functor 
    $F$ assigns both $a$ and $b$ in $\AmAdj[1]$ to an invertible 1-morphism in $2\hspace{-0.05cm}\Cat^{\operatorname{adj}}$. In other words, 
    The functor $F \colon \AmAdj[1] \to 2\hspace{-0.05cm}\Cat^{\operatorname{adj}}$ admits an extension 
    \begin{equation}
        \begin{aligned}
        \xymatrix{
            (\AmAdj[1])[a^{-1},b^{-1}] \ar@{-->}[dr]^-{} & \\
            \AmAdj[1] \ar[r]^{} \ar[u] & 2\hspace{-0.05cm}\Cat^{\operatorname{adj}}
        }
        \end{aligned}
    \end{equation}
    where $(\AmAdj[1])[a^{-1},b^{-1}]$ is the 3-category obtained from $\AmAdj[1]$ by freely adding inverses to $a$ and $b$ and the vertical arrow is the 
    inclusion/localization. \\
    In a similar vein, recall that any $\Sotwo$-volutive 2-category is in particular a lax $\Sotwo$-volutive 2-category, hence also fits into this scheme. These 
    are even more special, in that the associated functor $F \colon \AmAdj[1] \to 2\hspace{-0.05cm}\Cat^{\operatorname{adj}}$ admits an extension 
    \begin{equation}
        \begin{aligned}
        \xymatrix{
            (\AmAdj[1])[a^{-1},b^{-1},\ev^{-1},\coev^{-1},\widetilde{\ev}^{-1},\widetilde{\coev}^{-1}] \ar@{-->}[ddr] & \\
            (\AmAdj[1])[a^{-1},b^{-1}] \ar[u] \ar@{-->}[dr]^-{} & \\
            \AmAdj[1] \ar[r]^{} \ar[u] & 2\hspace{-0.05cm}\Cat^{\operatorname{adj}}
        }
        \end{aligned}
    \end{equation}
    where $(\AmAdj[1])[a^{-1},b^{-1},\ev^{-1},\coev^{-1},\widetilde{\ev}^{-1},\widetilde{\coev}^{-1}]$ is the 3-category obtained from 
    $(\AmAdj[1])[a^{-1},b^{-1}]$ by also freely adding inverses to $\ev,\coev,\widetilde{\ev}$ and $\widetilde{\coev}$. Note that this  
    implies that $S \dashv S^R$ is an adjoint equivalence (and in particular that $S$ is invertible, with $S^{-1} \cong S^R \cong S^L$). 
    In particular, we have $(\AmAdj[1])[a^{-1},b^{-1},\ev^{-1},\coev^{-1},\widetilde{\ev}^{-1},\widetilde{\coev}^{-1}] \cong \star$, the final 3-cateogry. 
    In other words, we have repackaged the fact that an $\Sotwo$-volutive 3-category is a 2-contravariant $\Sotwo$-volutive 3-category 
    $\star \to 2\hspace{-0.05cm}\Cat^{\operatorname{adj}}$ by providing a more intricate (but equivalent) model for $\star$. 
\end{remark}
\begin{remark}
    Currently we do not know whether there is a variant of lax $\Sotwo$-volutive 2-categories that corresponds precisely to 
    2-contravariant $\Sotwo$-volutive 3-functors $F \colon \AmAdj[1] \to 2\hspace{-0.05cm}\Cat^{\operatorname{adj}}$, rather than such that descend 
    to a localization of the domain. 
\end{remark}

\section{Lax hermitian fixed points}\label{Section: Lax hermitian fixed points}
In this section, we discuss (lax) hermitian fixed points. We start with a discussion in the lax $\Oone$-volutive case and then turn our attention the 
the lax $\Sotwo$-volutive case. After discussing some general theory, we explain how lax $\Oone$-volutive categories and lax $\Sotwo$-volutive 2-categories 
are to be understood as lax hermitian fixed points in $\Cat$ and $2\hspace{-0.05cm}\Cat^{\operatorname{adj}}$, respectively.
\subsection{The lax $\Oone$-volutive case}
In this section we discuss the notion of a lax hermitian fixed point in a lax $\Oone$-volutive category.
\begin{definition}
    Let $(\C,d,\eta)$ be a lax $\Oone$-volutive category. A \emph{lax hermitian fixed point} in $\C$ is a pair $(a,\theta_a)$ consisting 
    of an object $a \in \C$ and a morphism $\theta_a \colon a \to d(a)$ satisfying $\theta_a = d(\theta_a) \circ \eta_a$. 
\end{definition}
\begin{remark}
    Lax hermitian fixed points have appeared in the hermitian algebraic K-theory literature before under the name 
    \emph{symmetric form}, see e.g. \cite[Section 3]{Schlichting2010}.
\end{remark}
\begin{remark}\label{Remark: hermitian fixed points of a lax Oone volutive category}
    Let $(\C,d,\eta)$ be a lax $\Oone$-volutive category. Consider the category $\hat{\C}$ whose objects are lax hermitian fixed points 
    and whose morphisms are those of $\C$. The full subcategory of those pairs $(a,\theta_a)$ for which $\theta_a$ is an isomorphism is the dagger category 
    of \Cref{Remark: Oone volutive categories give dagger categories}; here, the dagger structure assigns a morphism $X \colon (a,\theta_a) \to (b,\theta_b)$ to 
    the morphism $\dagger(X) \colon (b,\theta_b) \to (a,\theta_a)$ given as the composition 
    \begin{equation}
        \xymatrix{
            b \ar[r]^-{\theta_b} & d(b) \ar[r]^-{d(X)} & d(a) \ar[r]^-{\theta_a^{-1}} & a.
        }
    \end{equation}
    We call pairs $(a,\theta_a)$ for which $\theta_a$ is an isomorphism \emph{(honest) hermitian fixed points}. Note that, in passing to the full subcategory of hermitian fixed points it does not matter whether we consider $\C$ or 
    its associated $\Oone$-volutive category described in \Cref{Construction: Oone volutions from lax Oone volutions}: the invertibility of $\theta_a$ forces 
    the invertibility of $\eta_a$ according to the coherence condition. 
\end{remark}
\begin{definition}\label{Definition: lax hermitian fixed points in 1-contravariant Oone volutive 2-categories}
    Let $(\B,d,\eta,\tau)$ be a 1-contravariant $\Oone$-volutive 2-category. A \emph{lax hermitian fixed point} in $\B$ is a triple consisting 
    of an object $a \in \C$, a morphism $\theta_a \colon a \to d(a)$, and a 2-morphism $\omega_a \colon \theta_a \to d(\theta_a) \circ \eta_a$ satisfying 
    \begin{equation}\label{Equation: coherence condition of lax hermitian fixed point in 2-category}
        \begin{pmatrix}
            \xymatrix{
                a \ar[d]^-{\theta_a} \ar[r]^-{\eta_a} & d^2(a) \ar[dl]^-{d(\theta_a)} \\
                d(a) &
            }
        \end{pmatrix}
        = 
        \begin{pmatrix}
            \xymatrix{
                a \ar[dd]_-{\theta_a} \ar[rr]^-{\eta_a} && d^2(a) \ar[dl]_-{d^2(\theta_a)} \ar@/^3pc/[ddll]^-{d(\theta_a)} \\
                & d^3(a) \ar@/^0.5pc/[dl]^-{d(\eta_a)} & \\
                d(a) \ar@/^0.5pc/[ur]^-{\eta_{d(a)}} && 
            }
        \end{pmatrix}
    \end{equation}
    where the left hand diagram is filled with $\omega_a$ while the right hand diagram is filled with $\eta_{\theta_a}$, $\tau_a$, and $d(\omega_a)$. 
\end{definition}
\begin{variant}\label{Variant: lax hermitian fixed points in 2-contravariant Oone volutive 2-categories}
    Let $(\B,d,\eta,\tau)$ be a 2-contravariant $\Oone$-volutive 2-category. A \emph{lax hermitian fixed point} in $\B$ is a triple consisting 
    of an object $a \in \C$, a morphism $\theta_a \colon a \to d(a)$, and a 2-morphism $\omega_a \colon \eta_a \to d(\theta_a) \circ \theta_a$ satisfying 
    a coherence condition involving $d^3(a)$ similar to \Cref{Equation: coherence condition of lax hermitian fixed point in 2-category}.
\end{variant}
\begin{remark}
    Lax hermitian fixed points in the sense of \Cref{Variant: lax hermitian fixed points in 2-contravariant Oone volutive 2-categories} are essentially 
    the same as volutive adjunctions in the sense of \Cref{Definition: volutive adjunction}.
\end{remark}
\begin{example}
    Consider the 2-contravariant $\Oone$-volutive 2-category $(\Cat,(-)^{\opp},\id,\id)$. A lax hermitian fixed point in $(\Cat,(-)^{\opp},\id,\id)$ is 
    a lax $\Oone$-volutive category. 
\end{example}
\begin{remark}\label{Remark: honest hermitian fixed points}
    Let $(\B,d,\eta,\tau)$ be a 2-contravariant $\Oone$-volutive 2-category. Consider the 2-category $\hat{\B}$ whose objects are lax hermitian fixed points 
    and whose 1- and 2-morphisms are those of $\B$. The full sub 2-category of those pairs $(a,\theta_a,\omega_a)$ for which $\theta_a$ and $\omega_a$ are 
    isomorphisms is the 2-contravariant $\Oone$-dagger 2-category of \Cref{Construction: turning 2-contra Oone vol 2-cats dagger}. We call triples 
    $(a,\theta_a,\omega_a)$ for which $\theta_a$ and $\omega_a$ are invertible \emph{(honest) hermitian fixed points}.
\end{remark}
\begin{remark}
    The notion of a lax hermitian fixed point presented in \Cref{Definition: lax hermitian fixed points in 1-contravariant Oone volutive 2-categories} and 
    \Cref{Variant: lax hermitian fixed points in 2-contravariant Oone volutive 2-categories} works more generally for lax 1-, 2- or (1,2)-contravariant 
    $\Oone$-volutive 2-categories.
\end{remark}
\begin{lemma}\label{Lemma: lax hermitian fixed points in the walking adjunction}
    Consider the 2-contravariant $\Oone$-volutive 2-category $\Adj$ described in \Cref{Lemma: 2-contravariant Oone volutive structure on the walking adjunction}. 
    We claim that $a \in \Adj$ admits a lax hermitian fixed point structure, while no object admits an (honest) hermitian fixed point structure.
\end{lemma}
\begin{proof}
    We will construct a lax hermitian fixed point structure for $a \in \Adj$. First, we have the 1-morphism $\theta_a := X \colon a \to F(a)=b$. 
    Second, we have the 2-morphism $\Pi_a := f \colon \id_a \to F(X) \circ X = Y \circ X$. The Zorro moves ensure that this defines a lax hermitian fixed point.
    To see that no object admits an (honest) hermitian fixed point structure, it suffices to note that $a$ and $b$ are not isomorphic. 
\end{proof}
\begin{remark}
    The lax hermitian fixed point structure on $a \in \Adj$ is not unique: a different one is given by the 1-morphism $\theta_a' := X \circ Y \circ X \colon a \to F(a)=b$ 
    together with the 2-morphism $\Pi_a' := f \circ f \colon \id_a \to F(X \circ Y \circ X) \circ X = Y \circ X \circ Y \circ X$. In fact, there are infinitely many 
    such lax hermitian fixed point structures, with the one presented in the proof of \Cref{Lemma: lax hermitian fixed points in the walking adjunction} being 
    the ``minimal'' one. This fact essentially implies \Cref{Remark: lax Oone volutive structures are many}. \\
    The object $b \in \Adj$ does not admit a lax hermitian fixed point structure, to the best of our knowledge. This apparent defect can be lifted by passing 
    to the walking ambidextrous adjunction, in which $Y$ is not only a right adjoint of $X$, but also a left adjoint. The coevaluation of the latter adjunction 
    then precisely supplies the missing 2-morphism of a lax hermitian fixed point structure on $b \in \Adj$, with 1-morphism given by $\theta_b := Y \colon b \to F(b) =a$.
\end{remark}
\begin{construction}\label{Construction: Oone volutive functors descend to lax hermitian fixed points}
    Let $(\B,d,\eta,\tau)$ and $(\B',d',\eta',\tau')$ be 2-contravariant $\Oone$-volutive 2-categories. According to \Cref{Remark: honest hermitian fixed points} we 
    may form the 2-categories $\hat{\B}$ and $\hat{\B}'$ whose objects are lax hermitian fixed points, and whose 1- and 2-morphisms are those of $\B$ and $\B'$, 
    respectively. Moreover, let $(F,\alpha,\Xi) \colon (\B,d,\eta,\tau) \to (\B',d',\eta',\tau')$ be a 2-contravariant $\Oone$-volutive functor. 
    We wish to construct a 2-functor $\hat{\B} \to \hat{\B}'$. To an object $(a,\theta_a,\omega_a) \in \hat{\B}$ we assign the triple consisting of 
    (i) the object $F(a) \in \B'$ (ii) the 1-morphism $\theta_{F(a)} := \alpha_a \circ F(\theta_a) \colon F(a) \to d'F(a)$ (iii) the 2-morphism $\omega_{F(a)}$ 
    is constructed from $\alpha_{\theta_a}, \Xi_a,$ and $\omega_a$ by appropriately filling the diagram 
    \begin{equation}
        \begin{aligned}
        \xymatrix{
            F(a) \ar[d]_-{\id_{F(a)}} \ar[r]^-{F(\theta_a)} & Fd(a) \ar[r]^-{\alpha_a} \ar[d]^-{Fd(\theta_a)} & d'F(a) \ar[d]^-{d'F(\theta_a)} \\
            F(a) \ar[r]^-{F(\eta_a)} \ar[dr]_-{\eta'_{F(a)}} & Fd^2(a) \ar[r]^-{\alpha_{d(a)}} & d'Fd(a) \ar[dl]^-{d'(\alpha_a)} \\
            & d'^2F(a) &
        }
        \end{aligned}
    \end{equation}
    One checks that $(F(a),\theta_{F(a)},\omega_{F(a)})$ defines a lax hermitian fixed point. On 1- and 2-morphism level our assignment is given simply by $F$.  
    In summary, 2-contravariant $\Oone$-volutive functors descend to the 2-categories of lax hermitian fixed points. 
\end{construction}
\begin{remark}
    It is evident from \Cref{Construction: Oone volutive functors descend to lax hermitian fixed points} that we could consider lax 2-contravariant $\Oone$-volutive 
    2-categories and lax 2-contravariant $\Oone$-volutive 2-functors instead; at no point did we need the invertibility of any map in sight. For the purposes 
    of this paper, the given construction is sufficient. 
\end{remark}
\begin{definition}\label{Definition: lax isometries}
    Let $(\C,d,\eta)$ be a lax $\Oone$-volutive category and let $(a,\theta_a)$ and $(b,\theta_b)$ be lax hermitian fixed points in $(\C,d,\eta)$. A 
    \emph{lax isometry} $(a,\theta_a) \to (b,\theta_b)$ is a morphism $X \colon a \to b$ satisfying $d(X) \circ \theta_b \circ X = \theta_a$.
\end{definition}
\begin{remark}
    Let $(\C,d,\eta)$ be a lax $\Oone$-volutive category, let $(a,\theta_a)$, $(b,\theta_b)$, and $(c,\theta_c)$ be lax hermitian fixed points in $(\C,d,\eta)$, 
    and let $X \colon (a,\theta_a) \to (b,\theta_b)$ and $Y \colon (b,\theta_b) \to (c,\theta_c)$ be lax isometries. Then, $Y \circ X \colon (a,\theta_a) \to (c,\theta_c)$
    is a lax isometry as well; indeed, we have 
    \begin{equation}
        d(Y \circ X) \circ \theta_c \circ (Y \circ X) = d(X) \circ d(Y) \circ \theta_c \circ Y \circ X = d(X) \circ \theta_b \circ X = \theta_a.
    \end{equation}
    We also note that $\id_a \colon (a,\theta_a) \to (a,\theta_a)$ is a lax isometry. 
\end{remark}
We record our findings in the following.
\begin{definition}\label{Definition: the category of lax hermitian fixed points}
    Let $(\C,d,\eta)$ be a lax $\Oone$-volutive category. We denote the category of lax hermitian fixed points and lax isometries in $(\C,d,\eta)$ by 
    $\LaxHerm(\C,d,\eta)$. 
\end{definition}
\begin{remark}
    Let $(\C,d,\eta)$ be a lax $\Oone$-volutive category. A lax isometry in its associated dagger category is the same as an ordinary isometry. 
    Indeed, recall that the dagger structure assigns each morphism $X \colon (a,\theta_a) \to (b,\theta_b)$ to $\dagger(X) = \theta_a^{-1} \circ d(X) 
    \circ \theta_b$, so that a lax isometry precisely satisfies $\dagger(X)\circ X = \id_a$.
\end{remark}
\begin{remark}
    Similarly, one defines lax unitaries to be lax isometries which are also invertible.
\end{remark}
\begin{lemma}
    Let $(F,\alpha) \colon (\C,d,\eta) \to (\C',d',\eta')$ be a lax $\Oone$-volutive functor. Then, $(F,\alpha)$ induces a functor 
    $\LaxHerm(\C,d,\eta) \to \LaxHerm(\C',d',\eta')$. 
\end{lemma}

\subsection{The lax $\Sotwo$-volutive case}\label{Lax hermitian fixed points of lax Sotwo-volutive categories}
In this section we discuss the notion of a lax hermitian fixed point in a lax $\Sotwo$-volutive 2-category. 
\begin{definition}
    Let $(\B,S)$ be a lax $\Sotwo$-volutive 2-category. A \emph{lax hermitian fixed point} in $\B$ is a pair $(a,\lambda_a)$ consisting of an object 
    $a \in \B$ and a 2-morphism $\lambda_a \colon \id_a \to S_a$.
\end{definition}
\begin{remark}
    Let $(\B,S)$ be a lax $\Sotwo$-volutive 2-category. Consider the 2-category $\hat{\B}$ whose objects are lax hermitian fixed points and whose 1- and 2-morphisms 
    are those of $\B$. The full sub 2-category of those objects for which $\lambda_a$ is a 2-isomorphism is an $\Sotwo$-dagger 2-category, see 
    \cite[Section 2.2.3]{carqueville2025orbifoldshigherdaggerstructures}; here, the $\Sotwo$-dagger structure has as its component at 
    $X \colon (a,\lambda_a) \to (b,\lambda_b)$ in $\hat{B}$ the 2-isomorphism
    \begin{equation}
        \xymatrix{
            X^{RR} \cong X^{RR} \circ \id_a \ar[r]^-{\id \circ \lambda_a} & X^{RR} \circ S_a \ar[r]^-{S_X} & S_b \circ X \ar[r]^-{\lambda_b^{-1} \circ \id} & \id_b \circ X \cong X.
        }
    \end{equation}
    We call pairs $(a,\lambda_a)$ for which $\lambda_a$ is a 2-isomorphism \emph{(honest)} hermitian fixed points. 
\end{remark}
\begin{example}
    Recall the $\Sotwo$-volutive 2-category $\AmAdj$ described in \Cref{Lemma: Ambidextrous adjunction is volutive}. We have seen that the $\Sotwo$-volutive 
    structure on $\AmAdj$ has identity 1-morphism components. In particular, every object in $\AmAdj$ admits both a lax and an honest hermitian fixed 
    point structure. 
\end{example}
\begin{definition}
    Let $(\D,S,\sigma)$ be a 2-contravariant $\Sotwo$-volutive 3-category. A \emph{lax hermitian fixed point} in $\D$ is a pair $(a,\lambda_a)$ 
    consisting of an object $a \in \B$ and a 2-morphism $\lambda_a \colon \id_a \to S_a$ satisfying a compatibility condition with $\sigma$, namely, that 
    the diagram of 3-morphisms 
    \begin{equation}\label{Equation: coherence condition of Sotwo lax hermitian fixed points}
        \begin{aligned}
        \xymatrix@C=3em{
            \id_a \ar[r]^-{\cong} \ar[d]^-{\cong} \ar@/_4pc/[dd]_-{\id} & \id_a \ar[d]^-{\cong} \ar@/^4pc/[dd]^-{\id} \\
            S_a^{-1} \circ S_a \ar[r]^-{S_{S_a^{-1}}} \ar[d]^-{\lambda_a \circ \lambda_a^R} & S_a \circ S_a^{-1} \ar[d]^-{\lambda_a^R \circ \lambda_a} \\
            \id_a \cong \id_a \circ \id_a \ar[r]^-{\cong} & \id_a \circ \id_a \cong \id_a
        }
        \end{aligned}
    \end{equation}
    equals the identity 3-morphism on $\id_{\id_a}$, where the left and right diagrams are filled with the naturality data of the adjunction/inversion data, 
    the upper diagram is filled with the 3-morphism component $\sigma_a$, and the lower diagram is filled with the naturality data of $S$ with respect to $\lambda_a$.
\end{definition}
\begin{remark}
    As in the 2-categorical case, we call a lax hermitian fixed point an (honest) hermitian fixed point if $\lambda_a$ is a 2-isomorphism. 
\end{remark}
\begin{example}
    Recall the 2-contravariant $\Sotwo$-volutive 3-category $2\hspace{-0.05cm}\Cat^{\operatorname{adj}}$ from \Cref{Example: The Sotwo-volutive 3-category of 2-categories with adjoints}.
    (Lax) hermitian fixed points in $2\hspace{-0.05cm}\Cat^{\operatorname{adj}}$ are precisely (lax) $\Sotwo$-volutive 2-categories. Indeed, 
    the object $a \in 2\hspace{-0.05cm}\Cat^{\operatorname{adj}}$ corresponds precisely to a 2-category $\B$ with adjoints, the 2-morphism 
    $\lambda_a \colon \id_a \to S_a$ corresponds precisely to a 2-transformation $\id_{\B} \to (-)^{RR}_{\B}$, and the coherence condition of 
    \Cref{Equation: coherence condition of Sotwo lax hermitian fixed points} corresponds precisely to \Cref{Diagram: SO2volution 3}. 
\end{example}
\begin{example}
    Recall the 2-contravariant $\Sotwo$-volutive 3-category $\AmAdj[1]$ from \Cref{Corollary: the categorified ambidextrous adjunction admits an Sotwo volutive structure}. 
    Every object in $\AmAdj[1]$ admits the structure of an (honest) hermitian fixed point, hence also the structure of a lax hermitian fixed point. 
\end{example}

\section{Closed 2-categories}\label{Section: Closed 2-categories}
In this section, we discuss (fully) closed symmetric monoidal 2-categories and lax volutive structures. We highlight key difference to the theory 
of rigid symmetric monoidal 2-categories, such as the fact that 1-closedness leads only to an oplax 2-functor, whereas having all adjoints leads to an honest 
2-functor (which in addition is invertible). A first instance of this phenomenonen we have seen earlier, noting that the functor underlying the canonical 
lax $\Oone$-volutive structure of a closed symmetric monoidal category is only oplax monoidal. We will illustrate our theory at the example of the Morita 2-category 
of a sufficiently nice closed symmetric monoidal category, as well as the 2-category of Profunctors.
\subsection{Lax 2-functors and 2-transformations}\label{Subsection: lax 2-functors and transformations}
In this section we review the notions of lax 2-functors and 2-transformations for the convenience of the reader, and to establish conventions.
\begin{definition}
    Let $\B$ and $\B'$ be 2-categories. A \emph{lax 2-functor} $F \colon \B \to \B'$ assigns to each object $a \in \B$ an object $F(a) \in \B'$, to each 
    1-morphism $X \colon a \to b$ in $\B$ a 1-morphism $F(X) \colon F(a) \to F(b)$ in $\B'$, to each 2-morphism $f \colon X \to Y$ in $\B$ a 2-morphism 
    $F(f) \colon F(X) \to F(Y)$ in $\B'$, to each pair of 1-morphisms $X \colon a \to b$ and $Y \colon b \to c$ in $\B$ a 2-morphism 
    $F_{Y,X} \colon F(Y) \circ F(X) \to F(Y \circ X)$ in $\B'$, and to each object $a \in \B$ a 2-morphism $F_a \colon \id_{F(a)} \to F(\id_a)$ in $\B'$, 
    subject to the coherence conditions of lax associativity and lax (left and right) unity, spelled out in \cite[Definition 4.1.2]{johnson20202dimensionalcategories}.
\end{definition}
\begin{definition}
    A lax 2-functor is called \emph{normal} (or \emph{unitary}) if each of the 2-morphisms $F_a \colon \id_{F(a)} \to F(\id_a)$ is a 2-isomorphism. 
    If moreover each $F_{Y,X}$ is a 2-isomorphism, we call $F$ a \emph{pseudofunctor} or simply \emph{2-functor}. 
\end{definition}
\begin{definition}
    Let $F,G \colon \B \to \B'$ be lax 2-functors. A \emph{lax 2-transformation} $\alpha \colon F \Rightarrow G$ consists of 1-morphism components 
    $\alpha_c \colon F(c) \to G(c)$ for each $c \in \C$ and 2-morphism components $\alpha_X \colon G(X) \circ \alpha_c \Rightarrow \alpha_d \circ F(X)$ for 
    $X \colon c \to d$ a 1-morphism in $\C$ satisfying compatibility with units and composition of 1-morphisms, and naturality with respect to 2-morphisms, 
    spelled out in \cite[Definition 4.2.1]{johnson20202dimensionalcategories}.
\end{definition}
\begin{definition}
    A lax 2-transformation is called a \emph{pseudonatural transformation} or simply \emph{2-transformation} if each $\alpha_X$ is a 2-isomorphism.
\end{definition}
\begin{variant}
    Both lax 2-functors and lax 2-transformations admit oplax versions, in which the structure 2-morphisms $F_a,F_{Y,X}$ and the 2-morphisms $\alpha_X$ 
    are reversed, respectively. We emphasize that one may consider oplax 2-transformations between lax 2-functors, as discussed in 
    \cite[Section 4.3]{johnson20202dimensionalcategories} (where oplax is called colax).
\end{variant}
\begin{remark}
    Let $\B$ and $\B'$ be 2-categories. Then, there is a 2-category $\Fun(\B,\B')$ whose objects are lax 2-functors, whose 1-morphisms are lax 2-transformations, 
    and whose 2-morphisms are modifications, see \cite[Definition 4.4.10 \& Theorem 4.4.11]{johnson20202dimensionalcategories}.
\end{remark}
\begin{remark}
    Let $\B,\B',\B''$ be 2-categories and $F \colon \B \to B'$ and $G \colon \B' \to \B'$ lax 2-functors. Then, we may compose $G$ and $F$ to obtain a lax 
    2-functor $G \circ F \colon \B \to \B'$, see \cite[Definition 4.1.26 \& Lemma 4.1.29]{johnson20202dimensionalcategories}. Moreover, there is a category 
    whose objects are 2-categories and whose morphisms are lax 2-functors, see \cite[Theorem 4.1.30]{johnson20202dimensionalcategories}.
\end{remark}
\begin{remark}
    For each lax 2-functor $F \colon \B \to \B'$, there is a \emph{1-opposite} lax 2-functor $F^{1\opp} \colon \B^{1\opp} \to \B^{1\opp}$ defined 
    in \cite[Example 4.1.10]{johnson20202dimensionalcategories}. On the other hand, there is a \emph{2-opposite} oplax 2-functor 
    $F^{2\opp} \colon \B^{2\opp} \to \B^{2\opp}$, and a \emph{(1,2)-opposite} oplax 2-functor $F^{(1,2)\opp} \colon \B^{(1,2)\opp} \to \B^{(1,2)\opp}$.
\end{remark}
\begin{remark}
    The 1-categories of 2-categories and lax (respectively oplax) 2-functors are equivalent via the 2-opposite construction. 
\end{remark}
\begin{remark}\label{Remark: oplax 2-functors and local functors}
    Note that, by definition, an (op)lax 2-functor $\B \to \B'$ induces honest functors on Hom-category level. The (op)laxness is reflected only 
    in the compability of these induced functors with units and composition.\\
    Similarly, an (op)lax natural transformation between (op)lax 2-functors induces natural transformations between the the induced functors on 
    Hom-category level; the (op)laxness is again reflected in the compatibility with respect to units and composition.\\
    Given a lax/oplax 2-functor $F \colon \B \to \B'$ which induces functors $F_{a,b} \colon \Hom_{\B}(a,b) \to \Hom_{\B'}(F(a),F(b))$, the 2-opposite 
    oplax/lax 2-functor $F \colon \B^{2\opp} \to (\B')^{2\opp}$ induces on the respective Hom-categories the functors 
    $F_{a,b}^{\opp} \colon \Hom_{\B}(a,b)^{\opp} \to \Hom_{\B'}(F(a),F(b))^{\opp}$.
\end{remark}

\subsection{1-closed 2-categories}\label{Subsection: 1-closed 2-categories}
In this section we discuss the basic theory of 1-closed 2-categories. 
\begin{definition}\label{Definition: closed 2-categories}
    Let $\B$ a 2-category. We will say that $\B$ is \emph{right 1-closed} if for every object $c \in \C$ and every 1-morphism $X \colon a \to b$ in $\B$, 
    the functor $(-) \circ X \colon \Hom_{\B}(b,c) \to \Hom_{\B}(a,c)$ admits a right adjoint. Similarly, we will say that $\B$ is \emph{left 1-closed} if 
    for every object $c \in \C$ and every 1-morphism $X \colon a \to b$ in $\B$, the functor $X \circ (-) \colon \Hom_{\B}(c,a) \to \Hom_{\B}(c,b)$ 
    admits a right adjoint. We will say that $\B$ is \emph{1-closed} if it is both left and right 1-closed.
\end{definition}
\begin{example}
    Any 2-category $\B$ with left/right adjoints is left/right 1-closed.
\end{example}
\begin{variant}
    If the compostion functors in \Cref{Definition: closed 2-categories} admit left adjoints rather than right adjoints, we call $\B$ \emph{right/left 1-coclosed}.
\end{variant}
\begin{lemma}\label{Lemma: opposites of closed 2-categories}
    The 1-opposite of a 1-closed 2-category is 1-closed again. The 2-opposite of a 1-closed 2-category is 1-coclosed. 
\end{lemma}
\begin{remark}
    \Cref{Lemma: opposites of closed 2-categories} suggests that passing to the 2-opposite is not a (classical) $\Z_2$-symmetry of the (higher) category of 1-closed 
    2-categories. This is substentially different than in the case of 2-categories (with adjoints). In particular, while 2-categories with adjoints have 
    more symmetry than ordinary 2-categories, 1-closed 2-categories in contrast have less symmetry. 
\end{remark}
\begin{remark}
    We wish to spell out some of the data of a right 1-closed 2-category. Let $\B$ be a right 1-closed 2-category and $X \colon a \to b$ be a 1-morphism. 
    Then, $(-) \circ X \colon \Hom(b,c) \to \Hom(a,c)$ has a right adjoint which we denote by $(-)^X \colon \Hom(a,c) \to \Hom(b,c)$. By definition, 
    we have for any 1-morphism $Z \colon a \to c$ a 2-morphism $\ev_X^Z \colon Z^X \circ X \to Z$ so that for each $Y \colon b \to c$ we have a
    natural bijection $\hom_{\B}(Y,Z^X) \cong \hom_{\B}(Y \circ X,Z)$. \\
    Let us now assume that $\B$ is a left 1-closed 2-category and $X \colon a \to b$ is a 1-morphism. 
    Then, $X \circ (-) \colon \Hom(c,a) \to \Hom(c,b)$ has a right adjoint, denoted ${}^{X}\!(-) \colon \Hom(c,b) \to \Hom(c,a)$. By definition, we have 
    for any 1-morphism $Z \colon c \to b$ a 2-morphism $\widetilde{\ev}_X^Z \colon X \circ {}^{X}\!(Z) \to Z$ so that for each $Y \colon c \to a$ we have 
    a natural bijection $\hom_{\B}(Y,{}^{X}\!Z) \cong \hom_{\C}(X \circ Y,Z)$.
\end{remark}
\begin{construction}\label{Construction: 2-functor from closedness}
    Let $\B$ be a right 1-closed 2-category and let $a,b,c \in \B$ be objects. First, note that the assignment $\Hom(a,b)^{\opp} \times \Hom(a,c) \to \Hom(b,c)$ 
    that assigns $(X,Y) \mapsto Y^X$ is functorial. We define an oplax 2-functor $\mathcal{J} \colon \B \to \B^{(1,2)\opp}$ by setting 
    $\mathcal{J}(a) = a, \mathcal{J}(X) = \id_a^X$, and $\mathcal{J}(f) := \id_a^f$ for an object $a \in \B$, 1-morphisms $X,Y \colon a \to b$, and a 
    2-morphism $f \colon X \to Y$ in $\B$.\\
    For the sake of completeness, we also describe the coherence 2-morphisms $\id^X \circ \id^Y \to \id^{Y \circ X}$ and $\id_a \to \id_a^{\id_a}$ in $\B$ of our 
    oplax\footnote{Note that oplax functoriality data in $\B^{(1,2)\opp}$ resembles lax functoriality data in $\B$.} 2-functor $\mathcal{J}$. The former is constructed 
    as the preimage of the 2-morphism 
    \begin{equation}
        \xymatrix@C=4em{
            \id^X \circ \id^Y \circ Y \circ X \ar[r]^-{\id \circ \ev_Y \circ \id} & \id^X \circ X \ar[r]^-{\ev_X} & \id
        }
    \end{equation}
    under the bijection $\hom(\id^X \circ \id^Y,\id^{Y \circ X}) \cong \hom(\id^X \circ \id^Y \circ Y \circ X,\id)$ and the latter is constructed
    as the preimage of the 2-morphism $\id_{\id_a}$ under the bijection $\hom(\id_a,\id_a^{\id_a}) \cong \hom(\id_a \circ \id_a, \id_a) \cong \hom(\id_a,\id_a)$ 
    where we have used the right closedness of $\B$ in both cases.
\end{construction}
\begin{example}
    Let $\C$ be a closed monoidal category, which we may interpret as a 1-closed 2-category with a single object. 
    Then, the duality functor $d \colon \C \to \C^{\opp}_{\otimes^{\opp}}$ is oplax monoidal, hence induces an oplax 2-functor upon delooping. 
\end{example}
\begin{remark}
    In the case of a 2-category with right adjoints, \Cref{Construction: 2-functor from closedness} reduces to the adjunction 2-functor 
    $(-)^R \colon \B \to \B^{(1,2)\opp}$. 
\end{remark}
\begin{notation}
    \Cref{Construction: 2-functor from closedness} carries over to left 1-closed 2-category with the evident modifications. If we consider a 2-category 
    which is both left and right 1-closed, we will employ the notation $\mathcal{J}_L$ and $\mathcal{J}_R$ to distinguish the two respective oplax 2-functors.
\end{notation}
\begin{remark}\label{Remark: double closed functor does not exist}
    Let $\B$ be a 1-closed 2-category. For any two 1-morphisms $Y \colon a \to b$ and $X \colon b \to c$ in $\B$, we have 2-morphisms 
    \begin{equation}
        \xymatrix{
            \id^{\id^X} \circ \id^{\id^Y} \ar[r]^-{} & \id^{\id^Y \circ \id^X} & \id^{\id^{X \circ Y}} \ar[l]
        }
    \end{equation}
    fitting into the commutative diagram of 2-morphisms
    \begin{equation}
        \begin{aligned}
        \xymatrix{
            \id^{\id^X} \circ \id^{\id^Y} \circ (\id^Y \circ \id^X) \ar[r]^-{\id \circ \ev_{\id^Y} \circ \id} \ar[d]_-{} & \id^{\id^X} \circ \id^X \ar[d]^-{\ev_{\id^X}} \\
            \id^{\id^Y \circ \id^X} \circ (\id^Y \circ \id^X) \ar[r]^-{\ev_{\id^Y \circ \id^X}} & \id \\
            \id^{\id^{X \circ Y}} \circ (\id^Y \circ \id^X) \ar[u]_-{} \ar[r]^-{} & \id^{\id^{X \circ Y}} \circ (\id^{X \circ Y}) \ar[u]_-{\ev_{\id^{X \circ Y}}}
        }
        \end{aligned}
    \end{equation}
    where for the lower horizontal arrow we have used the oplax functoriality data of $\mathcal{J}$. Generically, the directions of the 2-morphisms 
    described above indicate that the assignment $X \mapsto \id^{\id^X}$ does not assemble into an (op)lax 2-functor $\B \to \B$.
\end{remark}
\begin{remark}
    In the case of a 2-category with adjoints, we do have a 2-functor $(-)^{RR} \colon \B \to \B$.
\end{remark}
\begin{variant}
    Let $\B$ be a 1-closed 2-category. We have 2-morphisms 
    \begin{equation}
        \xymatrix{
            \id^{^{X}\!\id} \circ \id^{^{Y}\!\id} \ar[r]^-{} & \id^{^{Y}\!\id \circ ^{X}\!\id} & \id^{^{X \circ Y}\!\id} \ar[l]
        }
    \end{equation}
    and similarly for the other possible combinations of $\mathcal{J}_L$ and $\mathcal{J}_R$. 
\end{variant}
\begin{remark}\label{Remark: maps into the left right closed 1-morphisms}
    Let $\B$ be a 1-closed 2-category. Let $X \colon a \to b$ be a 1-morphism in $\B$. We have a canonical 2-morphism
    \begin{equation}
        X \to \id^{^{X}\!\id}
    \end{equation}
    induced by the 2-morphism $\widetilde{\ev}_X \colon X \circ ^{X}\!\id \to \id$. Since the assignment $X \mapsto \id^{^{X}\!\id}$ 
    does not necessarily define an (op)lax 2-functor, it is not reasonable to ask whether the 2-morphisms we described here assemble into an (op)lax 2-transformation. 
    Completely analogously, there is a 2-morphism $X \to ^{\id^X}\!\id$.\\
    Let $f \colon X \to Y$ be a 2-morphism. We claim that the diagram of 2-morphisms 
    \begin{equation}
        \begin{aligned}
        \xymatrix{
            X \ar[r]^-{} \ar[d]^-{f} & \id^{^{X}\!\id} \ar[d]^-{\id^{^{f}\!\id}} \\
            Y \ar[r]^-{} & \id^{^{Y}\!\id}
        }
        \end{aligned}
    \end{equation}
    is commutative. Indeed, the two 2-morphisms $X \to \id^{^{Y}\!\id}$ induce two 2-morphisms $X \to ^{Y}\!\id$ which coincide by the commutativity of the 
    diagram 
    \begin{equation}
        \begin{aligned}
        \xymatrix{
            X \circ ^{Y}\!\id \ar[r]^-{\id \circ ^{f}\!\id} \ar[d]^-{f \circ \id} & X \circ ^{X}\!\id \ar[d]^-{\ev_X} \\
            Y \circ ^{Y}\!\id \ar[r]^-{\ev_Y} & \id.
        }
        \end{aligned}
    \end{equation}
    In other words, the 2-morphisms $X \to \id^{^{X}\!\id}$ are natural in $X$.
\end{remark}
\begin{remark}\label{Remark: right and left closedness for self-oneopposite 2-categories}
    Let $\B$ be a right 1-closed 2-category together with an equivalence of 2-categories $d \colon \B \to \B^{2\opp}$. 
    Then, $\B$ is left 1-closed with 
    \begin{equation}
        ^{X}\!\id = d^{-1}(\id^{d(X)})
    \end{equation}
    for any 1-morphism $X \colon a \to b$ in $\B$.
\end{remark}

\subsection{0-closed 2-categories}
In this section we consider the basic theory of 0-closed monoidal categories.
\begin{definition}
    Let $\B$ be a monoidal 2-category. We will say that $\B$ is \emph{right 0-closed} if the 2-functor $(-) \otimes a \colon \B \to \B$ admits a right adjoint 
    for every $a \in \B$. Similarly, we will say that $\B$ is \emph{left 0-closed} if the 2-functor $a \otimes (-) \colon \B \to \B$ admits a right adjoint 
    for every $a \in \B$. We will say that $\B$ is \emph{0-closed} if it is both left and right 0-closed.\\
    We will say that a symmetric monoidal 2-category $\B$ is \emph{0-closed} if its underlying monoidal 2-category is right (equivalently: left) 0-closed.
\end{definition}
\begin{remark}
    There are weaker versions of 0-closedness for monoidal 2-categories, as there are weaker versions of adjunctions between 2-functors. These will be of 
    no concern in the following.
\end{remark}
\begin{definition}
    Let $\B$ be a 2-category. A \emph{lax 1-contravariant $\Oone$-volutive structure} on $\B$ consists of a 2-functor $d \colon \B \to \B^{1\opp}$, a 
    2-transformation $\eta \colon \id \to d^{1\opp} \circ d$, and an invertible modification $\tau \colon (\id \circ \eta^{1\opp}) \bullet (\eta \circ \id) \to \id$
    satisfying the evident coherence condition over the fourfold product of $d$.
\end{definition}
\begin{remark}
    Despite the choice of terminology, a lax 1-contravariant $\Oone$-volutive structure consists entirely of pseudofunctors, transformations, etc. 
    As in the 1-categorical case, the terminology is motivated as a separation from honest $\Oone$-volutive structures and to emphasize the (op)lax monoidality of 
    the underlying functor.
\end{remark}
The proof of the following assertion is similar to its 1-categorical analogue and its rigid 2-categorical counterpart; we will be brief.
\begin{proposition}\label{Proposition: closed symmetric monoidal lax volutive twocats}
    Let $\B$ be a 0-closed symmetric monoidal 2-category. Then, $\B$ admits a lax 1-contravariant $\Oone$-volutive structure.
\end{proposition}
\begin{proof}
    First, we note that we have a 2-functor $\B^{1\opp} \times \B \to \B, (a,b) \mapsto b^a$. To a 1-morphism $(X,Y) \colon (a,b) \to (\hat{a},\hat{b})$, it assigns 
    the 1-morphism $Y^X \colon b^a \to \hat{b}^{\hat{a}}$ corresponding under the equivalence of categories $\Hom_{\B}(b^a,\hat{b}^{\hat{a}}) \cong 
    \Hom_{\B}(b^a \otimes \hat{a},\hat{b})$ to the 1-morphism 
    \begin{equation}
        \xymatrix{
            b^a \otimes \hat{a} \ar[r]^-{\id \otimes X} & b^a \otimes a \ar[r]^-{\ev_a^b} & b \ar[r]^-{Y} & \hat{b}.
        }
    \end{equation}
    To a 2-morphism $(f,g) \colon (X,Y) \to (\overline{X},\overline{Y})$, it assigns the 2-morphism $g^f \colon Y^X \to \overline{Y}^{\overline{X}}$ 
    corresponding to the 2-morphism 
    \begin{equation}
        \begin{aligned}
        \xymatrix@C=3em{
            b^a \otimes \hat{a} \ar[r]^-{\id \otimes X} \ar[d]^-{\id} & b^a \otimes a \ar[r]^-{\ev_a^b} \ar[d]^{\id} \ar@{=>}[dl]^-{\id \otimes f}
            & b \ar[d]^-{\id}\ar[r]^-{Y} \ar@{=>}[dl]^-{\id} & \hat{b} \ar[d]^-{\id} \ar@{=>}[dl]^-{g}\\
            b^a \otimes \hat{a} \ar[r]_-{\id \otimes \overline{X}} & b^a \otimes a \ar[r]_-{\ev_a^b} & b \ar[r]_-{\overline{Y}} & \hat{b}.
        }
        \end{aligned}
    \end{equation}
    One checks that this defines a 2-functor. 
    We now define a 2-functor $d \colon \B \to \B^{1\opp}$ by assigning $a \mapsto 1^a, X \mapsto 1^X, f \mapsto 1^f$, that is, we consider the composition 
    $\B \cong \B \times \{\star\} \to \B \times \B^{1\opp} \to \B^{1\opp}$ where the last 2-functor is the 1-opposite of the one constructed above and 
    the 2-functor $\{\star\} \to \B^{1\opp}$ is determined by its image, $1$. \\
    Next, we wish to describe a 2-transformation $\eta \colon \id \to d^{1\opp} \circ d$. Its 1-morphism component at $a \in \B$ is given by 
    the 1-morphism $\eta_a \colon a \to 1^{1^a}$ corresponding to the 1-morphism $a \otimes 1^a \cong 1^a \otimes a \to 1$ where we have used the braiding of $\B$. \\   
    Next, we wish to describe a modification $\tau \colon (\id \circ \eta^{1\opp}) \bullet (\eta \circ \id) \to \id$. Its 2-morphism component at $a \in \B$ 
    is a 2-isomorphism $1^\eta_a \circ \eta_{1^a} \to \id_{1^a}$ which is constructed completely analogously to the 1-categorical case (in which 
    this is a coherence condition). One checks that this satisfies the required coherence condition.
\end{proof}

\subsection{Fully closed 2-categories}
In this section we study the basic theory of (fully) closed monoidal 2-categories.
\begin{definition}
    Let $\B$ be a (symmetric) monoidal 2-category. We will say that $\B$ is \emph{(fully) closed} if $\B$ is both 0- and 1-closed. 
\end{definition}
In the following we wish to understand the interaction between the 0- and 1-closedness (oplax) 2-functors $d$ and $\mathcal{J}$ on a fully closed monoidal 
2-category. We start by working in slightly more generality. 
\begin{construction}\label{Construction: Closed structures are compatible with everything}
    Let $\B,\B'$ be right 1-closed 2-categories and let $F \colon \B \to \B'$ be a 2-functor. Then, we obtain an oplax 2-transformation 
    $\xi^F \colon \mathcal{J}_R \circ F \to F^{(1,2)\opp} \circ \mathcal{J}_R$ as follows: the 1-morphism components are trivial, 
    and the 2-morphism component at $X \colon a \to b$ is given by the image of the 2-morphism $\ev_X \colon \id^X \circ X \to \id$ under the map 
    \begin{equation}\label{Equation: Closed structures are compatible with everything}
        \xymatrix{
            \hom(\id^X \circ X,\id) \ar[r]^-{F} & \hom(F(\id^X) \circ F(X),\id) \ar[r]^-{\psi^{-1}} & \hom(F(\id^X),\id^{F(X)})
        }
    \end{equation}
    induced by $F$ and the adjunction defining the right 1-closed structure. In other words, the component is given by 
    $\xi^F_X := \psi^{-1}(F(\ev_X)) \colon F(\id^X) \to \id^{F(X)}$ (where we have left the structure 2-morphism of the 2-functor $F$ implicit). 
    Since the 1-morphism components of this 2-transformation are trivial, it can equivalently be understood as a lax 2-transformation 
    $F^{(1,2)\opp} \circ \mathcal{J}_R \to \mathcal{J}_R \circ F$ which we denote by the same symbol.
\end{construction}
\begin{remark}
    The construction of the 2-morphisms $\xi^F_X$ in \Cref{Construction: Closed structures are compatible with everything} is completely 
    analogous under the weaker assumption that $F$ is a lax 2-functor. In this case however, $F^{(1,2)\opp}$ is an oplax 2-functor so that 
    the composition of $F$ with $\mathcal{J}_R$ is not (globally) well-defined. This is similar to the phenomenona observed in 
    \Cref{Subsection: 1-closed 2-categories}.
\end{remark}
\begin{corollary}
    Let $\B$ be a (fully) closed monoidal 2-category. We have an oplax 2-transformation $(\mathcal{J}_R)^{1\opp} \circ d \to 
    d^{(1,2)\opp} \circ \mathcal{J}_R$.
\end{corollary}

\subsection{A local perspective on closedness}\label{Subsection: local perspective on closedness}
In this section we develop a local perspective on closedness. We start with a motivation.
\begin{remark}
    Let $\B$ be a 1-closed symmetric monoidal 2-category. We have observed earlier that the oplax 2-functor $\mathcal{J} \colon \B \to \B^{(1,2)\opp}$ 
    induces for each pair of objects $a,b \in \B$ a functor $\Hom_{\B}(a,b) \to \Hom_{\B}(b,a)^{\opp}$. While the 2-opposite 
    $\mathcal{J}^{(1,2)\opp} \colon \B^{(1,2)\opp} \to \B$ is a lax 2-functor, so that it cannot immediately be composed with $\mathcal{J}$, it induces 
    a functor $\Hom_{\B}(b,a)^{\opp} \to \Hom_{\B}(a,b)$ so that we may compose these two (op)lax 2-functors at least locally. The primary intuition 
    behind the constructions in this section is that this local perspective is sufficient for some purposes.
\end{remark}
\begin{construction}\label{Construction: Oone volutive to Oone dagger}
    Let $(\B,d,\eta,\tau)$ be a lax 1-contravariant $\Oone$-volutive 2-category. Then, the full sub 2-category on those objects $a \in \B$ for which $\eta_a$ is 
    a 1-isomorphism inherits a 1-contravariant $\Oone$-volutive structure. This 1-contravariant $\Oone$-volutive 2-category we can then turn into a 
    1-contravariant $\Oone$-dagger 2-category $\hat{B}$ by applying the construction of \cite[Section 2.3.3]{carqueville2025orbifoldshigherdaggerstructures}. 
    To be more explicit, the objects of the 2-category $\hat{B}$ are triples $(a,\theta_a,\omega_a)$ consisting of an object $a \in \hat{B}$, a 1-isomorphism 
    $\theta_a \colon a \to d(a)$, and a 2-isomorphism $\omega_a \colon d(\theta_a) \circ \eta_a \to \theta_a$ satisfying a coherence condition over $d^3(a)$, for 
    which we refer to \cite[Section 2.3.3]{carqueville2025orbifoldshigherdaggerstructures}. The 1- and 2-morphisms of $\hat{B}$ are the same as those of $\B$.
    The functor underlying the $\Oone$-dagger structure assigns to a 1-morphism $X \colon (a,\theta_a,\omega_a) \to (b,\theta_b,\omega_b)$ the 1-morphism 
    given by
    \begin{equation}
        \xymatrix{
            b \ar[r]^-{\theta_b} & d(b) \ar[r]^-{d(X)} & d(a) \ar[r]^-{\theta_a^{-1}} & a.
        }
    \end{equation}
\end{construction}
\begin{lemma}\label{Lemma: Oone daggerification is still closed}
    Let $\B$ be a fully closed symmetric monoidal 2-category. The 1-contravariant $\Oone$-dagger 2-category constructed by combining 
    \Cref{Proposition: closed symmetric monoidal lax volutive twocats} and \Cref{Construction: Oone volutive to Oone dagger} 
    is 1-closed.
\end{lemma}
\begin{proof}
    This follows from the observation that the Hom-categories and compositions of $\hat{B}$ described in \Cref{Construction: Oone volutive to Oone dagger} 
    are the same as those of $\B$.
\end{proof}
\begin{construction}\label{Construction: towards lax volutive structures on Hom-categories of sufficiently closed dagger 2-category}
    Let $\B$ be a 1-closed 1-contravariant $\Oone$-dagger 2-category. In the following we will make extensive use of \Cref{Remark: oplax 2-functors and local functors}.
    By composition we obtain an oplax 2-functor $d \circ \mathcal{J}_R \colon \B \to \B^{2\opp}$
    which induces a functor $\overline{d}_{a,b} \colon \Hom_{\B}(a,b) \to \Hom_{\B}(a,b)^{\opp}$ for each pair of objects $a,b\in \B$. On the other hand, the 
    2-opposite $(d \circ \mathcal{J}_R)^{2\opp} \colon \B^{2\opp} \to \B$ (which is a lax 2-functor) induces a functor 
    $\overline{d}_{a,b}^{\opp} \colon \Hom_{\B}(a,b)^{\opp} \to \Hom_{\B}(a,b)$ so that we may consider the composition 
    \begin{equation}
        \overline{d}_{a,b}^{\opp} \circ \overline{d}_{a,b} \colon \Hom_{\B}(a,b) \to \Hom_{\B}(a,b).
    \end{equation}
    Now, by \Cref{Remark: right and left closedness for self-oneopposite 2-categories}, we compute
    \begin{equation}
        (\overline{d}_{a,b}^{\opp} \circ \overline{d}_{a,b})(X) = ^{\id^X}\!\id
    \end{equation}
    and similar for 2-morphisms. In particular, we have a natural transformation 
    $\overline{\eta}_{a,b} \colon \id \to \overline{d}_{a,b}^{\opp} \circ \overline{d}_{a,b}$ whose component at $X \in \Hom_{\B}(a,b)$ is given by 
    the canonical morphism $X \to ^{\id^X}\!\id$ described in \Cref{Remark: maps into the left right closed 1-morphisms}. 
\end{construction}
\begin{lemma}\label{Lemma: lax volutive structures on Hom-category level}
    In the situation of \Cref{Construction: towards lax volutive structures on Hom-categories of sufficiently closed dagger 2-category}, 
    the natural transformation $\overline{\eta}_{a,b} \colon \id \to \overline{d}_{a,b}^{\opp} \circ \overline{d}_{a,b}$ satisfies
    $\overline{d}_{a,b}((\overline{\eta}_{a,b})_X) \circ (\overline{\eta}_{a,b})_{\overline{d}_{a,b}(X)} = \id_{\overline{d}_{a,b}(X)}$ for all 
    $X \in \Hom_{\B}(a,b)$. In particular, $(\overline{d}_{a,b},\overline{\eta}_{a,b})$ defines a lax $\Oone$-volutive structure on $\Hom_{\B}(a,b)$. 
\end{lemma}
\begin{proof}
    Essentially the same as in the non-lax case.
\end{proof}
\begin{construction}
    In the situation of \Cref{Lemma: lax volutive structures on Hom-category level}, let $a,b,c$ be objects in $\B$ and consider the composition functor 
    $\Hom(b,c) \times \Hom(a,b) \to \Hom(a,c)$ in $\B$. Both sides carry lax $\Oone$-volutive structures, the former of which is given as the product (cf. 
    \Cref{Definition: product of lax Oone volutive categories}) of the ones on $\Hom(a,b)$ and $\Hom(b,c)$. We wish to construct a lax $\Oone$-volutive structure 
    on the composition functor. First, recall that the functors underlying the lax $\Oone$-volutive structures assign $(Y,X) \in \Hom(b,c) \times \Hom(a,b)$ 
    to $(d(\id^Y),d(\id^X))$ and $Y \circ X \in \Hom(a,c)$ to $d(\id^{Y \circ X})$, respectively. By definition, we need to construct a natural transformation 
    fitting into the diagram 
    \begin{equation}
        \begin{aligned}
        \xymatrix{
            \Hom(b,c) \times \Hom(a,b) \ar[d]^-{} \ar[r]^-{\circ} & \Hom(a,c) \ar@{=>}[dl]^-{\alpha} \ar[d]\\
            \Hom(b,c)^{\opp} \times \Hom(a,b)^{\opp} \ar[r]^-{\circ}  & \Hom(a,c)^{\opp}.
        }
        \end{aligned}
    \end{equation}
    The component of $\alpha$ at $(Y,X) \in \Hom(b,c) \times \Hom(a,b)$ we define to be the canonical 2-morphism 
    \begin{equation}
        \xymatrix{
        d(\id^Y) \circ d(\id^X) \ar[r]^-{\cong} & d(\id^X \circ \id^Y) \ar[r]^-{} & d(\id^{Y \circ X}).
        }
    \end{equation}
    expressing the (oplax) functoriality of $d$ and $\mathcal{J}$. One checks that this defines a lax $\Oone$-volutive structure on the composition functor.\\
    Let us also construct for each object $a \in \B$ a lax $\Oone$-volutive functor $\star \to \Hom(a,a)$. The underlying functor we may take to be the constant 
    functor with value $\id_a \in \Hom(a,a)$. The lax $\Oone$-volutive structure is given by the natural transformation fitting in the diagram 
    \begin{equation}
        \begin{aligned}
        \xymatrix{
            \star \ar[d] \ar[r]^-{} & \Hom(a,a) \ar@{=>}[dl]^-{\lambda}  \ar[d]^-{\overline{d}_{a,a}} \\
            \star^{\opp} \ar[r]^-{} & \Hom(a,a)^{\opp}
        }
        \end{aligned}
    \end{equation}
    assigning the unique object of $\star$ to the canonical 2-morphism 
    \begin{equation}
        \id_a \to d(\id^{\id_a})
    \end{equation}
    expressing the (oplax) functoriality of $d$ and $\mathcal{J}$.
\end{construction}
We deduce the following result.
\begin{proposition}\label{Proposition: lax enrichment}
    Let $\B$ be a 1-closed 1-contravariant $\Oone$-dagger 2-category. Then, $\B$ is enriched in lax $\Oone$-volutive categories. 
\end{proposition}
\begin{remark}
    While categories enriched in $\Oone$-volutive categories are the same as 2-contravariant $\Oone$-dagger 2-categories 
    (see \Cref{Remark: 2-contra dagger 2-cats and Oone volutive enrichments}), we are not aware of a result 
    like this identifying categories enriched in lax $\Oone$-volutive categories with some flavour of ``classical'' dagger structure. We stress again that 
    simply asking only for lax ingredients in the definition of a 2-contravariant $\Oone$-volutive/dagger structure leads to immediate problems, due to the 
    fact that the 2-opposite changes the variance.
\end{remark}
As a consequence of \Cref{Proposition: lax enrichment}, we may perform the following. 
\begin{construction}\label{Passing to lax hermitian fixed points}
    Let $\B$ be a 1-closed 1-contravariant $\Oone$-dagger 2-category. Since $\B$ is enriched in lax $\Oone$-volutive categories, we may locally pass to 
    the categories of lax hermitian fixed points and lax isometries. This procedure is compatible with compsitions and units by the functoriality of 
    passing to $\LaxHerm$ as well as \Cref{Proposition: lax enrichment}. Explicitly,  we may associate to the lax $\Oone$-volutive category
    $(\Hom_{\B}(a,b),\overline{d}_{a,b},\overline{\eta}_{a,b})$ its category of lax hermitian fixed points 
    $\LaxHerm((\Hom_{\B}(a,b),\overline{d}_{a,b},\overline{\eta}_{a,b}))$. Here, an object in this category consists 
    of a 1-morphism $X \colon a \to b$ together with a 2-morphism $\theta_X \colon X \to \overline{d}_{a,b}(X)$ satisfying $\overline{d}_{a,b}(\theta_X) \circ 
    (\overline{\eta}_{a,b})_X = \theta_X$, and morphisms are lax isometries. To see the compatibility with compositions and units explicitly, let 
    $(X,\theta_X)$ be a lax hermitian fixed point in $(\Hom_{\B}(a,b),\overline{d}_{a,b},\overline{\eta}_{a,b})$ and 
    let $(Y,\theta_Y)$ be a lax hermitian fixed point in $(\Hom_{\B}(b,c),\overline{d}_{b,c},\overline{\eta}_{b,c})$. We claim that $Y \circ X$ admits the 
    structure of a lax hermitian fixed point in $(\Hom_{\B}(a,c),\overline{d}_{a,c},\overline{\eta}_{a,c})$. Indeed, we have a 2-morphism 
    \begin{equation}
        \xymatrix{
        Y \circ X \ar[r]^-{\theta_Y \circ \theta_X} & d(\id^Y) \circ d(\id^X) \ar[r]^-{\cong} & d(\id^X \circ \id^Y) \ar[r]^-{} & d(\id^{Y \circ X}).
        }
    \end{equation}
    Similarly, for any object $a$ the identity 1-morphism $\id_a$ admits the structure of a lax hermitian fixed point given by 
    the canonical 2-morphism $\id_a \to d(\id^{\id_a})$.
\end{construction}
\begin{remark}
    It is evident from the construction that the composition of two (honest) hermitian fixed points does not need to be an (honest) hermitian fixed point, 
    as the last morphism in the construction is generically non-invertible (as part of the oplax functoriality data of $\mathcal{J}$). This is illustrated further 
    in \Cref{Subsection: Bornological algebras}. On a structural level, this is a reflection of the fact that lax $\Oone$-volutive functors between lax 
    $\Oone$-volutive categories do not induce ($\Oone$-volutive/dagger) functors between the associated $\Oone$-volutive/dagger categories. 
\end{remark}
\begin{remark}
    Summarizing the results of this section, from each fully closed symmetric monoidal 2-category we may extract a 1-contravariant $\Oone$-dagger 2-category 
    enriched in lax $\Oone$-volutive categories, whose lax hermitian fixed points we may study locally.
\end{remark}

\subsection{Example: Morita 2-categories}\label{Subsection: Morita 2-categories}
In this section, we discuss duality/closedness in Morita 2-categories. We first establish some notation, loosely following \cite{dutta2025moritainfty}. 
Let $\C$ be a monoidal category. Let $A,B,C$ be monoids in $\C$ and let $M$ be an $A$-$B$-bimodule and $N$ be a $B$-$C$-bimodule. 
The \emph{balanced tensor product} of $M$ and $N$ over $B$ is the coequalizer (if it exists)
\begin{equation}
    M \otimes_B N := \coeq \left( M \otimes B \otimes N \rightrightarrows M \otimes N \right)
\end{equation}
where the two maps are induced by the left and right $B$-actions on $M$ and $N$, respectively, supressing associators. We will say that a monoidal category
$\C$ admits (a calculus of) balanced tensor products if all balanced tensor products exist and the respective coequalizers are preserved by the monoidal 
structure, see \cite[Definition 2.7]{dutta2025moritainfty}. We then obtain a 2-category\footnote{Recall that we use the term 2-category synonymously for bicategory. 
In \cite{dutta2025moritainfty} an $(\infty,2)$-category of monoids modeled by a 2-complicial set is constructed.} $\Alg(\C)$ whose objects are monoids, whose 1-morphisms 
are bimodules, whose 2-morphisms are intertwiners, and whose composition of 1-morphisms is given by the balanced tensor product (with the induced bimodule actions).\\

Assuming now in addition that $\C$ is a symmetric monoidal category, the 2-category $\Alg(\C)$ inherits a symmetric monoidal structure with respect to 
which every object $A$ has a dual, given by the opposite algebra $A^{\opp}$ whose multiplication is that of $A$ precomposed with the symmetric braiding.
\begin{construction}\label{Construction: towards closedness in morita categories part one}
    Let $\C$ be a right closed monoidal category. Let $A$ and $B$ be monoids in $\C$, let $M$ be a left $A$-module and 
    let $N$ be a left $B$-module in $\C$. We will abuse notation in the following and denote underlying objects of $M$ and $N$ in $\C$ by the same symbol.  
    By assumption of right closedness, there is an object $M^N$ in $\C$ together with a morphism $\ev_N^M \colon M^N \otimes N \to M$ such that for every other object 
    $P \in \C$ the induced map $\psi \colon \hom_{\C}(P,M^N) \to \hom_{\C}(P \otimes N,M)$ is a natural bijection. In particular, we have
    \begin{align*}
            \hom_{\C}(A \otimes M^N,M^N) &\cong \hom_{\C}(A \otimes M^N \otimes N,M) \\
            \hom_{\C}(M^N \otimes B, M^N) &\cong \hom_{\C}(M^N \otimes B \otimes N, M) 
    \end{align*}
    In consequence, the morphisms 
    \begin{equation}
        \begin{aligned}
        \xymatrix@R=0.5em{
            A \otimes M^N \otimes N \ar[r]^-{\id \otimes \ev_N^M} & A \otimes M \ar[r]^-{\lambda_M} & M  \\
            M^N \otimes B \otimes N \ar[r]^-{\id \otimes \lambda_N} & M^N \otimes N \ar[r]^-{\ev_N^M} & N
        }
        \end{aligned}
    \end{equation}
    induce morphisms $\lambda_{M^N} \colon A \otimes M^N \to M^N$ and $\rho_{M^N} \colon M^N \otimes B \to M^N$, respectively. 
    We claim that these morphisms equip $M^N$ with the structure of an $A$-$B$-bimodule. We first note the commutativity of the diagrams 
    \begin{equation}\label{Equation: verifying action property of internal hom}
        \begin{aligned}
        \xymatrix{
            A \otimes A \otimes M^N \otimes N \ar[d]^-{\id \otimes \ev_N^M} \ar[r]^-{\mu_A \otimes \id} & A \otimes M^N \otimes N \ar[d]^-{\id \otimes \ev_N^M} \\
            A \otimes A \otimes M \ar[d]^-{\id \otimes \lambda_M} \ar[r]^-{\mu_A \otimes \id} & A \otimes M \ar[d]^-{\lambda_M} \\
            A \otimes M \ar[r]^-{\lambda_M} & M 
        }
        \hspace{1cm} 
        \xymatrix{
            M^N \otimes B \otimes B \otimes N \ar[r]^-{\id \otimes \mu_B \otimes \id} \ar[d]^-{\id \otimes \lambda_N} & M^N \otimes B \otimes N \ar[d]^-{\id \otimes \lambda_N} \\
            M^N \otimes B \otimes N \ar[r]^-{\id \otimes \lambda_N} \ar[d]^-{\id \otimes \lambda_N} & M^N \otimes N \ar[d]^-{\ev_N^M} \\
            M^N \otimes N \ar[r]^-{\ev_N^M} & M^N
        }
        \end{aligned}
    \end{equation}
    where the upper square in the left diagram commutes by the interchange law, the lower square in the left diagram commutes by the action property of $\lambda_M$, 
    the lower square in the right diagram commutes evidently, and the upper square in the right diagram commutes by the action property of $\lambda_N$.
    The naturality of the bijections $\psi$ described above allow us to deduce the action properties of $\lambda_{M^N}$ 
    and $\rho_{M^N}$ from the commutative diagrams in \Cref{Equation: verifying action property of internal hom}. Next, we claim that the left and right action 
    $\lambda_{M^N}$ and $\rho_{M^N}$ are compatible. To see this, we note the commutativity of the diagram 
    \begin{equation}
        \begin{aligned}
        \xymatrix{
            A \otimes M^N \otimes B \otimes N \ar[r]^-{\id \otimes \lambda_N} \ar[d]^-{\lambda_{M^N} \otimes \id} & A \otimes M^N \otimes N \ar[d]^-{\lambda_{M^N} \otimes \id} \\
            M^N \otimes B \otimes N \ar[r]^-{\id \otimes \lambda_N} \ar[d]^-{\id \otimes \lambda_N} & M^N \otimes N \ar[d]^-{\ev_N^M} \\
            M^N \otimes N \ar[r]^-{\ev_N^M} & N 
        }
        \end{aligned}
    \end{equation}
    where the lower square commutes evidently, while the upper square commutes by the naturality of the bijections $\psi$. Summarizing, for a left $A$-module 
    $(M,\lambda_M)$ and a left $B$-module $(N,\lambda_N)$ we obtain an $A$-$B$-bimodule $(M^N,\lambda_{M^N},\rho_{M^N})$. 
\end{construction}
\begin{construction}\label{Construction: intertwiner space objects}
    In the situation of \Cref{Construction: towards closedness in morita categories part one}, asumme in addition that we are given a monoid $C$ in $\C$ and 
    $M$ and $N$ carry right $C$-actions so that $M$ is an $A$-$C$-bimodule and $N$ is a $B$-$C$-bimodule. We wish to construct two morphisms of the form
    \begin{equation}
        M^N \rightrightarrows M^{N \otimes C}
    \end{equation}
    Invoking the right closedness, we may use the bijection $\hom_{\C}(M^N, M^{N \otimes C}) \cong \hom_{\C}(M^N \otimes N \otimes C,M)$ so that it suffices 
    to construct two elements in the latter set. The two maps we choose are 
    \begin{equation}
        \begin{aligned}
        \xymatrix@R=0.5em{
            M^N \otimes N \otimes C \ar[r]^-{\ev_N^M \otimes \id} & M \otimes C \ar[r]^-{\rho_M} & M \\
            M^N \otimes N \otimes C \ar[r]^-{\id \otimes \rho_N} & M^N \otimes N \ar[r]^-{\ev_N^M} & M  
        }
        \end{aligned}
    \end{equation}
    Given these two maps, we may ask whether the respective equalizer 
    \begin{equation}\label{Equation: Intertwiner mapping object}
        M^N_C := \equa(M^N \rightrightarrows M^{N \otimes C})
    \end{equation}
    exists in the category of $A$-$B$-bimodules in $\C$. 
\end{construction}
\begin{construction}\label{Construction: towards closedness in Morita 2-categories}
    Let $\C$ be a right closed monoidal category, $A,B$ and $C$ monoids in $\C$, $M$ an $A$-$C$-bimodule, and $N$ a $B$-$C$-bimodule.
    Assume the balanced tensor product $M^N \otimes_B N$ and the object $M^N_C$ exists. 
    For every $A$-$B$-bimodule $P$, we wish to construct a (natural) bijection 
    \begin{equation}\label{Equation: toward Morita closedness}
        \xymatrix{
            \hom_{A,B}(P,M^N_C) \ar[r]^-{} & \hom_{A,C}(P \otimes_B N,M),
        }
    \end{equation}
    recalling the (natural) bijection $\hom_{\C}(P,M^N) \cong \hom_{\C}(P \otimes N,M)$ following from the closedness assumption on $\C$. 
    First, we note that we have a map $\hom_{A,B}(P,M^N_C) \to \hom_{A,B}(P,M^N)$ induced by the equalizer map. Second, we have a map 
    \begin{equation}
        \xymatrix{
            \hom_{A,B}(P,M^N) \ar[r]^-{} & \hom_{\C}(P,M^N) \ar[r]^-{\cong} & \hom_{\C}(P \otimes N,M)
        }
    \end{equation}
    whose image we wish to understand. First, it is evident that each element in the image will be an intertwiner of left $A$-modules. Second, we 
    note for every intertwiner of left $B$-modules $f \colon P \to M^N$, we have a commutative diagram
    \begin{equation}
        \begin{aligned}
        \xymatrix{
            & P \otimes B \otimes N \ar[dl]_-{\id \otimes \lambda_N} \ar[d]^-{f \otimes \id} \ar[r]^-{\rho_P \otimes \id} & P \otimes N \ar[d]^-{f \otimes \id}\\
            P \ar[dr]_-{f \otimes \id} \otimes N & M^N \otimes B \otimes N \ar[d]^-{\id \otimes \lambda_N} \ar[r]^-{\rho_{M^N} \otimes \id} & M^N \otimes N \ar[dl]^-{\id} \\
            & M^N \otimes N & 
        }
        \end{aligned}
    \end{equation}
    where the left diagram commutes by the interchange law, the right upper square commutes by assumption of $f$ being an intertwiner of left $B$-modules, 
    and the lower triangle commutes by definition of the right $B$-action on $M^N$. In particular, the two outer composite maps in the diagram
    \begin{equation}
        P \otimes B \otimes N \rightrightarrows M
    \end{equation}
    are equal. By the universal property of the balanced tensor product, $f$ hence induces a map $P \otimes_B N \to M$. Hence, we obtain a bijection 
    \begin{equation}
        \hom_{A,B}(P,M^N) \cong \hom_{A}(P \otimes_B N,M).
    \end{equation}
    Finally, we wish to understand the image of the composition 
    \begin{equation}
        \xymatrix{
            \hom_{A,B}(P,M^N_C) \ar[r]^-{} & \hom_{A,B}(P,M^N) \ar[r]^-{\cong} & \hom_{A}(P \otimes_B N,M)
        }
    \end{equation}
    By the universal property of $M^N_C$, the image of the fitst map consists precisely of those intertwiners $P \to N^M$ for which postcomposition 
    with the two maps $M^N \rightrightarrows M^{N \otimes C}$ coincides. By construction of these maps, we may characterize the image of the composite 
    map $\hom_{A,B}(P,M^N_C) \to \hom_{A}(P \otimes_B N,M)$ as the interwiners of right $C$-modules. All together, we have constructed a bijection 
    \begin{equation}\label{Equation: toward Morita closedness}
        \xymatrix{
            \hom_{A,B}(P,M^N_C) \ar[r]^-{} & \hom_{A,C}(P \otimes_B N,M),
        }
    \end{equation}
    which is natural by the closedness of $\C$ (implying naturality of $\hom_{\C}(P,M^N) \cong \hom_{\C}(P \otimes N,M)$), the universal property 
    of the balanced tensor product, and the universal property of \Cref{Construction: intertwiner space objects}. 
\end{construction}
We deduce the following result from our discussion so far. 
\begin{theorem}\label{Theorem: closedness of the Morita 2-category}
    Let $\C$ be a right closed monoidal category which admits equalizers and (a calculus of) balanced tensor products. Then, the 
    2-category $\Mor(\C)$ is right 1-closed. If $\C$ is moreover closed symmetric monoidal, then $\Mor(\C)$ is 1-closed and has duals for objects.
\end{theorem}
\begin{remark}
    \Cref{Theorem: closedness of the Morita 2-category} holds under slightly weaker assumptions on the existence of (co)limits: it suffices that $\C$ admits 
    (a calculus of) balanced tensor products and intertwiner mapping objects in the sense of \Cref{Construction: intertwiner space objects}. 
\end{remark}
\begin{lemma}\label{Lemma: closedness and balanced tensor products}
    Let $\C$ be a left and right closed monoidal category which admits coequalizers. Then $\C$ admits (a calculus of) balanced tensor products.
\end{lemma}
\begin{proof}
    By assumption, $\C$ admits coequalizers hence in particular balanced tensor products. It remains to show that the functors 
    $a \otimes (-), (-) \otimes a \colon \C \to \C$ preserve balanced tensor products for any $a \in \C$.
    By assumption $\C$ is left and right closed, so that both these functors have right adjoints, hence preserve (small) colimits and in particular coequalizers. 
    This finishes the proof. 
\end{proof}
\begin{corollary}\label{Corollary: Morita 2-categories of nice categories}
    Let $\C$ be a closed symmetric monoidal category with (co)equalizers. Then, $\Mor(\C)$ is a symmetric monoidal 1-closed 2-category with duals.
\end{corollary}
\begin{proof}
    Combine \Cref{Theorem: closedness of the Morita 2-category} and \Cref{Lemma: closedness and balanced tensor products}. 
\end{proof}
\begin{example}
    The 2-category $\Mor(\cBorn)$ of complete Bornological algebras is a symmetric monoidal 1-closed 2-category with duals. 
\end{example}
\begin{example}
    The 2-category $\Mor(\Ch(R))$ of dg-algebras over a commutative ring $R$ is a symmetric monoidal 1-closed 2-category with duals. 
\end{example}
\begin{lemma}\label{Lemma: Functors between Morita 2-categories}
    Let $\C,\D$ be closed symmetric monoidal categorues with (co)equalizers and let $F \colon \C \to \D$ be a symmetric monoidal functor which preserves 
    (co)equalizers. Then, $F$ induces a symmetric monoidal 2-functor $\Mor(F) \colon \Mor(\C) \to \Mor(\D)$.
\end{lemma}
\begin{corollary}
    Let $\C$ be a closed symmetric monoidal category with (co)equalizers and let $K \colon \C \to \C$ be a symmetric monoidal involution. Then, 
    $K$ induces a symmetric monoidal involution $\Mor(K) \colon \Mor(\C) \to \Mor(\C)$.
\end{corollary}
\begin{proof}
    By assumption $K$ admits an inverse, hence admits both a left and a right adjoint, hence preserves small limits and colimits. The claim then follows 
    from \Cref{Lemma: Functors between Morita 2-categories}.
\end{proof}
\begin{example}
    The cocomplete closed symmetric monoidal category $\cBorn$ admits a symmetric monoidal involution $\overline{(-)} \colon \cBorn \to \cBorn$ which 
    assigns each complete Bornological vector space $V$ to its complex conjugate $\overline{V}$. It induces a symmetric monoidal involution 
    $\Mor(\overline{(-)}) \colon \Mor(\cBorn) \to \Mor(\cBorn)$ which assigns each complete Bornological algebra $A$ to its complex conjugate $\overline{A}$.
\end{example}

\subsection{Example: Profunctors}\label{Subsection:Profunctors}
In this section we study the theory of profunctors from the point of view of fully closed symmetric monoidal 2-categories.
\begin{definition}
    Let $\C$ and $\D$ be categories. A \emph{profunctor} $\C \nrightarrow \D$ is a functor $F \colon \D^{\opp} \times \C \to \Set$.
\end{definition}
\begin{remark}
    Profunctors are also called \emph{distributors} or \emph{correspondences}. 
\end{remark}
\begin{remark}\label{Remark: the double category of profunctors}
    Profunctors $F \colon \C \nrightarrow \C'$ and $G \colon \C' \nrightarrow \C''$ can be composed via a (co)end, i.e.\ we define the 
    composition $G \circ F \colon \C \nrightarrow \C''$ to be 
    \begin{equation}
    (G \circ F)(c'',c) := \int^{c' \in \C'} G(c'',c') \times F(c',c)
    \end{equation}
    For more details and background on profunctors, we refer to \cite[Section 5]{Loregian2021}. 
    Categories, profunctors, and natural transformations form a double category whose horizontal 2-category we denote by $\Prof$. The identity 
    1-morphism of an object $\C$ is the profunctor $\id_{\C} \colon \C \nrightarrow \C$ given by the functor $\C(-,-) \colon \C^{\opp} \times \C \to \Set$. 
    Every functor $F \colon \C \to \D$ induces two profunctors, namely, $\D(1,F) \colon \C \nrightarrow \D,\; \D(1,F)(d,c) = \D(d,F(c))$
    and $\D(F,1) \colon \D \nrightarrow \C,\; \D(F,1)(c,d) = \D(F(c),d)$; moreover, $\D(F,1) \dashv \D(1,F)$. In other words, 
    these two constructions of \emph{representable} profunctors are the companions and conjoints in the double category. 
    The two embeddings of $\Cat$ into $\Prof$ discussed above are called the \emph{co-} and \emph{contravariant embeddings}.
\end{remark}
\begin{remark}
    The cartesian product (together with the evident braiding) equips $\Prof$ with the structure of a symmetric monoidal 2-category.
\end{remark}
\begin{lemma}\label{Lemma: Prof is rigid}
    Every object $\C$ in $\Prof$ is dualizible, with the (left) dual given by the opposite category $\C^{\opp}$. In particular $\Prof$ admits a 1-contravariant 
    $\Oone$-volutive structure.
\end{lemma}
\begin{proof}
    We construct evaluation and coevaluation profunctors witnessing the opposite category $\C^{\opp}$ as a dual of $\C$. 
    Consider the profunctor $\widetilde{\ev} \colon \C^{\opp} \times \C \nrightarrow 1$ which is given by the functor 
    \begin{equation}
        \xymatrix{
        1^{\opp} \times \C^{\opp} \times \C \cong \C^{\opp} \times \C \ar[rr]^-{\C(-,-)} && \Set
        }
    \end{equation}
    i.e.\ the identity profunctor suitably reinterpreted. Similarly, we consider the profunctor 
    $\widetilde{\coev} \colon 1 \nrightarrow \C \times \C^{\opp}$ given by the functor 
    \begin{equation}
        \xymatrix{
        (\C \times \C^{\opp})^{\opp} \times 1 \cong \C^{\opp} \times \C \ar[rr]^-{\C(-,-)} && \Set
        }
    \end{equation}
    which is also the identity profunctor suitably reinterpreted. To see that the Zorro moves are satisfied, we simply note that 
    \begin{equation}
        \int^{c' \in \C} \C(c'',c') \times \C(c',c) \cong \C(c'',c)
    \end{equation}
    according to the ninja Yoneda lemma, see \cite[Proposition 2.2.1]{Loregian2021}. This finishes the proof.
\end{proof}
\begin{remark} \label{Remark: Spelling out the O(1)-volution on Prof}
    We wish to spell out the action of the dualization 2-functor $(-)^{\opp} \colon \Prof \to \Prof^{1\opp}$ on 1- and 2-morphisms. 
    Let $F \colon \C \nrightarrow \D$ be a profunctor. Then, 
    $F^{\opp} \colon \D^{\opp} \nrightarrow \C^{\opp}$ is given by the functor 
    \begin{equation}
    \xymatrix{
    F^{\opp} \colon \C \times \D^{\opp} \ar[r]^-{\cong} & \D^{\opp} \times \C \ar[r]^-{F} & \Set
    }
    \end{equation}
    where in the first step we have used the symmetric braiding on $\Prof$. Explicitly, $F^{\opp}(c,d) = F(d,c)$. The action of the dualization 
    functor on 2-morphisms is then given by sending a natural transformation $\alpha \colon F \to G$ to the natural transformation 
    $\alpha^{\opp} \colon F^{\opp} \to G^{\opp}$ given by
    \begin{equation}
    \xymatrix{ \C \times \D^{\opp} \ar[r]^-{\cong} &\D^{\opp} \times \C \rtwocell^F_G{\alpha}& \Set };
    \end{equation}
    explicitly, its component at $(c,d) \in \D \times \D^{\opp}$ is given by $\alpha^{\opp}_{(c,d)} := \alpha_{(d,c)}$. We note that, 
    since $(-)^{\opp}$ squares to the identity strictly, the higher coherence data of our $\Oone$-volutive structure is trivial.
\end{remark}
\begin{remark}\label{Remark: comparing opposite functors in Prof}
    In this remark wish to compare the dualization 2-functor $(-)^{\opp} \colon \Prof \to \Prof^{1\opp}$ to the 2-functor $(-)^{\opp} \colon \Cat \to \Cat^{2\opp}$.
    There is a commutative diagram of 2-functors 
    \begin{equation}\begin{aligned}
        \xymatrix{
            \Cat \ar[d]_-{(-)^{\opp}} \ar[r]^-{} & \Prof \ar[d]^-{(-)^{\opp}} \\ 
            \Cat^{2\opp} \ar[r]_-{} & \Prof^{1\opp}
        }
    \end{aligned}
    \end{equation}
    where the horizontal arrows are the co- and contravariant embeddings of $\Cat$ into $\Prof$. This follows immediately from our explicit descriptions.
\end{remark}
The following is a useful characterization of the Cauchy/idempotent completion, presented e.g. in \cite{borceux1986cauchy}.
\begin{lemma}
    Let $\C$ be a (small) category. Then, the following are equivalent 
    \begin{itemize}
        \item[(i)] $\C$ is Cauchy complete, and 
        \item[(ii)] for every (small) category $\D$ a profunctor $\D \nrightarrow \C$ has a left adjoint if and only if it is representable (that is, 
        induced by a functor $F \colon \D \to \C$ via \Cref{Remark: the double category of profunctors}).
    \end{itemize}
\end{lemma}
\begin{remark}\label{Remark: The 2-category of Profunctors is closed}
    The 2-category $\Prof$ of categories, profunctors, and natural transformations is right 1-closed, see \cite[Proposition 4.3]{Benabou1973LesDistributeurs}. 
    Given 1-morphisms $X \colon \C \nrightarrow \D$ and $Y \colon \C \nrightarrow \E$ in $\Prof$, that is, functors 
    $X \colon \D^{\opp} \times \C \to \Set$ and $Y \colon \E^{\opp} \times \C \to \Set$, the profunctor $Y^X \colon \E^{\opp} \times D \to \Set$ 
    is defined on objects as the equalizer of the diagram 
    \begin{equation}
        \Pi_{c \in \C} \hom(X(d,c),Y(e,c)) \rightrightarrows \Pi_{c,c' \in \C} \hom(X(d,c) \times \hom_{\C}(c,c'), Y(e,c'))
    \end{equation}
    where the first map assigns to an object $\Pi_{c \in \C} f_c \in \Pi_{c \in \C} \hom(X(d,c),Y(e,c))$ the object which on the pair $c,c' \in \C$ is defined to be 
    \begin{equation}
        \xymatrix{
            X(d,c) \times \hom(c,c') \ar[r]^-{f_c \times \id} & Y(e,c) \times \hom(c,c') \ar[r]^-{\ev_c^Y} & Y(e,c')
        }
    \end{equation}
    where in the last step we have used the map $\ev_c^Y$ corresponding under the standard hom-tensor adjunction in $\Set$ to the map 
    $\hom_{\C}(c,c') \to \hom(Y(e,c) \to Y(e,c')), f \mapsto Y(\id_e,f)$. Similarly, the second map is defined to be the composition 
    \begin{equation}
        \xymatrix{
            X(d,c) \times \hom(c,c')\ar[r]^-{\ev_c^X} & X(d,c') \ar[r]^-{f_{c'}} & Y(e,c').
        }
    \end{equation}
    We will not give all details of this construction and refer to \cite{Benabou1973LesDistributeurs} instead.
\end{remark}
\begin{remark}\label{Remark: Prof is a categorification of Set}
    We have an equivalence of closed monoidal categories
    \begin{equation}
        (\Hom_{\Prof}(\star,\star),\circ) \cong (\Set,\times)
    \end{equation}
    where $\Set$ is the category of sets equiped with the cartesian product. Indeed, by definition a profunctor 
    $\star \to \star$ is a functor $\star^{\opp} \times \star \cong \star \to \Set$, hence can be understood as a specific set; 
    similarly, a natural transformation between such profunctors amounts to a map between sets. It is also clear that the definition 
    of the horizontal composition in $\Prof$ reduces to the cartesian product of sets. 
\end{remark}
\begin{remark}
    Since $\Prof$ is symmetric monoidal, right 1-closed, and has duals, it is also left closed. Let $X \colon \D \nrightarrow \C$ and $Y \colon \E \nrightarrow \C$
    be 1-morphisms in $\Prof$. We wish to compute the 1-morphism ${}^{X}\!Y :=  d^{-1}(d(Y)^{d(X)})$ in the following, where $d=(-)^{\opp}$ denotes the dualization 
    functor on $\Prof$. By definition, $d(X) \colon \C^{1\opp} \nrightarrow \D^{1\opp}$ is given by $d(X)(d,c) = X(c,d)$ and similarly for $d(Y)$. Next, we compute 
    $d(Y)^{d(X)} \colon \D^{1\opp} \nrightarrow \E^{1\opp}$ on objects to be the equalizer of
    \begin{equation}
        \Pi_{c \in \C} \hom(X(c,d),Y(c,e)) \rightrightarrows \Pi_{c,c' \in \C} \hom(X(c,d) \times \hom_{\C}(c,c'), Y(c',e)).
    \end{equation}
    Finally, $d^{-1}(d(Y)^{d(X)}) \colon \E \nrightarrow \D$ is given on $(d,e) \in \D^{1\opp} \times \E$ again as this equalizer.
    In this generality, this is all that we can say here. \\
    We now consider a more special situation, in which we want to compare the left and right closed assignments more explicitly: given a 1-morphism 
    $Z \colon \C \nrightarrow \D$ in $\Prof$, we may consider the 1-morphism 
    $\id_{\C}^Z \colon \D \nrightarrow \C$ which is given on an object $(\hat{c},d) \in \C^{1\opp} \times \D$ as the equalizer of
    \begin{equation}
        \Pi_{c \in \C} \hom(Z(d,c),\hom_{\C}(\hat{c},c)) \rightrightarrows \Pi_{c,c' \in \C} \hom(Z(d,c) \times \hom_{\C}(c,c'), \hom_{\C}(\hat{c},c')).
    \end{equation}
    On the other hand, we evaluate ${}^{Z}\!\id_{\D} \colon \D \nrightarrow \C$ on an object $(c,\hat{d}) \in \C^{1\opp} \times \D$ to be the equalizer of
    \begin{equation}
        \Pi_{d \in \D} \hom(Z(d,c),\hom_{\D}(d,\hat{d})) \rightrightarrows \Pi_{d,d' \in \D} \hom(Z(d,c) \times \hom_{\D}(d,d'),\hom_{\D}(d',\hat{d})).
    \end{equation}
    We note that by the closedness of the monoidal category $\Set$, we have 
    \begin{align*}
        &\hom(Z(d,c) \times \hom_{\C}(c,c'), \hom_{\C}(\hat{c},c')) \cong \\
        &\hom(Z(d,c),\hom_{\Set}(\hom_{\C}(c,c'),\hom_{\C}(\hat{c},c'))).
    \end{align*}
    and similar for the left handed version.
\end{remark}
\begin{remark}
    In this remark, we wish to study (lax) hermitian fixed points in the context of the fully closed symmetric monoidal 2-category $\Prof$.
    First, since $\Prof$ is symmetric monoidal with duals, it carries an $\Oone$-volutive structure, so that we may consider the associated 
    $\Oone$-dagger 2-category $\B := S_{\Oone}\Prof$ without steps in between. The objects of this $\Oone$-dagger 2-category are triples 
    $(\C,\theta_{\C},\omega_{\C})$ consisting of a category $\C$, an invertible profunctor $\theta_{\C} \colon \C \to \C^{\opp}$, and an invertible 
    transformation $\omega_{\C} \colon \id_{\C} \to \theta_{\C}^{\opp} \circ \theta_{\C}$ satisfying a coherence condition. Note that $\Oone$-volutive 
    categories are special objects in $\B$: they are precisely those objects, for which $\theta_{\C}$ is an honest functor\footnote{Under certain additional 
    assumptions, invertible profunctors and invertible functors are the same, such as if $\C$ is Cauchy complete.}.\\
    According to our general theory, $\B$ is enriched in lax $\Oone$-volutive categories, so that we may consider the categories of lax hermitian fixed 
    points and isometries locally, i.e. consider $\LaxHerm(\Hom_{\B}(a,b))$. The objects of this category are lax $\Oone$-volutive profunctors, and the 
    morphisms are lax $\Oone$-volutive transformations between such. 
\end{remark}
\begin{remark}
    Via the embedding $\Cat \to \Prof$ and its compatibility with passing to opposites explained in \Cref{Remark: comparing opposite functors in Prof},
    we could develop the theory of lax $\Oone$-volutive categories and their (higher) morphisms purely from the perspective of the fully closed symmetric 
    monoidal 2-category $\Prof$. From this duality based point of view, $\Prof$ hence carries an advantage over $\Cat$, the latter of which neither 
    being 1-closed nor having duals for objects. 
\end{remark}

\section{Examples and Applications}\label{Appendix: supplementary background material}
In this appendix, we wish to give a more detailed account on specific examples of interest, for instance, (complete) bornological vector spaces and modules over $*$-rings.
We will review some of the general theory and discuss their respective (lax) $\Oone$-volutive structures. We will also consider the 2-category of 
Hilbert spaces, relations, and inclusions, discuss how unbounded operators are subsumed by this theory, and whether it carries a dagger structure. 
Finally, we illustrate some of the developed theory on fully closed symmetric monoidal 2-categories in the context of complete bornological algebras. 
\subsection{Bornological vector spaces}\label{Subsection: Bornological vector spaces}
In this section we wish to give a more detailed account of (complete) bornological vector spaces, which we will consider again in \Cref{Subsection: Bornological algebras}.
Most of the material covered in this section is well-known and we do no claim originality, see e.g. \cite{Meyer2007Local}.
\begin{definition}\label{Definition: bornology}
    A \emph{bornology} on a set $X$ is a collection $\mathfrak{B} \subset \mathcal{P}(X)$ of subsets of $X$ such that 
    (i) $\bigcup_{B \in \mathfrak{B}} B = X$ (ii) if $B \in \mathfrak{B}$ and $A \subseteq B$ then $A \in \mathfrak{B}$
    (iii) if $B_1,...,B_n \in \mathfrak{B}$, then $\bigcup_{1 \leq i \leq n} B_i \in \mathfrak{B}$. A \emph{bornological space}
    is a set $X$ together with a bornology $\mathfrak{B}$. If $(X,\mathfrak{B}_X),(Y,\mathfrak{B}_Y)$ are bornological sets, a 
    map $f \colon X \to Y$ is called \emph{bounded} if $f(B) \in \mathfrak{B}_Y$ for all $B \in \mathfrak{B}_X$. 
\end{definition}
\begin{example}\label{Example: von Neumann bornology}
    Let $X$ be a semi-normed space. Then, we define a bornology on $X$ by declaring a set $B \subseteq X$ to be bounded if $\sup_{b \in B} \| b \| < \infty$. 
    We call this bornology the \emph{von Neumann bornology} on $X$. In particular, we may think of any normed or Banach space as a bornological space.
\end{example}
\begin{definition}
    A \emph{(complex) bornological vector space} is a (complex) vector space internal to the category of bornological spaces. 
\end{definition}
\begin{remark}
    Explicitly, a bornological vector space consists of a vector space together with a bornology so that addition and scalar multiplication are 
    bounded, where the underlying field (here $\mathbb{C}$) is equiped with the standard bornology. We could also consider $\mathbb{R}$ as the underlying field.
\end{remark}
\begin{definition}
    A bornological vector space is called 
    \begin{itemize}
        \item[(i)] \emph{convex} if every bounded set is contained in a bounded disk,
        \item[(ii)] \emph{separated} if every bounded set is contained in a bounded norming disk,
        \item[(iii)] \emph{complete} if every bounded set is contained in a bounded completant disk.
    \end{itemize}
    where we infer the definitions above from \cite[Definition 2.1.12]{aretz2025functorialitybornological} and the preceding text in loc. cit.
\end{definition}
\begin{remark}\label{Remark: von Neumann bornologies and bornological notions}
    The von Neumann bornology on a semi-normed space is convex. %Indeed, let $B$ be bounded. Define $M := \sup_{b \in B} \| b \|$. Then $B$ is contained in the disk of radius $M$.
    The von Neumann bornology on a semi-normed space is separated if and only if the space is normed. 
    The von Neumann bornology on a normed space is complete if and only if the space is Banach.
\end{remark}
\begin{definition}
    We denote by $\Born$ the category of convex bornological vector spaces and bounded linear maps, by $\sBorn$ the full subcategory of separated convex 
    bornological vector spaces, and by $\cBorn$ the full subcategory of complete convex bornological vector spaces. 
\end{definition}
\begin{remark}
    Every bornological vector spaces in this paper will be convex, so that we may leave this adjective implicit throughout. 
\end{remark}
\begin{remark}
    The category of complete (convex) bornological vector spaces is closely related to the ind-completion of the category of Banach spaces and bounded 
    linear maps: it is the category of \emph{injective/monomorphic} ind-objects in $\Ban$. We refer to \cite[Proposition 5.15]{ProsmansSchneiders2000} 
    for a precise account and proof of this statement.
\end{remark}
We recall the following results from \cite[Section 2.3]{aretz2025functorialitybornological} and \cite{Meyer2007Local}.
\begin{lemma}
    The inclusion functors $I \colon \sBorn \to \Born$ and $J \colon \cBorn \to \sBorn$ admit left adjoints 
    \begin{equation}
        \xymatrix{
        \cBorn \ar@<0.5ex>[r]^J & \sBorn \ar@<0.5ex>[l]^{\operatorname{Comp}} \ar@<0.5ex>[r]^I & \Born \ar@<0.5ex>[l]^{\operatorname{Sep}}
        }
    \end{equation} 
    Moreover, $\Born, \sBorn$, and $\cBorn$ have all (small) limits and colimits, see \cite[Proposition 1.126]{Meyer2007Local}. 
    We summarize \Cref{Remark: von Neumann bornologies and bornological notions} by noting the diagram of fully faithful functors, see \cite[Section 1.1.5]{Meyer2007Local},
    \begin{equation}
        \begin{aligned}
        \xymatrix{
            \cBorn \ar[r]^{\subset} & \sBorn \ar[r]^{\subset} & \Born \\
            \Ban \ar[u]^{\vN} \ar[r]^{\subset} & \Nor \ar[r]^{\subset} \ar[u]^{\vN} & \sNor \ar[u]^{\vN}
        }
        \end{aligned}
    \end{equation}
    where $\Ban,\Nor$, and $\sNor$ denote the categories of Banach spaces, normed spaces, and semi-normed spaces, respectively, 
    all with bounded linear maps as morphisms. 
\end{lemma}
We also infer the following results from \cite[Proposition 1.111]{Meyer2007Local} or \cite[Sections 2.4-2.6]{aretz2025functorialitybornological}.
\begin{lemma}\label{Lemma: Bornological vector spaces and closed symmetric monoidal structures}
    $\cBorn$ is a closed symmetric monoidal category with respect to the completed tensor product of complete convex bornological vector spaces.
\end{lemma}
\begin{remark}
    The tensor product of convex bornological vector spaces is defined by the following universal property. Let $V$ and $W$ be convex bornological 
    vector spaces. Their tensor product, if it exists, is a convex bornological vector space $V \otimes^b W$ together with a bounded bilinear map 
    $\tau \colon V \times W \to V \otimes^b W$ such that for all $U \in \Born$ and all bounded bilinear maps $\varphi \colon V \times W \to U$ 
    there is a unique linear map $f \colon V \otimes^b W \to U$ such that the diagram 
    \begin{equation}
        \begin{aligned}
        \xymatrix{
            V \times W \ar[r]^-{\tau} \ar[d]_-{\varphi} & V \otimes^b W \ar[dl]^-{\exists ! f}  \\
            U 
        }
        \end{aligned}
    \end{equation}
    is commutative. Moreover, the tensor product of bornological spaces exists; it is given by the algebraic tensor product of the underlying 
    vector spaces together with the bornology described in \cite[Proposition 2.4.2]{aretz2025functorialitybornological}.\\
    The tensor product of complete bornological vector spaces is given by the completion of the tensor product of convex bornological vector spaces, 
    $V \hat{\otimes} W := \operatorname{Comp}(V \otimes^b W)$.
\end{remark}
\begin{remark}
    The tensor product of complete convex bornological vector spaces recovers the projective tensor product $\otimes_{\pi}$ of Banach spaces (and more generally 
    Frechet spaces, see \cite[Definition 1.15 \& Theorem 1.87]{Meyer2007Local}) in the sense that $\vN(V \otimes_{\pi} W) \cong \vN(V) \hat{\otimes} \vN(W)$ holds 
    for all Banach spaces $V$ and $W$.
\end{remark}
\begin{corollary}\label{Corollary: bornological vector spaces are lax volutive}
    The category $\cBorn$ of complete convex bornological vector spaces and bounded linear maps admits a lax $\Oone$-volutive structure.
\end{corollary}
\begin{proof}
    Combine \Cref{Lemma: Bornological vector spaces and closed symmetric monoidal structures} and \Cref{Theorem: closed symmetric monoidal categories are lax volutive}.
\end{proof}
\begin{remark}
    We wish to describe the lax $\Oone$-volutive structure constructed in \Cref{Corollary: bornological vector spaces are lax volutive} explicitly. 
    First, let $V$ and $W$ be complete convex bornological vector spaces. Them, the exponential object $W^V$ is the vector space $\Born(V,W)$ 
    together with the complete convex bornology described in \cite[Section 2.6 \& Proposition 2.5.1]{aretz2025functorialitybornological}. In particular, 
    the functor of our lax $\Oone$-volutive structure assigns each complete convex bornological vector space $V$ to its dual space $V' := \Born(V,\mathbb{C})$.
    The natural transformation of our lax $\Oone$-volutive structure has as its component at $V$ the canonical map 
    $\iota_V \colon V \to \Born(\Born(V,\mathbb{C}),\mathbb{C})$.
\end{remark}
\begin{definition}
    A complete convex bornological vector space $V$ is called \emph{reflexive} if $\iota_V$ is an isomorphism. 
\end{definition}
\begin{corollary}
    The category of reflexive complete convex bornological vector spaces $\cBorn^{\refl}$ admits an $\Oone$-volutive structure.
\end{corollary}
\begin{proof}
    Combine \Cref{Corollary: bornological vector spaces are lax volutive} and \Cref{Construction: Oone volutions from lax Oone volutions}.
\end{proof}
\begin{lemma}
    The functor $\vN \colon \Ban \to \cBorn$ is lax $\Oone$-volutive. Moreover, it descends to an $\Oone$-volutive functor 
    $\vN \colon \Ban^{\refl} \to \cBorn^{\refl}$. 
\end{lemma}
\begin{proof}
    Let $V$ be a Banach space. We claim that $\vN(V^*) \cong \Born(\vN(V),\mathbb{C})$ as bornological vector spaces. By definition, both sides 
    have the same underlying vector spaces, so that it suffices to show that their bounded sets coincide. First, recall that $V^*$ is a Banach space 
    equiped with the operator norm. In particular, a set $A \subset V^*$ is bounded if and only if 
    $\sup_{f \in A} \| f \| = \sup_{f \in A} \sup_{ \|x\| \leq 1} |f(x)| < \infty$. In other words, $A$ is bounded if and only if the image of the unit 
    ball $A(B_V)$ is bounded in $\mathbb{C}$. In particular, any bounded set in $\Born(V,\mathbb{C})$ is bounded in $V^*$. Conversely, any bounded set $A$
    in $V^*$ is bounded in $\Born(V,\mathbb{C})$ since $| f(x) | \leq \| f \| \cdot  \| x \|  \leq (\sup_{g \in A} \| g \|) \cdot \| x \| < \infty$ for all 
    $f \in A, x \in V$. This proves our claim. Similarly, the respective double dual spaces coincide, and the canonical morphisms into the double dual 
    spaces coincide. Noting that the lax $\Oone$-volutive structure on $\vN$ is already a natural isomorphism, the second claim follows immediately. 
\end{proof}

\subsection{Modules over star-rings}\label{Subsection: Modules over star-rings}
In this section we wish to give a more detailed account of modules over $*$-rings. We start with a discussion of dualities for modules over ordinary (non-commutative) rings.
\begin{construction}
    Let $R$ be a (possibly non-commutative) ring and let $M$ and $N$ be right $R$-modules. Then the set $\hom_{R}(M,N)$ of intertwiners generically does not 
    inherit a right $R$-module structure. Assume now in addition that $S$ is a second (possibly non-commutative) ring and $N$ is an $S$-$R$-bimodule. 
    Then $\hom_{R}(M,N)$ carries the structure of a left $S$-module: we take the pointwise addition together with the $S$-action 
    \begin{equation}
        (s \triangleright \phi)(m) = s \phi(m).
    \end{equation}
    In particular, since every ring $R$ is canonically an $R$-$R$-bimodule over itself, $\hom_{R}(M,R)$ carries the structure of a left $R$-module. 
    We will refer to $M^{\vee_r} := \hom_{R}(M,R)$ as the \emph{dual} left $R$-module to $M$. 
\end{construction}
\begin{variant}
    Similarly, we may start with a left $R$-module $M$ in which case the dual $M^{\vee_\ell}$ will be a right $R$-module. 
\end{variant}
\begin{construction}\label{Construction: dualization functor for non-commutative modules}
    Let $R$ and $S$ be (possibly non-commutative) rings, let $M_1,M_2$ be right $R$-modules, and let $N_1,N_2$ be $S$-$R$-bimodules. Furthermore, let 
    $f \colon M_1 \to M_2$ be an intertwiner of right $R$-modules and let $g \colon N_2 \to N_1$ an intertwiner of $S$-$R$-bimodules. Then, we obtain 
    an intertwiner of left $S$-modules $\hom_{R}(M_2,N_2) \to \hom_{R}(M_1,N_1), \phi \mapsto g \circ \phi \circ f$. This construction is clearly compatible 
    with respect to compositions of intertwiners $f,f'$ and $g,g'$. \\
    In particular, for any intertwiner of right $R$-modules $f \colon M_1 \to M_2$ we obtain an intertwiner of left $R$-modules $M_2^{\vee_r} \to M_1^{\vee_r}$. 
    The assignment $M \mapsto M^{\vee}_r$ hence defines a functor $(-)^{\vee_r} \colon \Mod_R \to {}_{R}\!\Mod^{\opp}$. 
\end{construction}
\begin{variant}
    Similarly, we obtain a functor $(-)^{\vee_\ell} \colon {}_{R}\!\Mod \to \Mod_R^{\opp}$. 
\end{variant}
\begin{construction}
    Let $R$ be a (possibly non-commutative) ring. For each right $R$-module $M$, we obtain an intertwiner of right $R$-modules $\iota_M \colon M \to M^{\vee_r \vee_\ell}, 
    m \mapsto \ev_m$ where $\ev_m(\phi) = \phi(m)$ is the evaluation at $m$. The assignment $M \mapsto \iota_M$ defines a natural transformation 
    $\iota^r \colon \id_{\Mod_R} \to (-)^{\vee_r \vee_\ell}$. \\
    The natural transformation $\iota^r$ satisfies a coherence condition over the triple dual, namely $\iota_M^{\vee} \circ \iota_{M^\vee} = \id_{M^\vee}$. 
    The proof that $\iota^r$ defines a natural transformation and is coherent in the sense of the previous sentence is completely analogous to the one 
    given in \Cref{Example: Banach spaces are lax volutive}.
\end{construction}
\begin{variant}
    Similarly, we obtain a natural transformation $\iota^{\ell} \colon \id_{{}_{R}\!\Mod} \to (-)^{\vee_\ell \vee_r}$ which satisfies the evident coherence condition.
\end{variant}
\begin{lemma}\label{Lemma: The adjunction involving modules over non-commutative rings}
    The assignment $a \mapsto \Mod_R, b \mapsto {}_{R}\!\Mod^{\opp}, X \mapsto (-)^{\vee_r}, Y \mapsto ((-)^{\vee_\ell})^{\opp}, f \mapsto \iota^r, g \mapsto (\iota^{\ell})^{\opp}$
    defines a 2-functor $\Adj \to \Cat$. 
\end{lemma}
Without further assumption, \Cref{Lemma: The adjunction involving modules over non-commutative rings} is the best statement we can make. 
We now consider more favorable conditions, starting with an observation/definition. 
\begin{remark}
    The category of rings $\Ring$ carries an $\Oone$-action which assigns each ring $R$ to its opposite $R^{\opp}$. The category of $\star$-rings 
    is precisely the category of $\Oone$-homotopy fixed points with respect to this action. Spelled out, a $\star$-structure on a ring is an isomorphism 
    $\star \colon R \to R^{\opp}$ satisfying $\star^2=\id_R$. 
\end{remark}
\begin{example}
    Any commutative ring $R$ is a $*$-ring with respect to the identity $\id \colon R \to R$. 
\end{example}
\begin{example}
    The quaternions $\mathbb{H}$ together with the quaternionic conjugation form a $\star$-ring. 
\end{example}
\begin{proposition}\label{Proposition: modules over star rings admits lax volutive structure}
    Let $R$ be a (possibly non-commutative) $*$-ring. Then, $\Mod_R$ admits a lax $\Oone$-volutive structure. 
\end{proposition}
\begin{proof}
    We have already constructed in \Cref{Lemma: The adjunction involving modules over non-commutative rings} a 2-functor $\Adj \to \Cat$ 
    whose value at $a \in \Adj$ is $\Mod_R$. It remains to describe a 2-contravariant $\Oone$-volutive structure on this functor. The construction of such 
    a structure from the $*$-structure is straightforward. 
\end{proof}
\begin{corollary}
    Let $R$ be a (possibly non-comuutative) $*$-ring. Then, the category of reflexive right $R$-modules $\Mod_R^{\refl}$ admits an $\Oone$-volutive structure. 
\end{corollary}
\begin{proof}
    Combine \Cref{Proposition: modules over star rings admits lax volutive structure} and \Cref{Construction: Oone volutions from lax Oone volutions}. 
\end{proof}
\begin{remark}
    A more sophisticated version of the discussion here (involving Banach $*$-modules) allows us to reconstruct the dagger structure on the category 
    of quaternionic hermitian(/Hilbert) spaces from the lax $\Oone$-volutive structure on the category of Banach modules over the $*$-ring $\mathbb{H}$. 
    The discussion here is sufficient only for the finite-dimensional case. 
\end{remark}

\subsection{Unbounded operators}\label{Subsection: unbounded operators}
In this section we discuss unbounded operators and their higher category. We will start by reviewing some standard material in the theory of 
unbounded operators, which may be found e.g. in \cite{Schmuedgen2012unbounded}.
\begin{definition}
    Let $B$ and $B'$ be Banach spaces. An \emph{unbounded operator} $B \to B'$ is a linear map $f \colon D(f) \to B'$ where $D(f)$ is a linear subspace of $B$, 
    called the \emph{domain} of $f$. 
\end{definition}
\begin{example}
    Let $B$ and $B'$ be Banach spaces. Any bounded operator $f \colon B \to B'$ is in particular an unbounded operator with domain $D(f)=B$. 
\end{example}
A good class of examples of unbounded operators is given by differential operators. The following is one such example.
\begin{example}
    Consider the Hilbert space $L^2(\mathbb{R})$ of square integrable functions and the differentiation operator $\mathrm{d}/\mathrm{d}\mathrm{x}$ with domain 
    the (dense) subspace of Schwartz functions $\mathcal{S}(\mathbb{R}) \subseteq L^2(\mathbb{R})$.
\end{example}
\begin{remark}\label{Remark: composition of unbounded operators}
    Unbounded operators can be (unitaly and associatively) added and composed, with domains $D(f + g) := D(f) \cap D(g)$ for $f,g \colon B \to B'$ and 
    $D(f \circ g) := \{x \in D(g) | g(x) \in D(f) \}$ for $g \colon B \to B'$ and $f \colon B' \to B''$, respectively. However, in general we have 
    \begin{equation}
        f(g+h) \neq fg + fh
    \end{equation}
    for $g,h \colon B \to B'$ and $f \colon B' \to B'$. A counterexample\footnote{which can be found in lecture notes \cite{Remline2021unbounded}.} 
    is given by the unbounded operators $f = d/dx, g=d/dx, h = -d/dx$ on the Banach space $C^0([0,1])$; indeed, we have $D(f(g+h)) = C^1([0,1])$ while 
    $D(fg+fh) = C^2([0,1])$.
\end{remark}
\begin{definition}
    Let $f \colon B \to B'$ be an unbounded operator.
    \begin{itemize}
        \item The \emph{graph} of $f$ is the linear subspace $\Gamma(f) := \{(x,f(x)) | x \in \D(f)\} \subseteq B \oplus B'$.
        \item An unbounded operator $\hat{f} \colon B \to B'$ is called an \emph{extension} of $f$ if $\Gamma(f) \subseteq \Gamma(\hat{f})$; equivalently, 
                $D(f) \subseteq D(\hat{f})$ and $\hat{f}x = fx$ for all $x \in D(f)$.
        \item $f$ is said to be \emph{closed} if $\Gamma(f) \subseteq B \oplus B'$ is a closed subset. 
        \item $f$ is said to be \emph{closeable} if it admits a closed extension.
        \item $f$ is said to be \emph{densely defined} if $D(f)$ is dense in $\B$. 
    \end{itemize}
\end{definition}
\begin{remark}
    Any bounded densely defined operator $f \colon D(f) \subseteq B \to B'$ extends to all of $B$, by continuity.%Recall that for normed spaces, boundedness and continuity are the same notion.
\end{remark}
\begin{remark}
    By the closed graph theorem, if an unbounded operator $f \colon B \to B'$ is closed and has $D(f) = B$, then $f$ is bounded. 
\end{remark}
\begin{notation}
    We denote by $\overline{T}$ the (unique) \emph{closure} of a closable unbounded operator $T$, that is, the smallest closed 
    extension of $T$.
\end{notation}
We now turn our attention to Hilbert spaces, following \cite[Section 1.2]{Schmuedgen2012unbounded}.
\begin{construction}\label{Construction: adjoint of a densely defined unbounded operator}
    Let $H$ and $H'$ be Hilbert spaces. Let $T \colon H \to H'$ be an unbounded operator. We define 
    \begin{equation}
        D(T^\dagger) := \{y \in H' \ | \ \exists z \in H : \langle Tx,y \rangle = \langle x , z \rangle \forall x \in D(T) \}
    \end{equation}%Asking for z to be unique will likely result in a subset which is not a linear subspace.
    which is clearly a linear subspace of $H'$. In order for the element $z$ to be unique, we assume in addition that $T$ is densely defined; 
    in this case, we have $\langle x, z - z' \rangle = 0$ for all $x \in D(T)$ in a dense subset of $H$, implying that $z = z'$. We may then define 
    $T^\dagger y = z$, obtaining an unbounded operator $T^\dagger \colon H' \to H$.
\end{construction}
\begin{notation}
    Let $T \colon H \to H'$ be a densely defined unbounded operator. The unbounded operator  $T^\dagger \colon H' \to H$ described in 
    \Cref{Construction: adjoint of a densely defined unbounded operator} is called the \emph{adjoint} of $T$. By construction, we have 
    \begin{equation}\label{Equation: adjunction equation for unbounded operators}
        \langle Tx, y \rangle = \langle x, T^\dagger y\rangle \ \forall x \in D(T), \ y \in D(T^\dagger).
    \end{equation}
    %The domain of $T^\dagger$ does not need to be dense in $H'$.
\end{notation}
\begin{warning}
    Many of the properties that the adjoint of a bounded operator satisfies do not carry over to this more general setting. 
    We will collect desired results under favorable additional assumptions in the following.
\end{warning}
\begin{proposition}[{\cite[Propositions 1.6 \& 1.7 and Theorem 1.8]{Schmuedgen2012unbounded}}]\label{Proposition: properties of the adjoint operator}
    Let $T \colon H \to H'$ be a densely defined unbounded operator. Then: 
    \begin{itemize}
        \item[(i)] $T^\dagger \colon H' \to H$ is a closed unbounded operator, 
        \item[(ii)] $T$ is closeable if and only if $D(T^\dagger)$ is dense in $H'$,
        \item[(iii)] if $T$ is closable, then $\overline{T}^\dagger = T^\dagger$ and we have $\overline{T} = (T^\dagger)^\dagger =: T^{\dagger \dagger}$, and 
        \item[(iv)] $T$ is closed if and only if $T = T^{\dagger \dagger}$. 
    \end{itemize}
    Furthermore, let $S \colon H' \to H''$ be another densely defined unbounded operator so that $ST$ is densely defined. Then 
    \begin{itemize}
        \item[(v)] $\Gamma(T^\dagger S^\dagger) \subset \Gamma((ST)^\dagger)$, and 
        \item[(vi)] if $S$ is bounded (and hence extends to a bounded operator $\hat{S}$ defined on all of $H'$), 
            then $T^\dagger \hat{S}^\dagger = (\hat{S}T)^\dagger$. 
    \end{itemize}
\end{proposition}
\begin{remark}
    \Cref{Proposition: properties of the adjoint operator} suggests that the differences of the theory of ordinary dagger categories and unbounded operators are 
    too great to be ignored. If one were to attempt to formalize the latter in the setting of the former, one would first of all have to consider 
    closed and densely defined operators. However, the composition of closed and densely defined operators is not necessarily closed and densely defined, 
    so that it is not clear which category to consider; second, while the supposed dagger structure would be identity-on-object and strictly involutive, 
    it need not be a functor without further assumptions. However, we do note that there is a category of Hilbert spaces and unbounded operators.
\end{remark}
In the following, we pursue a relational approach towards organizing unbounded operators and their adjoints, partially following lecture notes 
\cite{Remline2021unbounded} (which treat the case of unbounded operators on a single Hilbert space). The theory is motivated by the following
observation.
\begin{remark}
    Let $T \colon H \to H'$ be an unbounded operator. Then, $T$ and its domain can be covered from its graph $\Gamma(T)$. Indeed, we obtain the domain as 
    \begin{equation}
        D(T) := \{x \in H \ | \ (x,y) \in \Gamma(T) \ \text{for some} \ y \in H' \}
    \end{equation}
    and the operator $T$ by setting $Tx = y$ for the unique $y \in H'$ with $(x,y) \in \Gamma(T)$. 
\end{remark}
\begin{definition}
    Let $H$ and $H'$ be Hilbert spaces. A \emph{relation} from $H$ to $H'$ is a linear subspace of $H \oplus H'$. 
    If $H'=H$, we will call a relation from $H$ to $H'$ a \emph{relation on $H$}.
\end{definition}
\begin{example}
    Let $T \colon H \to H'$ be an unbounded operator. Then $\Gamma(T)$ is a relation from $H$ to $H'$.
\end{example}
\begin{definition}\label{Definition: various notions for relations}
    Let $V$ be a relation from $H$ to $H'$ and $W$ be a relation from $H'$ to $H''$. We define the following notions:
    \begin{itemize}
        \item the \emph{diagonal} relation on $H$ is the linear subspace $\Delta_H := \{(x,x) \ | \ x \in H\}$, 
        \item the \emph{reverse}\footnote{In the literature, the reverse is also called the \emph{inverse}. We do not use this terminology since it is not to 
        be understood as an inverse with respect to the composition of relations we describe, with unit given by the diagonal relation.} of 
        $V$ is the relation from $H'$ to $H$ given by
        \begin{equation}
            V^{\rev} := \{(y,x) \in H' \oplus H \ | \ (x,y) \in V \subseteq H \oplus H'\},
        \end{equation}
        \item the \emph{closure} $\overline{V}$ of $V$ is the closure of the linear subspace $V \subseteq H \oplus H'$. 
        We call a relation \emph{closed} if $\overline{V} = V$,
        \item the \emph{adjoint} of $V$ is the relation from $H'$ to $H$ given by 
        \begin{equation}
            V^\dagger := \{(u,v) \in H' \oplus H \ | \ \langle v,x \rangle_{H} = 
            \langle u,y \rangle_{H'} \ \text{for all} \ (x,y) \in V \},
        \end{equation}
        to see that this is a linear subspace, let $(u,v),(s,t) \in V^\dagger$. Then 
        $\langle v+t, x \rangle_H = \langle v,x \rangle_H + \langle t, x \rangle_H = \langle u,y \rangle_{H'} + \langle s,y \rangle_{H'} = \langle u+s,y \rangle$ 
        for all $(x,y) \in V$, hence $(u+s,v+t) \in V^\dagger$, and
        \item the \emph{composition} of $V$ and $W$ is the relation from $H$ to $H''$ given by 
        \begin{equation}
            W \circ V := \{ (x,z) \in H \oplus H'' \ | \ \exists y \in H' : (x,y) \in V \land (y,z) \in W \},
        \end{equation}
        to see that this is a linear subspace, let $(x,z),(x',z') \in W \circ V$. Then there are $y,y'$ such that $(x,y),(x',y') \in V$ and $(y,z),(y',z') \in W$, 
        hence $(x+x',y+y') \in V$ and $(y+y',z+z') \in W$, hence $(x+x',z+z') \in W \circ V$ and similarly $(\lambda x,\lambda z) \in W \circ V$ for any complex 
        scalar $\lambda$. 
    \end{itemize}
\end{definition}
\begin{lemma}\label{Lemma: the graph is compatible with composition}
    Let $T \colon H \to H'$ and $S \colon H' \to H''$ be unbounded operators. Then $\Gamma(ST) = \Gamma(S) \circ \Gamma(T)$.
\end{lemma}
\begin{proof}
    We have 
    \begin{align*}
        \Gamma(ST) &= \{ (x,STx) \ | \ x \in D(ST) \} \\
        &= \{ (x,STx) \ | \ x \in D(T) \land T(x) \in D(S) \} \\
        &= \{ (x,STx) \ | \ (x,Tx) \in \Gamma(T) \land (Tx,STx) \in \Gamma(S) \} \\
        &= \Gamma(S) \circ \Gamma(T)
    \end{align*}
    where in the last equation we have used the evident inclusion together with the reverse inclusion which is obtain by applying twice the observation 
    that if $(x,y) \in \Gamma(T)$ for fixed $x$, then $y$ is unique. 
\end{proof}
\begin{notation}
    We denote by $\Hilb^{\unb}$ the category of Hilbert spaces and unbounded operators, whose composition is described in 
    \Cref{Remark: composition of unbounded operators}. We denote by $\HilbRel$ the category of Hilbert spaces and relations, whose composition is 
    described in \Cref{Definition: various notions for relations}.
\end{notation}
\begin{proposition}
    The assignment $T \mapsto \Gamma(T)$ extends to an identity-on-objects functor $\Hilb^{\unb} \to \HilbRel$. 
\end{proposition}
\begin{proof}
    Compatibility with composition is the subject of \Cref{Lemma: the graph is compatible with composition}. The identity morphism on $H$ 
    is assigned to the diagonal relation on $H$, which is the identity in $\HilbRel$.
\end{proof}
\begin{lemma}\label{Lemma: the graph is compatible with the adjoint}
    Let $T \colon H \to H'$ be a densely defined unbounded operator. Then $\Gamma(T^\dagger) = \Gamma(T)^\dagger$. 
\end{lemma}
\begin{proof}
    We have 
    \begin{align*}
        \Gamma(T^\dagger) &= \{(z,T^\dagger(z)) | z \in \D(T^\dagger)\} \\
        &= \{(u,v) \in D(T^\dagger) \oplus H \ | \ \langle v,x \rangle_{H} = \langle T^\dagger u,x \rangle_{H} \ \text{for all} \ x \in D(T) \} \\ 
        &= \{(u,v) \in D(T^\dagger) \oplus H \ | \ \langle v,x \rangle_{H} = \langle u, Tx \rangle_{H'} \ \text{for all} \ x \in D(T) \} \\
        &= \{(u,v) \in H' \oplus H \ | \ \langle v,x \rangle_{H} = \langle u,Tx \rangle_{H'} \ \text{for all} \ x \in D(T) \} \\
        &= \{(u,v) \in H' \oplus H \ | \ \langle v,x \rangle_{H} = \langle u,y \rangle_{H'} \ \text{for all} \ (x,y) \in \Gamma(T)\} \\
        &= \Gamma(T)^\dagger
    \end{align*}
    where in the second equation we have used that $T$ is densely defined, in the third equation we have used 
    \Cref{Equation: adjunction equation for unbounded operators}, in the fourth equation we have used the evident inclusion together with the reverse inclusion 
    guaranteed by the definition of $D(T^\dagger)$, and in the remaining equations we only used definitions.
\end{proof}
Summarizing, we can define the adjoint of any relation which recovers the adjoint of an unbounded densely defined operator. In the following we investigate 
functorial properties of the adjoint relation operation.
\begin{lemma}\label{Lemma: relations and adjoints}
    Let $V$ be a relation from $H$ to $H'$ and $W$ be a relation from $H'$ to $H''$. Then $V^\dagger \circ W^\dagger  \subseteq (W \circ V)^{\dagger}$.
\end{lemma}
\begin{proof}
    We claim that the following equations and inclusions hold.
    \begin{align*}
            V^\dagger \circ W^\dagger 
            &= \{(u,v) \in H'' \oplus H \ | \ \exists y \in H': (u,y) \in W^\dagger \land (y,v) \in V^\dagger \} \\
            &= \{(u,v) \in H'' \oplus H \ | \ \exists y \in H': \langle y,x \rangle_{H'} = \langle u, z \rangle_{H''} \ \text{for all} \ (x,z) \in W \\
            &\hspace{1cm} \land \langle v,a \rangle_H = \langle y, b \rangle_{H'} \ \text{for all} \ (a,b) \in V\}\\
            &\subseteq \{(u,v) \in H'' \oplus H \ | \ \langle v,s \rangle_{H} = 
            \langle u,t \rangle_{H''} \ \text{for all}  \ (s,t) \in H \oplus H'' \\ 
            & \hspace{1cm} \text{such that} \ \exists o \in H': (s,o) \in V \land (o,t) \in W \} \\
            &= \{(u,v) \in H'' \oplus H \ | \ \langle v,s \rangle_{H} = 
            \langle u,t \rangle_{H''} \ \text{for all} \ (s,t) \in W \circ V \} \\
            &= (W \circ V)^{\dagger} &
    \end{align*}
    The only non-trivial statement here is the inclusion, which we will prove in the following. Let $(s,t) \in H \oplus H''$ such that there exists an $o \in H'$ 
    with $(s,o) \in V$ and $(o,t) \in W$. By assumption, there then exists a $y \in H'$ such that 
    \begin{equation}
        \langle v, s \rangle_H = \langle y, o \rangle_{H'} = \langle u,t \rangle_{H''},
    \end{equation}
    proving the claim.
\end{proof}
\begin{corollary}
    Let $T \colon H \to H'$ and $S \colon H' \to H''$ be densely defined unbounded operators so that $S \circ T \colon H \to H''$ is densely defined. Then
    \begin{equation}
        \Gamma(T^\dagger S^\dagger) \subseteq \Gamma((ST)^\dagger)
    \end{equation}
\end{corollary}
\begin{proof}
    We have 
    \begin{align*}
        \Gamma(T^\dagger S^\dagger) &= \Gamma(T^\dagger) \circ \Gamma(S^\dagger)  \tag*{\text{by \Cref{Lemma: the graph is compatible with composition}}} \\
        &= \Gamma(T)^\dagger \circ \Gamma(S)^\dagger \tag*{\text{by \Cref{Lemma: the graph is compatible with the adjoint}}} \\
        &\subseteq (\Gamma(S) \circ \Gamma(T))^\dagger \tag*{\text{by \Cref{Lemma: relations and adjoints}}} \\
        &= \Gamma(ST)^\dagger \tag*{\text{by \Cref{Lemma: the graph is compatible with composition}}}
    \end{align*}
\end{proof}
\begin{lemma}\label{Lemma: the adjoint of the diagonal is the diagonal}
    Let $H$ be a Hilbert space. Then $\Delta_H^\dagger = \Delta_H$
\end{lemma}
\begin{proof}
    We have 
    \begin{align*}
        \Delta_H^\dagger 
        &= \{(u,v) \in H \oplus H \ | \ \langle v,x \rangle_{H} = 
            \langle u,y \rangle_{H} \ \text{for all} \ (x,y) \in \Delta_H \} \\
        &= \{(u,v) \in H \oplus H \ | \ \langle v,x \rangle_{H} = 
            \langle u,x \rangle_{H} \ \text{for all} \ x \in H \} 
        = \Delta_H
    \end{align*}
    where in the last line we have used the non-degeneracy of the inner product.
\end{proof}
\begin{definition}\label{Definition: Mapping sets of relations}
    Let $H$ and $H'$ be Hilbert spaces and let $V$ and $W$ be relations from $H$ to $H'$. The \emph{mapping set} $\hom(V,W)$ is defined as follows: 
    \begin{equation}
        \hom(V,W) := 
        \begin{cases}
            \{\star\} \hspace{1cm} \text{if} \hspace{1cm} V \subseteq W \\
            \emptyset \hspace{1.25cm} \text{else}\\
        \end{cases}
    \end{equation}
\end{definition}
\begin{lemma}\label{Lemma: inclusion and adjoints for relations}
    Let $H$ and $H'$ be Hilbert spaces, and let $V$ and $W$ be relations from $H$ to $H'$. Then $V \subseteq W \Rightarrow W^\dagger \subseteq V^\dagger$.
\end{lemma}
\begin{proof}
    We have 
    \begin{align*}
        W^\dagger 
        &= \{(u,v) \in H' \oplus H \ | \ \langle v,x \rangle_{H} = \langle u,y \rangle_{H'} \ \text{for all} \ (x,y) \in W \},\\
        &\subseteq \{(u,v) \in H' \oplus H \ | \ \langle v,x \rangle_{H} = \langle u,y \rangle_{H'} \ \text{for all} \ (x,y) \in V \}
        = V^\dagger,
    \end{align*}
    proving the claim.
\end{proof}
\begin{lemma}
    Let $H,H'$, and $H''$ be Hilbert spaces, let $V \subseteq W$ be relations from $H$ to $H'$ and $V' \subseteq W'$ relations from $H'$ to $H''$. Then 
    $V' \circ V \subseteq W' \circ W$.
\end{lemma}
\begin{proof}
    Let $(x,z) \in V' \circ V$. Then there exists a $y \in H'$ such that $(x,y) \in V$ and $(y,z) \in V'$. Hence $(x,y) \in W$ and $(y,z) \in W'$ and 
    in consequence $(x,z) \in W' \circ W$.
\end{proof}
\begin{notation}
    We denote by $\HilbRel^{\text{bi}}$ the 2-category of Hilbert spaces and relations, with the respective mapping sets of 
    \Cref{Definition: Mapping sets of relations} as the sets of 2-morphisms. 
\end{notation}
\begin{proposition}
    The construction $V \to V^\dagger$ of \Cref{Definition: various notions for relations} extends to an identity on objects 
    lax 2-functor $\HilbRel^{\text{bi}} \to (\HilbRel^{\text{bi}})^{(1,2)\opp}$.
\end{proposition}
\begin{proof}
    The assignment is well-defined by construction and \Cref{Lemma: inclusion and adjoints for relations}. It respects composition 
    of 2-morphisms, since $V \subseteq W \subseteq X$ implies $X^\dagger \subseteq W^\dagger \subseteq V^\dagger$. We have 
    constructed the lax 2-cell implementing comaptibility with composition of 1-morphisms in \Cref{Lemma: relations and adjoints} and 
    we have constructed the 2-cell implementing compatibility with identity 1-morphisms in \Cref{Lemma: the adjoint of the diagonal is the diagonal}.
\end{proof}
\begin{remark}
    There are more intricate versions of mapping spaces between relations; we have chosen the one in \Cref{Definition: Mapping sets of relations} 
    as a minimalistic approach to interpret \Cref{Lemma: relations and adjoints} as a lax compatibility of the adjoint operation with composition of relations.
\end{remark}
Having established functorial properties of the adjoint operation, we now turn our attention to the closure of relations.
\begin{lemma}\label{Lemma: compatibility of closure and composition}
    Let $V$ be a relation from $H$ to $H'$ and let $W$ be a relation from $H'$ to $H''$. Then, $\overline{W \circ V} \subseteq \overline{W} \circ \overline{V}$.
\end{lemma}
\begin{proof}
    Let $(x,z) \in \overline{W \circ V}$. By definition of the closure, there exists a sequence $(x_n,z_n)_n \in W \circ V$ with $\lim_n (x_n,z_n)_n =(x,z)$. 
    By definition of the composition, there exists a sequence $(y_n)_n \in H'$ with $(x_n,y_n) \in V$ and $(y_n,z_n) \in W$ for all $n \in \mathbb{N}$. Hence, 
    we have $(x,y) = \lim_n (x_n,y_n) \in \overline{V}$ and $(y,z) = \lim_n (y_n,z_n) \in \overline{W}$, using the fact that limits are computed pointwise.
    Hence, $(x,z) \in \overline{W} \circ \overline{V}$. 
\end{proof}
\begin{lemma}\label{Lemma: the diagonal of a Hilbert space is closed}
    Let $H$ be a Hilbert space. Then $\overline{\Delta_H} = \Delta_H$.
\end{lemma}
\begin{proof}
    Recall that a topological space $X$ is Hausdorff if and only if the diagonal is closed as a subset of the product space $X \times X$. 
    Recalling that every Hilbert space is a metric space and in particular Hausdorff, this implies the claim.
\end{proof}
\begin{proposition}
    The operation $V \mapsto \overline{V}$ extends to an identity-on-objects oplax 2-functor $\HilbRel^{\text{bi}} \to \HilbRel^{\text{bi}}$. 
\end{proposition}
\begin{proof}
    To see that the assignment is well-defined, we note that an inclusion of relations $V \subseteq W$ implies that $\overline{V} \subseteq \overline{W}$. 
    The oplax functoriality is subject of \Cref{Lemma: the diagonal of a Hilbert space is closed} and \Cref{Lemma: compatibility of closure and composition}. 
    Compatibility with composition of 2-morphisms follows from the observation that, for any two Hilbert spaces $H$ and $H$', an inclusion of relations 
    $V \subseteq W \subseteq X$ implies that $\overline{V} \subseteq \overline{W} \subseteq \overline{X}$ for any three relations $V,W,X$ from $H$ to $H'$
\end{proof}
\begin{remark}
    We emphasize that the adjoint is a lax 2-functor, whereas the closure is an oplax 2-functor on $\HilbRel^{\text{bi}}$.
\end{remark}
\begin{remark}
    The composition of two closed relations is closed again: indeed, let $V$ be a closed relation from $H$ to $H'$ and let $W$ be a closed relation from 
    $H'$ to $H''$. Combining the observation $W \circ V \subseteq \overline{W \circ V}$ with the inclusion 
    \begin{equation}
        \overline{W \circ V} \subseteq \overline{W} \circ \overline{V} = W \circ V,
    \end{equation}
    where in the first step we have used \Cref{Lemma: compatibility of closure and composition} and in the second step we have used our assumption 
    that $V$ and $W$ are closed, we find that $W \circ V = \overline{W \circ V}$. Together with \Cref{Lemma: the diagonal of a Hilbert space is closed}, 
    this implies that Hilbert spaces, closed relations, and inclusions form a sub 2-category $\HilbRel^{\text{cl,bi}} \subseteq \HilbRel^{\text{bi}}$. 
\end{remark}
In the following we wish to study the involutiveness of the adjoint relation construction. Our discussion is a (slight) generalization of the discussion in 
\cite{Remline2021unbounded} to relations between different Hilbert spaces. We refer in particular to \cite[Theorem 11.7]{Remline2021unbounded} for the following result.
\begin{lemma}\label{Lemma: relations adjoints and closedness}
    Let $H$ and $H'$ be Hilbert spaces and let $V$ be a relation from $H$ to $H'$. Then we have
    \begin{itemize}
        \item[(i)] $\overline{V^\dagger} = V^\dagger$,
        \item[(ii)] $V^{\dagger \dagger} = \overline{V}$, and
        \item[(iii)] $\overline{V}^{\dagger} = V^\dagger$. 
    \end{itemize}
    Note that this implies in particular that $V^\dagger = \overline{V^{\dagger}} = V^{\dagger \dagger \dagger}$.
\end{lemma}
\begin{proof}
    (i) First, define a bounded operator $J \colon H \oplus H' \to H \oplus H', (h',h) \mapsto (h',-h)$. We claim that $V^\dagger = (JV)^{\perp}$ where 
    $JV$ denotes the image of $V$ under $J$ and $(-)^{\perp}$ denotes the orthogonal complement with respect to the inner product on $H \oplus H'$. 
    Indeed, for $(u,v) \in H' \oplus H$ we have
    \begin{align*}
        (u,v) \in V^\dagger &\Leftrightarrow \langle v,x \rangle_H = \langle u, y \rangle_{H'} \ \forall \ (x,y) \in V  \\
        &\Leftrightarrow 0 = \langle v,x \rangle_H + \langle u, - y \rangle_{H'} \ \forall \ (x,y) \in V \\
        &\Leftrightarrow 0 = \langle v,x \rangle_H + \langle u, z \rangle_{H'} \ \forall \ (x,z) \in JV \\
        &\Leftrightarrow (u,v) \in (JV)^{\perp}.
    \end{align*}
    However, the orthogonal complelement $(JV)^{\perp}$ is closed so that $V^\dagger$ is closed as well. (ii) We first note that $J(V^\perp) = (JV)^{\perp}$ 
    by (sesqui)linearity of the inner products. We then compute
    \begin{equation}
        V^{\dagger \dagger} = (JV^\dagger)^{\perp} = J((JV)^{\perp})^{\perp} = (J^2V)^{\perp \perp} = V^{\perp \perp} = \overline{V}
    \end{equation}
    where the last equation is a standard fact. (iii) We compute
    \begin{equation}
        \overline{V}^{\dagger} \overset{(ii)}{=} (V^{\dagger \dagger})^\dagger = (V^{\dagger})^{\dagger \dagger} \overset{(ii)}{=} \overline{V^{\dagger}} \overset{(i)}{=} V^\dagger.
    \end{equation}
    This finishes the proof.
\end{proof}
\begin{corollary}
    The inclusions $V \subseteq \overline{V} = V^{\dagger \dagger}$ induce a lax 2-transformation $(-)^{\dagger \dagger} \to \id_{\HilbRel^{\bi}}$ 
    with identity 1-morphism components. Its 2-morphism componets are trivial on the sub 2-category $\HilbRel^{\text{cl,bi}} \subseteq \HilbRel^{\text{bi}}$ 
    of Hilbert spaces, closed relations, and inclusions between them.
\end{corollary}
\begin{remark}
    We summarize the results of this section in the following. 
    \begin{itemize}
        \item There is a 2-category $\HilbRel^{\bi}$ of Hilbert spaces, relations (containing graphs of unbounded operators), and inclusions. 
        \item There is a lax identity-on-objects 2-functor $(-)^\dagger \colon \HilbRel^{\bi} \to (\HilbRel^{\bi})^{(1,2)\opp}$ which 
        assigns each relation $V$ to its adjoint $V^\dagger$ (generalizing the adjoint operator construction for a densely defined unbounded operator), and each inclusion
        to the induced inclusion of adjoint relations. The (normal) lax functoriality data is described in \Cref{Lemma: the diagonal of a Hilbert space is closed} 
        and \Cref{Lemma: relations and adjoints}.
        \item There is a (normal) oplax identity-on-objects 2-functor $(-)^{\dagger\dagger} \colon \HilbRel^{\bi} \to \HilbRel^{\bi}$ whose (normal) 
        oplax functoriality data for relations $V$ from $H$ to $H'$ and $W$ from $H'$ to $H''$ is given by the inclusions 
        \begin{equation}
            (W \circ V)^{\dagger \dagger} = \overline{W \circ V} \subseteq \overline{W} \circ \overline{V} = W^{\dagger \dagger} \circ V^{\dagger\dagger}
        \end{equation}
        where we have used \Cref{Lemma: compatibility of closure and composition} and \Cref{Lemma: relations adjoints and closedness} (ii).
        \item There is an identity 1-morphism component lax 2-transformation $\chi \colon (-)^{\dagger \dagger} \to \id_{\HilbRel^{\bi}}$ whose 2-morphism components are 
        given by the inclusions $V \subseteq \overline{V} = V^{\dagger \dagger}$. Moreover, we have $V^{\dagger} = V^{\dagger \dagger \dagger}$ so that 
        $\chi$ satisfies a coherence property over the triple adjoint. 
        \item The operation $(-)^{\dagger}$ restricts to a lax 2-functor $\HilbRel^{\text{cl,bi}} \to (\HilbRel^{\text{cl,bi}})^{(1,2)\opp}$ on the sub 2-category 
        of closed relations $\HilbRel^{\text{cl,bi}} \subseteq \HilbRel^{\text{bi}}$, which satisfies $(-)^{\dagger \dagger} = \id$. 
    \end{itemize}
\end{remark}
\begin{remark}
    A particularly interesting aspect of the theory we have developed here is that the lax 2-functor $(-)^\dagger$ gives rise to an oplax 2-functor 
    $(-)^{\dagger\dagger}$. A priori, the (1,2)-opposite of $(-)^\dagger$ would be an oplax 2-functor, so that it is not clear whether the composition 
    of $(-)^\dagger$ and its (1,2)-opposite would be any kind of 2-functor at all, whether pseudo, lax, or oplax. This is the primary reason why we 
    do not state our results as the existence of a kind of lax higher dagger structure on $\HilbRel^{\bi}$, despite the obvious resemblence. 
    The same objection still holds assuming closed relations only.
\end{remark}

\subsection{Bornological algebras}\label{Subsection: Bornological algebras}
In this subsection, we focus on the closed symmetric monoidal category with (co)equalizers $\cBorn$ of complete (convex) bornological vector spaces. 
According to \Cref{Theorem: closedness of the Morita 2-category}, $\Mor(\cBorn)$ is a symmetric monoidal 2-category which is 1-closed and has duals 
for objects. In the following, we wish to study the procedures discussed in \Cref{Subsection: Morita 2-categories} in the present context.\\

Let $A$ be an object in $\Mor(\cBorn)$, that is, a complete bornological algebra. Recall that the identity 1-morphism $A \to A$ in $\Mor(\cBorn)$ is given 
by the Bornological $A$-$A$-bimodule $A$. We compute 
\begin{equation}
    A^A_A = \equa (A^A \rightrightarrows A^{A \otimes A}) = \hom_A(A,A) \cong A
\end{equation}
where in the last step we have used that an intertwiner of $A$-modules $A \to A$ is completely determined by its value on $1 \in A$. Next, let us consider 
three objects $A,B,C$ and two 1-morphisms $X \colon A \to B$ and $Y \colon B \to C$ in $\Mor(\cBorn)$. We compute $B^X_B = \hom_B(X,B)$ and $C^Y_C = \hom_C(Y,C)$. 
On the other hand, we have $C^{Y \otimes_B X}_C = \hom_C(Y \otimes_B X,C)$. By closedness of $\Mor(\cBorn)$ we have an intertwiner of $A$-$C$-bimodules
\begin{align*}
    \hom_B(X,B) \otimes_B \hom_C(Y,C) &\to \hom_C(Y \otimes_B X,C)\\
    \phi \otimes \psi &\mapsto \left( 
    \xymatrix{
        Y \otimes_B X \ar[r]^-{\id \otimes_B \phi} & Y \otimes_B B \cong Y \ar[r]^-{\psi} & C
    } \right)
\end{align*}
which however does not need to be an isomorphism in general. This finishes our discussion of the structure 2-morphisms of our lax 
2-functor $\mathcal{J}$ on $\Mor(\cBorn)$: the ones implementing unitality are isomorphisms while the ones expressing compatibility with composition are not (unless 
one imposes strong finiteness conditions on the modules at hand). \\

Recalling that $\Mor(\cBorn)$ is symmetric monoidal and has duals for objects as well as a symmetric monoidal involution given by complex conjugation, it carries 
in particular an $\Oone$-volutive structure, whose underlying functor assigns each bornological algebra $A$ to its complex conjugate opposite $\overline{A}^{\opp}$, each 
bornological $A$-$B$-bimodule to its associated bornological $\overline{B}^{\opp}$-$\overline{A}^{\opp}$-bimodule, 
and similar for 2-morphisms. Passing to the $\Oone$-dagger 2-category $S_{\Oone}\Mor(\cBorn)$ amounts to considering bornological $*$-algebras\footnote{Strictly 
speaking, the $\star$-structures are implemented via invertible bimodules rather than algebra isomorphisms. The latter are however a special case of the former.}, 
as well as the same 1- and 2-morphisms as before. In the following we will denote the dagger structure by $(d,\eta)$. Let $M \colon A \to B$ be a 1-morphism in 
$S_{\Oone}\Mor(\cBorn)$.
Then, we have the $A$-$B$-bimodule $J(M) = B^M_B$ and moreover the $B$-$A$-bimodule $dJ(M) = \overline{B^M_B} = \overline{\hom_B(M,B)}$ where the bimodule actions of the latter are 
defined using the $*$-structures of $A$ and $B$. \\
We are now in position to discuss (lax) hermitian fixed points in this context. In order to make the discussion more digestible, one may employ the 
formulation of (lax) hermitian fixed points in terms of pairings. Let $M \colon A \to B$ be a 1-morphism together with a lax hermitian fixed point structure 
$\theta_M \colon M \to \overline{\hom_B(M,B)}$. We may interpret this data as a potentially degenerate $B$-valued hermitian pairing on $M$. Assume now we are given another 
1-morphism $N \colon B \to C$ together with a lax hermitian fixed point structure. Then their composition as defined in 
\Cref{Passing to lax hermitian fixed points} is $N \otimes_B M$ equipped with the pairing given by
\begin{equation}\label{Equation: pretensorproduct of hilbert modules}
    \langle m \otimes n, m' \otimes n' \rangle = \langle n, n'\langle m,m' \rangle \rangle
\end{equation}
where we have used the pairings corresponding to $\theta_M$ and $\theta_N$ on the right hand side, as well 
as the right $B$-action on $N$.\\

We have emphasized in our general discussion that the pairing of \Cref{Equation: pretensorproduct of hilbert modules} does not need to be non-degenerate 
at this point, that is, the corresponding lax hermitian fixed point structure $\theta_{N \otimes_B M}$ does not need 
to be an isomorphism, even if $\theta_M$ and $\theta_N$ are. In the present case, however, there is a procedure one may perform to dispose of potential 
degeneracies. Namely, one may quotient out the isotropic vectors, that is, those satisfying $\langle v,v \rangle = 0$. After suitable completion, 
one obtains an honest hermitian fixed point. With this new composition, we obtain a sub 2-category 
whose 1-morphisms are honest hermitian fixed points, whose Hom-categories carry dagger structures in consequence, and which 
in total carries a fully dagger structure. We stress again that these are facts specific to the present context, and do not generalize well 
to an arbitrary fully closed symmetric monoidal 2-category.\\

In the remainder of this section, we wish to draw comparisons to von Neumann algebras and Hilbert/hermitian bimodules. Recall first that any von Neumann 
algebra is in particular a Banach $\star$-algebra, hence also a complete bornological $\star$-algebra, that is, a special object in the $\Oone$-dagger 2-category 
$S_{\Oone}\Mor(\cBorn)$. Here we recall that the functor $\Ban \to \cBorn$ is symmetric monoidal, hence induces a functor $\Alg(\Ban) \to \Alg(\cBorn)$ where 
$\Alg(\C)$ refers to the category of algebras and algebra homomorphisms in the monoidal category $\C$. \\
Recall further the classical result \cite[Theorem 3.2]{Paschke1973} that for any Hilbert module over a von Neumann algebra, the pairing induces an isomorphism 
rendering it an (honest) hermitian fixed point in our sense; this is a generalization of the classical Riesz representation theorem, to which it reduces upon 
considering the trivial algebra $\mathbb{C}$. Any Hilbert (bi)module over a von Neumann algebra in the sense of \cite{Paschke1973,Paschke1974} therefore defines 
an (honest) hermitian fixed point in our sense, where we again used the functor $\Ban \to \cBorn$. The modified composition law we have explained above is precisely 
the composition for such Hilbert bimodules (or rather their bornological analogues).

\addcontentsline{toc}{section}{References}
\printbibliography

\medskip 

\noindent
Universität Wien, Fakultät für Physik, Boltzmanngasse 5, 1090 Wien, Österreich\\
\href{mailto:tim.lueders@univie.ac.at}{tim.lueders@univie.ac.at}

\end{document}